\documentclass[a4paper,12pt,reqno]{amsart}
\usepackage{amsfonts}
\usepackage{amsmath}
\usepackage{amssymb}
\usepackage[a4paper]{geometry}
\usepackage{mathrsfs}
\usepackage{csquotes}

\usepackage[colorlinks]{hyperref}
\renewcommand\eqref[1]{(\ref{#1})} 
\numberwithin{equation}{section}
\theoremstyle{plain}
\newtheorem{thm}{Theorem}[section]
\newtheorem{prop}[thm]{Proposition}
\newtheorem{cor}[thm]{Corollary}
\newtheorem{lem}[thm]{Lemma}
 \newtheorem{exa}[thm]{Example}
\theoremstyle{definition}

\newtheorem{rem}[thm]{Remark}

\renewcommand{\wp}{\mathfrak S}
\newcommand{\Rn}{\mathbb R^{n}}
\newcommand{\G}{\mathbb G}

\def\R{\mathcal R}

\def\R{\mathcal R}

\def\e[#1]{{\textrm{e}}^{#1}}

\def\G{{\mathbb G}}

\def\L{\mathcal{L}}

\begin{document}

   \title[Hypoelliptic functional inequalities]
 {Hypoelliptic functional inequalities}

\author[M. Ruzhansky]{Michael Ruzhansky}
\address{
  Michael Ruzhansky:
  \endgraf
  Department of Mathematics
  \endgraf
  Imperial College London
  \endgraf
  180 Queen's Gate, London SW7 2AZ
  \endgraf
  United Kingdom
  \endgraf
  {\it E-mail address} {\rm m.ruzhansky@imperial.ac.uk}
  }

\author[N. Yessirkegenov]{Nurgissa Yessirkegenov}
\address{
  Nurgissa Yessirkegenov:
  \endgraf
  Institute of Mathematics and Mathematical Modelling
  \endgraf
  125 Pushkin str.
  \endgraf
  050010 Almaty
  \endgraf
  Kazakhstan
  \endgraf
  and
  \endgraf
  Department of Mathematics
  \endgraf
  Imperial College London
  \endgraf
  180 Queen's Gate, London SW7 2AZ
  \endgraf
  United Kingdom
  \endgraf
  {\it E-mail address} {\rm n.yessirkegenov15@imperial.ac.uk}
  }

 \dedicatory{Dedicated to the 70$^{th}$ birthday of Fulvio Ricci}

\thanks{The authors were supported in parts by the EPSRC
 grant EP/R003025/1 and by the Leverhulme Grant RPG-2017-151, as well as by the MESRK grant AP05133271. No new data was collected
 or
generated during the course of research.}


     \begin{abstract} In this paper we derive a variety of functional inequalities for general homogeneous invariant hypoelliptic differential operators on nilpotent Lie groups. The obtained inequalities include Hardy, Rellich, Hardy-Littllewood-Sobolev, Gag\-li\-ardo-Nirenberg, Caffarelli-Kohn-Nirenberg and Trudinger-Moser inequalities. So\-me of these estimates have been known in the case of the sub-Laplacians, however, for more general hypoelliptic operators almost all of them appear to be new as no approaches for obtaining such estimates have been available. Moreover, we obtain several versions of local and global weighted Trudinger-Moser inequalities with remainder terms, critical Hardy and weighted Gagliardo-Nirenberg inequalities, which appear to be new also in the case of the sub-Laplacian. Curiously, we also show the equivalence of many of these critical inequalities as well as asymptotic relations between their best constants. The approach developed in this paper relies on establishing integral versions of Hardy inequalities on homogeneous groups, for which we also find necessary and sufficient conditions for the weights for such inequalities to be true. Consequently, we link such integral Hardy inequalities to different hypoelliptic inequalities by using the Riesz and Bessel kernels associated to the described hypoelliptic operators.
     \end{abstract}
     \maketitle

       \tableofcontents

\section{Introduction}
\label{SEC:intro}

In this paper we are interested in developing approaches that allow one to derive a variety of functional inequalities for general homogeneous invariant hypoelliptic differential operators on nilpotent Lie groups. Inequalities of such type are important by themselves but also play an important role in wider analysis, in particular in view of the seminal results of Rothschild and Stein \cite{RS76} linking the analysis of hypoelliptic differential operators on nilpotent Lie groups to differential operators on manifolds. 

To give an idea of the obtained results and to put them in perspective we start by describing a collection of some of the obtained inequalities in the setting of sub-Laplacians on stratified Lie groups (homogeneous Carnot groups).

\subsection{Hardy-Sobolev inequalities on stratified groups}

Hardy inequalities on stratified groups are extremely well investigated topic, with different versions of such inequalities known, also with best constants. While we can not possibly give a comprehensive bibliography for it here, we can refer to \cite{RS17_JDE} and to \cite{RS17_AMg} for the literature reviews of the subject for the horizontal norm and for norms given in terms of the fundamental solutions of the sub-Laplacian, respectively.

However, the starting point for the investigation of this paper is the following version of the Hardy inequality recently obtained by Ciatti, Cowling and Ricci \cite{CCR15}. Let $\mathbb{G}$ be a stratified group of homogeneous dimension $Q$ and let $\L$ be a sub-Laplacian on $\G$. Let $|\cdot|$ be homogeneous norm on $\mathbb{G}$. We refer to Section \ref{SEC:prelim} for more details of this classical setting.

Let $1<p<\infty$ and let
$T_{\gamma}f:=|\cdot|^{-\gamma}\L^{-\gamma/2}f$ with
$0<\gamma<Q/p$.
Then, as it was shown in \cite[Theorem A]{CCR15}, the Hardy inequality for the fractional order operator $\L^{\gamma/2}$ can take the following form:
the operator $T_{\gamma}$ extends uniquely to a bounded operator on $L^{p}(\G)$, and we have
\begin{equation}\label{Ricci1}
\|T_{\gamma}\|_{L^{p}(\G)\rightarrow L^{p}(\G)}\lesssim1+C\gamma+O(\gamma^{2}).
\end{equation}
We also refer to \cite{CCR15} for the history of \eqref{Ricci1}.

Among other things, in this paper we extend the boundedness in \eqref{Ricci1} to the setting of general homogeneous invariant hypoelliptic differential operators taking place of the operator $\L$. Moreover, we extend such estimates to the range of $L^{p}-L^{q}$ estimates as well as give their critical versions in the case of $\gamma=Q/p$.

Let us list some of such results still in the simplified setting of the sub-Laplacians.
First we observe that by combining \eqref{Ricci1} with Sobolev inequalities for the sub-Laplacian, we have the following extended version of \eqref{Ricci1}:
\begin{itemize}
\item ({\bf Hardy-Sobolev inequalities on stratified groups}) Let $1<p\leq q<\infty$ and $0<a<Q/p$. Let $0\leq b<Q$ and $\frac{a}{Q}=\frac{1}{p}-\frac{1}{q}+\frac{b}{qQ}$. Then there exists a positive constant $C$ such that
\begin{equation}\label{Hardy_new1_intro1}
\left\|\frac{f}{|x|^{\frac{b}{q}}}\right\|_{L^{q}(\G)}\leq
C\|(-\L)^{\frac{a}{2}}f\|_{L^{p}(\G)}
\end{equation}
holds for all $f\in \dot{L}^{p}_{a}(\G)$.
\end{itemize}
Here the space $ \dot{L}^{p}_{a}(\G)$ is the homogeneous Sobolev space over $L^p$ of order $a$, based on the sub-Laplacian $\L$. The theory of such spaces has been extensively developed by Folland \cite{F75}.
Consequently, more general results of this paper yield the following two new versions of the critical case of \eqref{Hardy_new1_intro1} for $a=Q/p$: the global logarithmic version and the local version without logarithmic  term:
\begin{itemize}
\item ({\bf Critical global Hardy inequality for $a=Q/p$ on stratified groups}) Let $1<p<r<\infty$ and $p<q<(r-1)p'$, where $1/p+1/p'=1$. Then there exists a positive constant $C_{6}=C_{6}(p, q, r, Q)$ such that
\begin{equation}\label{Hardy_crit2_intro_str}
\left\|\frac{f}{\left(\log\left(e+\frac{1}{|x|}\right)\right)^{\frac{r}{q}}|x|^{\frac{Q}{q}}}\right\|_{L^{q}(\G)}\leq
C_{6}(\|f\|_{L^{p}(\G)}+\|(-\L)^{\frac{Q}{2p}} f\|_{L^{p}(\G)})
\end{equation}
holds for all $f\in L^{p}_{Q/p}(\G)$.

\item ({\bf Critical local Hardy inequality for $a=Q/p$ on stratified groups}) Let $1<p<\infty$ and $\beta\in[0,Q)$. Let $r>0$ be given and let
    $x_{0}$ be any point of $\G$. Then for any $p\leq q<\infty$ there exists a positive constant $C_{4}=C_{4}(p, Q, \beta, r, q)$ such that
\begin{equation}\label{Hardy_crit1_intro_str}
\left\|\frac{f}{|x|^{\frac{\beta}{q}}}\right\|_{L^{q}(B(x_{0},r))}\leq C_{4}q^{1-1/p}(\|f\|_{L^{p}(B(x_{0},r))}+\|(-\L)^{\frac{Q}{2p}} f\|_{L^{p}(B(x_{0},r))})
\end{equation}
holds for all $f\in L^{p}_{Q/p}(B(x_{0},r))$, and such that $\underset{q\rightarrow\infty}{\rm limsup\;}C_{4}(p, Q, \beta, r, q)<\infty$. The asymptotically sharp constant for \eqref{Hardy_crit1_intro_str} in the sense of Remark \ref{rem_F} is given in Theorem \ref{Hardy_Rock_thm}.
\end{itemize}
The space $L_{Q/p}^{p}(B(x_{0},r))$ appearing in \eqref{Hardy_crit1_intro_str} is defined as follows. Let $B(x_{0},r)$ be the quasi-ball of radius $r$ in a stratified group $\G$ with respect to $|\cdot|$, centred at $x_{0}$. Then $L_{Q/p}^{p}(B(x_{0},r))$ with $1<p<\infty$ denotes the completion of $C_{0}^{\infty}(B(x_{0},r))$ with respect to the norm
\begin{equation}\label{def_space_str}
\|f\|_{L_{Q/p}^{p}(B(x_{0},r))}:=\left(\int_{B(x_{0},r)}(|(-\L)^{\frac{Q}{2p}}f(x)|^{p}+|f(x)|^{p})dx\right)^{1/p}.
\end{equation}
Here we understand the operator $(-\L)^{\frac{Q}{2p}}$ as defined by the functional calculus on the whole group $\G$.

In the case $p=Q$, we have the following homogeneous improvement of \eqref{Hardy_crit1_intro_str}, where $\nabla_{H}$ is the horizontal gradient in $\G$:

\begin{itemize}
\item ({\bf Hardy inequalities for $p=Q$ on stratified groups})
Let $Q\geq 3$ and $\beta\in[0,Q)$. Let $r>0$ be given and let
$x_{0}$ be any point of $\G$. Then for any $Q\leq q<\infty$ there exists a positive constant $C_{5}=C_{5}(Q, \beta, r, q)$ such that
\begin{equation}\label{Hardy_Rock1_pQ_intro}
\left\|\frac{f}{|\cdot|^{\frac{\beta}{q}}}\right\|_{L^{q}(B(x_{0},r))}\leq C_{5}\|\nabla_{H}f\|_{L^{Q}(B(x_{0},r))}
\end{equation}
holds for all $f\in L^{Q}_{1}(B(x_{0},r))$, and such that $\underset{q\rightarrow\infty}{\rm limsup\;}C_{5}(Q, \beta, r, q)<\infty$. The asymptotically sharp constant for \eqref{Hardy_Rock1_pQ_intro} in the sense of Remark \ref{rem_B7} is given in Theorem \ref{Hardy_Rock_thm_pQ}.
\end{itemize}

Thus, \eqref{Hardy_crit2_intro_str}, \eqref{Hardy_crit1_intro_str} and \eqref{Hardy_Rock1_pQ_intro} give the critical cases of the Hardy type inequalities in \cite[Theorem A]{CCR15}.

Actually, in Section \ref{SEC:Hardy} we obtain all of the above inequalities  for more general hypoelliptic operators on more general nilpotent groups, namely, on graded groups. As far as we are aware there are no other Hardy type inequalities known on graded groups in the literature.

\subsection{Trudinger-Moser inequality on stratified groups}

If $\beta=0$, the inequality \eqref{Hardy_Rock1_pQ_intro} can be also understood as a local Sobolev embedding. In turn, the weighted case links to the following weighted Trudinger-Moser inequality established in this paper, which appears to be new also in the setting of stratified groups:
\begin{itemize}
\item ({\bf Trudinger-Moser inequality on stratified groups}) Let $Q\geq 3$ and let $|\cdot|$ be a homogeneous quasi-norm on $\G$. Let $\alpha_{Q}$ be the constant rather explicitly given in Theorem \ref{Tyson_thm2}. Then we have
\begin{equation}\label{weightedTrud1_Qp2_intro}
\sup_{\|f\|_{L^{Q}_{1}(\G)}\leq1}\int_{\G}\frac{1}{|x|^{\beta}}\left(\exp(\alpha|f(x)|^{Q'})
-\sum_{k=0}^{Q-2}\frac{\alpha^{k}|f(x)|^{kQ'}}{k!}\right)dx
< \infty
\end{equation}
for any $\beta\in[0,Q)$ and $\alpha\in (0,\alpha_{Q}(1-\beta/Q))$, where $Q'=Q/(Q-1)$. When $\alpha>\alpha_{Q}(1-\beta/Q)$, the integral in \eqref{weightedTrud1_Qp2_intro} is still finite for any $f\in L^{Q}_{1}(\G)$, but the supremum is infinite.
\end{itemize}

Let us put this inequality in perspective. It has been known since the work of Saloff-Coste \cite{S88}, that there exist positive constants $\alpha_{Q}$ and $C$ such that the following version of the Trudinger-Moser inequality holds for any bounded smooth domain $\Omega$ in a startiifed group $\G$, and $f\in L^{Q}_{1}(\Omega)$ on $\G$:
\begin{equation}\label{intro:trud1}
\int_{\Omega}\exp\left(\alpha_{Q}\left(\frac{|f(x)|}{\|\nabla_{H} f\|_{L^{Q}(\Omega)}}\right)^{\frac{Q}{Q-1}}\right)dx\leq C|\Omega|,
\end{equation}
where $L^{Q}_{1}(\Omega)$ is the completion of $C_{0}^{\infty}(\Omega)$ with respect to the norm
$\|f\|_{L^{Q}(\Omega)}+\|\nabla_{H} f\|_{L^{Q}(\Omega)}$. In the Euclidean space $\Rn$, the inequality \eqref{intro:trud1} was proved independently by Poho\v{z}aev \cite{P65}, Yudovi\v{c} \cite{Y61} and Trudinger \cite{T67}. The optimal constant $\alpha_{n}$ was found by Moser \cite{M79}.

In the setting of stratified groups, the sharp exponent $\alpha_{Q}$ for the Heisenberg groups was obtained in \cite{CL01} and for the general stratified groups in \cite{BMT03}. However, when $\Omega$ has infinite volume, such results and inequality \eqref{intro:trud1} are not meaningful.

According to our knowledge, most of the existing proofs of the Trudinger-Moser inequalities on unbounded domain of the Euclidean space are based on rearrangement arguments, which are not available on stratified (Carnot) Lie groups. However, the Trudinger-Moser inequalities on the entire Heisenberg group without using rearrangement argument were obtained by Yang \cite{Yang14}, namely, the author showed the following inequality by gluing local estimates with the help of cut-off functions:
Let $Q=2n+2$ be the homogeneous dimension of $\mathbb{H}^{n}$ and let $Q'=Q/(Q-1)$. Let $\tau>0$ and let $\omega_{2n-1}$ be the surface area of the unit sphere in $\mathbb{R}^{2n}$. Let $\sigma_{Q}=\Gamma(1/2)\Gamma(n+1/2)\omega_{2n-1}/n!$ and $\alpha_{Q}=Q\sigma_{Q}^{1/(Q-1)}$. Then for any $\beta\in [0,Q)$ and $\alpha\in(0,\alpha_{Q}(1-\beta/Q))$ we have
\begin{equation}\label{intro:trud3}
\sup_{\|f\|_{1,\tau}\leq1}\int_{\mathbb{H}^{n}}\frac{1}{(\rho(\xi))^{\beta}}\left(\exp(\alpha|f(\xi)|^{Q'})-\sum_{k=0}^{Q-2}
\frac{\alpha^{k}|f(\xi)|^{kQ'}}{k!}\right)d\xi<\infty,
\end{equation}
where
$$\|f\|_{1,\tau}=\left(\int_{\mathbb{H}^{n}}(|\nabla_{H}f(\xi)|^{Q}+\tau|f(\xi)|^{Q})d\xi\right)^{1/Q}$$
and $\rho(\xi)=(|z|^{4}+t^{2})^{1/4}$ with $z=(x,y)\in \mathbb{R}^{2n}$ and $\xi=(z,t)\in \mathbb{H}^{n}$. Moreover, it is shown that when $\alpha>\alpha_{Q}(1-\beta/Q)$ the integral in \eqref{intro:trud3} is still finite for any $f\in L^{Q}_{1}(\mathbb{H}^{n})$, but the supremum is infinite.
In this respect inequality \eqref{intro:trud1} extends \eqref{intro:trud3} to the setting of general stratified groups.
We also refer to \cite{CL02}, \cite{LLT12}, \cite{LL12}, \cite{CLLY12}, \cite{LLT14} and \cite{LT13} for different versions of the Trudinger-Moser inequalities in the special cases of Heisenberg groups. 

In this paper, however, our methods allow us to obtain the weighted Moser-Trudinger inequality \eqref{intro:trud3} with a
sharp constant $\alpha_{Q}$ for general stratified groups (see Theorem \ref{locweightedTrud_thm_Qp2}), where $\alpha_{Q}$ is given in terms
of an integral on a \textquote{unit sphere} of the horizontal gradient of a certain
homogeneous norm.
We note that to obtain the weighted Trudinger-Moser inequality on fractional order Sobolev space even on the Heisenberg groups, the simplest non-trivial stratified groups, is a much more delicate matter. However, in this paper we show the Trudinger-Moser inequality on fractional order Sobolev space in a much more general setting, namely that of graded groups, which includes the cases of $\mathbb R^n$, Heisenberg, and general stratified Lie groups, building on the strategy developed in \cite{Oz97} and \cite{INW14}.

\subsection{Hardy-Sobolev inequalities on graded groups}

The setting of graded groups as developed by Folland and Stein \cite{FS-book} allows one to work efficiently with higher order hypoelliptic operators, contrary to only sub-Laplacians appearing on stratified groups.

We assume now that $\mathbb{G}$ is a nilpotent Lie group with a compatible dilation structure, i.e. a homogeneous group. We refer to
Section \ref{SEC:prelim} for a precise (well-known) definition.
Let $Q$ be the  homogeneous dimension of $\mathbb{G}$ and let  $|\cdot|$ be a homogeneous
quasi-norm on $\mathbb{G}$. Let $\mathcal{R}$ be a positive left-invariant homogeneous
hypoelliptic invariant differential operator on $\mathbb{G}$ of homogeneous degree $\nu$. Such operators are called {\em Rockland operators}.

In particular, the existence of such an operator is equivalent to the condition that the group is graded, and such operators can be characterised in terms of the representation theory of the group by the celebrated result of Helffer and Nourrigat \cite{HN-79}. We note that examples of graded groups include $\Rn$, the Heisenberg group, and general stratified groups.  Again, for brevity, we refer to Section \ref{SEC:prelim} for precise definitions.

Therefore, results for Rockland operators on graded groups can be viewed as the most general {\em differential} results in the setting on nilpotent Lie groups. As far as we know, none of the inequalities we now describe are known in such settings.

From now on we let $\mathcal{R}$ be a positive Rockland operator, that is, a positive left-invariant homogeneous
hypoelliptic invariant differential operator on $\mathbb{G}$ of homogeneous degree $\nu$.
Its powers $\R^{a}$ are understood through the functional calculus on the whole of $\G$, extensively analysed in \cite{FR16,FR:Sobolev}.

We start with the following analogue of
\eqref{Hardy_new1_intro1}, which we also call the Hardy-Sobolev inequalities since it contains the classical Hardy, Rellich and Sobolev inequalities:

\begin{itemize}
\item ({\bf Hardy-Sobolev inequalities on graded groups}) Let $1<p\leq q<\infty$ and $0<a<Q/p$. Let $0\leq b<Q$ and $\frac{a}{Q}=\frac{1}{p}-\frac{1}{q}+\frac{b}{qQ}$. Then there exists a positive constant $C$ such that
\begin{equation}\label{Hardy_new1_intro_11}
\left\|\frac{f}{|x|^{\frac{b}{q}}}\right\|_{L^{q}(\G)}\leq
C\|\R^{\frac{a}{\nu}}f\|_{L^{p}(\G)}
\end{equation}
holds for all $f\in \dot{L}^{p}_{a}(\G)$.
\end{itemize}
In particular, for $p=q$ we obtain the general hypoelliptic family of the Hardy inequalities:
\begin{equation}\label{Hardy_new1_intro-gr}
\left\|\frac{f}{|x|^{a}}\right\|_{L^{p}(\G)}\leq
C\|\R^{\frac{a}{\nu}}f\|_{L^{p}(\G)}, \;\; 1<p<\infty,\; 0<a<Q/p.
\end{equation}
In particular, for $a=1$ and $a=2$ we obtain the hypoelliptic versions of Hardy and Rellich inequalities, respectively, which in this form are new already on the stratified groups since the operator $\R$ does not have to be a sub-Laplacian and can be of any order.
At the same time, for $b=0$, \eqref{Hardy_new1_intro_11} gives a simple proof of the Sobolev inequality obtained in \cite{FR:Sobolev}:
\begin{equation}\label{Hardy_new1_intro-sob}
\left\|f\right\|_{L^{q}(\G)}\leq
C\|\R^{\frac{a}{\nu}}f\|_{L^{p}(\G)},\;\; 1<p<q<\infty,\;  a=Q\left(\frac{1}{p}-\frac{1}{q}\right).
\end{equation}

The homogeneous and inhomogeneous Sobolev spaces $\dot{L}^{p}_{a}(\mathbb{G})$ and ${L}^{p}_{a}(\mathbb{G})$ based on the
positive left-invariant hypoelliptic differential Rockland operator $\R$ have been extensively investigated in \cite{FR:Sobolev}
and \cite[Section 4.4]{FR16} to which we refer for the details of their properties. In these works, the authors generalised to
graded groups the Sobolev spaces based on the sub-Laplacian on stratified groups analysed by Folland in \cite{F75}.

We also note that in Remark \ref{rem_crit_embed} we discuss a byproduct of the obtained family of Gagliardo-Nirenberg inequalities yielding the critical case $a=Q/p$ of the embedding $L^{p}_{a}(\G)\hookrightarrow L^{q}(\G)$ with $1/q=1/p-a/Q$ and $0<a<Q/p$ shown in \cite[Proposition 4.4.13]{FR16}, namely, the continuous embedding
$L^{p}_{Q/p}(\G)\hookrightarrow L^{q}(\G)$ for $1<p<\infty$, $p\leq q<\infty$. The critical case of these embeddings with $q=\infty$ are the Trudinger-Moser inequalities, also obtained in several forms in this paper.

As a consequence of \eqref{Hardy_new1_intro_11}, we also get the following
\begin{itemize}
\item ({\bf Uncertainty type principle on graded groups}). Let $1<p\leq q<\infty$ and $0<a<Q/p$. Let $0\leq b<Q$ and $\frac{a}{Q}=\frac{1}{p}-\frac{1}{q}+\frac{b}{qQ}$. Then there exists a positive constant $C$ such that
\begin{equation}\label{uncer_grad1_intro}
\|\R^{\frac{a}{\nu}}f\|_{L^{p}(\G)}\||x|^{\frac{b}{q}}f\|_{L^{q'}(\G)}\geq
C\int_{\G}|f(x)|^{2}dx
\end{equation}
holds for all $f\in \dot{L}^{p}_{a}(\G)$, where $1/q+1/q'=1$.
\end{itemize}
As in the stratified case, we have the following critical cases of Hardy-Sobolev inequalities:
\begin{itemize}
\item{\bf (Critical global Hardy inequality for $a=Q/p$ on graded groups)}. Let $1<p<r<\infty$ and $p<q<(r-1)p'$, where $1/p+1/p'=1$. Then there exists a positive constant $C_{6}=C_{6}(p, q, r, Q)$ such that
\begin{equation}\label{Hardy_wholeG2_log_intro1}
\left\|\frac{f}{\left(\log\left(e+\frac{1}{|x|}\right)\right)^{\frac{r}{q}}|x|^{\frac{Q}{q}}}\right\|_{L^{q}(\G)}\leq
C_{6}\|f\|_{L^{p}_{Q/p}(\G)}
\end{equation}
holds for all $f\in L^{p}_{Q/p}(\G)$.

\item{\bf (Critical local Hardy inequality for $a=Q/p$ on graded groups)}.
Let $1<p<\infty$ and $\beta\in[0,Q)$. Let $r>0$ be given and let $x_{0}$ be any point of $\G$. Then for any $p\leq q<\infty$ there exists a positive constant $C_{4}=C_{4}(p, Q, \beta, r, q)$ such that
\begin{equation}\label{Hardy_Rock1_intro}
\left\|\frac{f}{|x|^{\frac{\beta}{q}}}\right\|_{L^{q}(B(x_{0},r))}\leq C_{4}q^{1-1/p}\|f\|_{L^{p}_{Q/p}(B(x_{0},r))}
\end{equation}
holds for all $f\in L^{p}_{Q/p}(B(x_{0},r))$, and such that $\underset{q\rightarrow\infty}{\rm limsup\;}C_{4}(p, Q, \beta, r, q)<\infty$, where the space $L^{p}_{Q/p}(B(x_{0},r))$ is defined as the completion of $C_{0}^{\infty}(B(x_{0},r))$ with respect to the norm
\begin{equation}\label{def_space_str_gr}
\|f\|_{L_{Q/p}^{p}(B(x_{0},r))}:=\left(\int_{B(x_{0},r)}(|\R^{\frac{Q}{\nu p}}f(x)|^{p}+|f(x)|^{p})dx\right)^{1/p},
\end{equation}
and $B(x_{0},r)$ is the quasi-ball of radius $r$ in $\G$ centred at $x_{0}$. Here and elsewhere in similar expressions, we understand $\R^{\frac{Q}{\nu p}}$ through the functional calculus on the whole of $\G$, which we then integrate over $B(x_{0},r)$.

The asymptotically sharp constant for \eqref{Hardy_Rock1_intro} in the sense of Remark \ref{rem_F} is given in Theorem \ref{Hardy_Rock_thm}.
\end{itemize}

Moreover, the critical Hardy type inequalities \eqref{Hardy_Rock1_intro} are actually equivalent to the local weighted Trudinger-Moser inequalities with remainder terms \eqref{weightedTrud1_intro} that we will now describe.

\subsection{Trudinger-Moser inequalities on graded groups}

In this paper we establish the following local and global versions of Trudinger-Moser inequalities with remainder terms on general graded groups. It is interesting to note that the local Trudinger-Moser inequalities \eqref{weightedTrud1_intro} are equivalent to the critical Hardy type inequalities \eqref{Hardy_Rock1_intro}, with an asymptotic relation between best constants.

\begin{itemize}
\item{\bf (Local weighted Trudinger-Moser inequalities with remainder terms)}. Let $1<p<\infty$ and $\beta\in[0,Q)$. Let $r>0$ be given and let
    $x_{0}$ be any point of $\G$. Then we have
$$\int_{B(x_{0},r)}\frac{1}{|x|^{\beta}}\left(\exp(\alpha|f(x)|^{p'})-\sum_{0\leq
k<p-1,\;k\in\mathbb{N}}\frac{\alpha^{k}|f(x)|^{kp'}}{k!}\right)dx$$
\begin{equation}\label{weightedTrud1_intro}
\leq C_{1}\|f\|^{p}_{L^{p}_{Q/p}(B(x_{0},r))},
\end{equation}
for any $\alpha\in[0,C_{2})$ and for all $f\in L^{p}_{Q/p}(B(x_{0},r))$ with $\|f\|_{L^{p}_{Q/p}(B(x_{0},r))}\leq1$, where the space $L^{p}_{Q/p}(B(x_{0},r))$ is defined in \eqref{def_space_str_gr}, and $B(x_{0},r)$ is the quasi-ball of radius $r$ in $\G$ centred at $x_{0}$. Here $1/p+1/p'=1$, and $C_{1}=C_{1}(p,Q,\alpha,\beta, r)$ and $C_{2}=C_{2}(p,Q,\beta)$ are given in Theorem
\ref{locweightedTrud_thm}.
\item {\bf (Global weighted Trudinger-Moser inequalities with remainder terms)}.
Let $1<p<\infty$ and $\beta\in[0,Q)$ with $\mu>Q/(Q-\beta)$. Then we have
$$\int_{\G}\frac{1}{|x|^{\beta}}\left(\exp(\alpha|f(x)|^{p'})-\sum_{0\leq
k<p-1,\;k\in\mathbb{N}}\frac{\alpha^{k}|f(x)|^{kp'}}{k!}\right)dx$$
\begin{equation}\label{Tru-Mos1_intro}
\leq C_{3}(\|f\|^{p}_{L^{p}(\G)}+\|f\|^{p/\mu}_{L^{p}(\G)}),
\end{equation}
for all $\alpha\in(0,C_{2})$, and for all functions $f\in L_{Q/p}^{p}(\G)$ with $\|\R^{\frac{Q}{\nu p}}f\|_{L^{p}(\G)}\leq1$,
where $1/p+1/p'=1$, and $C_{2}=C_{2}(p, Q, \beta, \mu)$ and $C_{3}=C_{3}(p,Q,\alpha, \beta, \mu)$ are given in Theorem
\ref{Tru-Mos_thm}.
\end{itemize}

When $p=Q$ and $\|f\|_{L^{p}_{Q/p}(\G)}\leq1$, the inequality \eqref{Tru-Mos1_intro} extends the weighted Trudinger-Moser inequality on the Heisenberg group obtained by Yang (see \cite[Theorem 1.1]{Yang14}).

\subsection{Caffarelli-Kohn-Nirenberg and Gagliardo-Nirenberg inequalities on gra\-ded groups}

The techniques developed in this paper also allow us to derive general hypoelliptic versions of Caffarelli-Kohn-Nirenberg and weighted Gagliardo-Nirenberg inequalities. Moreover, we can obtain asymptotically sharp constants as well as show that the obtained weighted Gagliardo-Nirenberg inequalities \eqref{Hardy_GN_Rock1_intro} are actually equivalent to the weighted
Trudinger-Moser inequalities with remainder terms given in \eqref{Tru-Mos1_intro}.

\begin{itemize}
\item {\bf (Caffarelli-Kohn-Nirenberg inequalities on graded groups)}. Let $1<p,q<\infty$, $\delta\in(0,1]$ and $0<r<\infty$ with $r\leq \frac{q}{1-\delta}$ for $\delta\neq1$. Let $0<a<Q/p$ and $\beta$, $\gamma\in\mathbb{R}$ with $\delta r (Q-ap-\beta p)\leq p(Q+r\gamma-r\beta)$ and $\beta (1-\delta)-\delta a \leq \gamma \leq \beta(1-\delta)$. Assume that $\frac{r(\delta Q+p(\beta(1-\delta)-\gamma-a\delta))}{pQ}+\frac{(1-\delta)r}{q}=1$. Then there exists a positive constant $C$ such that
\begin{equation}\label{CKN_thm2_intro}
\||x|^{\gamma}f\|_{L^{r}(\mathbb{G})}
\leq C \left\|\R^{\frac{a}{\nu}}f\right\|^{\delta}_{L^{p}(\mathbb{G})}
\left\||x|^{\beta}f\right\|^{1-\delta}_{L^{q}(\mathbb{G})}
\end{equation}
holds for all $f\in \dot{L}^{p}_{a}(\G)$.
\item {\bf (Weighted Gagliardo-Nirenberg inequalities on graded groups)}.
Let $1<p<\infty$ and $\beta\in[0,Q)$ with $Q/(Q-\beta)<\mu<\infty$. Then for any $p\leq q<\infty$ there exists a positive constant $C_{7}=C_{7}(p,Q,\beta,\mu,q)$ such that
$$\left\|\frac{f}{|x|^{\frac{\beta}{q}}}\right\|_{L^{q}(\G)}\leq C_{7}q^{1-1/p}\times$$
\begin{equation}\label{Hardy_GN_Rock1_intro}
\times(\|\R^{\frac{Q}{\nu p}}f\|_{L^{p}(\G)}^{1-p/q}\|f\|_{L^{p}(\G)}^{p/q}+\|\R^{\frac{Q}{\nu
p}}f\|_{L^{p}(\G)}^{1-p/(q\mu)}\|f\|_{L^{p}(\G)}^{p/(q\mu)})
\end{equation}
holds for all $f\in L^{p}_{Q/p}(\G)$, and such that $\underset{q\rightarrow\infty}{\rm limsup\;}C_{7}(p,Q,\beta,\mu,q)<\infty$. The asymptotically sharp constant for \eqref{Hardy_GN_Rock1_intro} in the sense of Remark \ref{rem_B} is given in Theorem \ref{Hardy_GN_Rock_thm}. Moreover, the weighted Gagliardo-Nirenberg inequalities \eqref{Hardy_GN_Rock1_intro} are actually equivalent to the weighted
Trudinger-Moser inequalities with remainder terms \eqref{Tru-Mos1_intro}.
\end{itemize}

\smallskip

 Similarly to \eqref{Hardy_wholeG2_log_intro1}, \eqref{Tru-Mos1_intro} and \eqref{Hardy_GN_Rock1_intro} were investigated in the Euclidean setting in \cite{NW10}. We also refer to \cite{RSY17_NoDEA} for the related analysis on stratified groups, \cite{RSY18_Tran} and \cite{RSY17_hom2} on homogeneous groups,
 namely, for Caffarelli-Kohn-Nirenberg type inequalities in terms of parameters but with radial derivative operator or horizontal gradient instead of Rockland operators.

 We note that for $\beta=\gamma=0$ \eqref{CKN_thm2_intro}
 also recoveres the Garliardo-Nirenberg inequality
 \eqref{CKN_thm2_exa_GN}, that is
\begin{equation}\label{CKN_thm2_exa_GNi}
\|f\|_{L^{r}(\mathbb{G})}
\leq C \left\|\R^{\frac{a}{\nu}}f\right\|^{\delta}_{L^{p}(\mathbb{G})}
\left\|f\right\|^{1-\delta}_{L^{q}(\mathbb{G})}
\end{equation}
for all $f\in \dot{L}^{p}_{a}(\G)\cap L^{q}(\G)$, previously established in \cite{RT17} and \cite{RTY17} with application to the global-in-time well-posedness of nonlinear hypoelliptic evolutions on graded groups, where $a>0$, $1<p<Q/a$, $1< q\leq r\leq pQ/(Q-ap)$ and $\delta=(1/q-1/r)(a/Q+1/q-1/p)^{-1}$.

We also refer to \cite{BFG12} for another type of Garliardo-Nirenberg inequality involving Besov norms on graded groups. 

\subsection{Integral Hardy inequalities on homogeneous groups}
\label{SEC:inthom}

The described hypoelliptic Hardy-Sobolev inequalities and their critical versions on graded groups follow from the following integral versions of Hardy inequalities that we can established in the setting of general homogeneous groups. For example, we can obtain the Hardy-Sobolev inequalities \eqref{Hardy_new1_intro} by taking $T^{(1)}_{a}$ in the following result to be the Riesz kernel of a positive Rockland operator. Similarly, we obtain its critical versions by taking $T^{(2)}_{a}$ in \eqref{Hardy_wholeG2_log_intro} to be a combination of Riesz and Bessel kernels.

Thus, let now $\mathbb{G}$ be a homogeneous group of homogeneous dimension $Q$, equipped with any fixed homogeneous quasi-norm $|\cdot|$. Then we have the following

\begin{itemize}
\item ({\bf Integral Hardy inequality on homogeneous groups}) Let $1<p\leq q<\infty$ and $0<a<Q/p$. Let $0\leq b<Q$ and $\frac{a}{Q}=\frac{1}{p}-\frac{1}{q}+\frac{b}{qQ}$. Assume that $|T^{(1)}_{a}(x)|\leq C_{2} |x|^{a-Q}$ for some positive $C_{2}=C_{2}(a,Q)$. Then there exists a positive constant $C_{1}=C_{1}(p,q,a,b)$ such that
\begin{equation}\label{Hardy_new1_intro}
\left\|\frac{f\ast T^{(1)}_{a}}{|x|^{\frac{b}{q}}}\right\|_{L^{q}(\G)}\leq
C_{1}\|f\|_{L^{p}(\G)}
\end{equation}
holds for all $f\in L^{p}(\G)$.
\item ({\bf Critical integral Hardy inequality on homogeneous groups}) let $1<p<r<\infty$ and $p<q<(r-1)p'$, where $1/p+1/p'=1$. Assume that for $a=Q/p$ we have
\begin{equation}\label{T2_B_intro}
|T^{(2)}_{a}(x)|\leq
C_{2}\begin{cases} |x|^{a-Q}, \text{\;for}\;x\in\mathbb{G}\backslash\{0\},\\
|x|^{-Q},   \text{\;for}\;x\in\mathbb{G}\; \text{\;with}\;|x|\geq1,\end{cases}
\end{equation} for some positive $C_{2}=C_{2}(a,Q)$. Then there exists a positive constant $C_{1}=C_{1}(p, q, r, Q)$ such that
\begin{equation}\label{Hardy_wholeG2_log_intro}
\left\|\frac{f\ast T^{(2)}_{Q/p}}{\left(\log\left(e+\frac{1}{|x|}\right)\right)^{\frac{r}{q}}|x|^{\frac{Q}{q}}}\right\|_{L^{q}(\G)}\leq
C_{1}\|f\|_{L^{p}(\G)}
\end{equation}
holds for all $f\in L^{p}(\G)$.
\end{itemize}
In the proof of \eqref{Hardy_new1_intro} and \eqref{Hardy_wholeG2_log_intro} the following characterisation of weighted integral Hardy type inequalities plays an important role. In fact, the following results provide the characterisation of pairs of weights for the integral versions of Hardy inequalities to hold. For brevity, we only indicate the type of the obtained results referring to the corresponding theorems for precise characterising conditions.

\begin{itemize}
\item ({\bf Integral Hardy inequality for $p\leq q$ on homogeneous groups}) Let $\{\phi_{i}\}_{i=1}^{2}$ and $\{\psi_{i}\}_{i=1}^{2}$ be positive functions on $\G$, and $1<p\leq q<\infty$. Then we have
\begin{equation}\label{high_Hardy1_intro}
\left(\int_{\G}\left(\int_{B(0,|x|)}f(z)dz\right)^{q}\phi_{1}(x)dx\right)^{\frac{1}{q}}\leq C_{1}
\left(\int_{\G}(f(x))^{p}\psi_{1}(x)dx\right)^{\frac{1}{p}}
\end{equation}
and
\begin{equation}\label{high_Hardy2_intro}
\left(\int_{\G}\left(\int_{\G\backslash B(0,|x|)}f(z)dz\right)^{q}\phi_{2}(x)dx\right)^{\frac{1}{q}}\leq C_{2}
\left(\int_{\G}(f(x))^{p}\psi_{2}(x)dx\right)^{\frac{1}{p}}
\end{equation}
hold for all $f\geq 0$ a.e. on $\G$ if and only if $A_{i}(\phi_{i},\psi_{i})<\infty$, $i=1,2$, where $\{A_{i}\}_{i=1}^{2}$ are given in \eqref{high_Hardy3}-\eqref{high_Hardy4}.
\item ({\bf Integral Hardy inequality for $p>q$ on homogeneous groups}) Let $\{\phi_{i}\}_{i=3}^{4}$ and $\{\psi_{i}\}_{i=3}^{4}$ be positive functions on $\G$, and $1<q<p<\infty$ with $1/\delta=1/q-1/p$. Then we have
\begin{equation}\label{high2_Hardy1_intro}
\left(\int_{\G}\left(\int_{B(0,|x|)}f(z)dz\right)^{q}\phi_{3}(x)dx\right)^{\frac{1}{q}}\leq C_{1}
\left(\int_{\G}(f(x))^{p}\psi_{3}(x)dx\right)^{\frac{1}{p}}
\end{equation}
and
\begin{equation}\label{high2_Hardy2_intro}
\left(\int_{\G}\left(\int_{\G\backslash B(0,|x|)}f(z)dz\right)^{q}\phi_{4}(x)dx\right)^{\frac{1}{q}}\leq C_{2}
\left(\int_{\G}(f(x))^{p}\psi_{4}(x)dx\right)^{\frac{1}{p}}
\end{equation}
hold for all $f\geq 0$ if and only if $A_{i}(\phi_{i},\psi_{i})<\infty$, $i=3,4$, where $\{A_{i}\}_{i=3}^{4}$ are given in \eqref{high2_Hardy3}-\eqref{high2_Hardy4}.

\item ({\bf Weighted Hardy inequality on homogeneous groups}) Let $\phi_{5},\psi_{5}$ be positive weight functions on $\G$ and let $1<p\leq q<\infty$. Then there exists a positive constant $C$ such that
\begin{equation}\label{weight_hardy1_intro}
\left(\int_{\G}\phi_{5}(x)|f(x)|^{q}dx\right)^{1/q}\leq C\left(\int_{\G}\psi_{5}(x)|\R_{|x|} f(x)|^{p}dx\right)^{1/p}
\end{equation}
holds for radial functions $f$ with $f(0)=0$ if and only if $A_{5}(\phi_{5},\psi_{5})<\infty$, where $A_{5}$ is given in \eqref{weight_hardy2} and $\mathcal{R}_{|x|}:=\frac{d}{d|x|}$ is the radial derivative.
\end{itemize}
We note that Hardy, Rellich and other related inequalities with respect to the radial derivative $\mathcal{R}_{|x|}$
have been investigated in \cite{RS17_AM}.

\subsection{Weighted Hardy-Littlewood-Sobolev inequalities}

Let us give another illustration of the method of applying inequalities on homogeneous groups to obtain the corresponding hypoelliptic inequalities. First, in this paper we show that the integral Hardy inequality \eqref{Hardy_new1_intro} implies the following weighted version of Hardy-Littlewood-Sobolev inequalities, still on general homogeneous groups:

\begin{itemize}
\item ({\bf Weighted Hardy-Littlewood-Sobolev inequalities on homogeneous groups}) Let $0<\lambda<Q$ and $1<p,q<\infty$ be such that $1/p+1/q+(\alpha+\lambda)/Q=2$ with $0\leq \alpha <Q/p'$ and $\alpha+\lambda\leq Q$, where $1/p+1/p'=1$. Then there exists a positive constant $C=C(Q,\lambda, p, \alpha)$ such that
\begin{equation}\label{HLS_ineq1_intro}
\left|\int_{\G}\int_{\G}\frac{\overline{f(x)}g(y)}{|x|^{\alpha}|y^{-1}x|^{\lambda}}dxdy\right|\leq C\|f\|_{L^{p}(\G)}\|g\|_{L^{q}(\G)}
\end{equation}
holds for all $f\in L^{p}(\G)$ and $g\in L^{q}(\G)$.
\end{itemize}
Consequently, similar to the outline of Section \ref{SEC:inthom}, by working with Riesz kernels of positive Rockland operators, we subsequently obtain the following hypoelliptic differential version of Hardy-Littlewood-Sobolev inequalities:
\begin{itemize}
\item {\bf (Weighted Hardy-Littlewood-Sobolev inequalities on graded groups)}. Let $1<p,q<\infty$, $0\leq a<Q/p$ and $0\leq b<Q/q$. Let $0<\lambda<Q$, $0\leq \alpha <a+Q/p'$ and $0\leq \beta\leq b$ be such that $(Q-ap)/(pQ)+(Q-q(b-\beta))/(qQ)+(\alpha+\lambda)/Q=2$ and $\alpha+\lambda\leq Q$, where $1/p+1/p'=1$. Then there exists a positive constant $C=C(Q,\lambda, p, \alpha, \beta, a, b)$ such that
\begin{equation}\label{HLS_ineq1_grad_intro}
\left|\int_{\G}\int_{\G}\frac{\overline{f(x)}g(y)}{|x|^{\alpha}|y^{-1}x|^{\lambda}|y|^{\beta}}dxdy\right|\leq C\|f\|_{\dot{L}^{p}_{a}(\G)}\|g\|_{\dot{L}^{q}_{b}(\G)}
\end{equation}
holds for all $f\in \dot{L}^{p}_{a}(\G)$ and $g\in \dot{L}^{q}_{b}(\G)$.
\end{itemize}
Certainly, the Hardy-Littlewood-Sobolev inequalities is a very classical subject going back to  Hardy-Littlewood \cite{HL28}, \cite{HL30} and Sobolev \cite{Sob38}. In the setting of homogeneous groups, it was established by Folland and Stein \cite{FS74} on the Heisenberg group, and its sharp constants were also investigated in \cite{Lie83} and \cite{FL12} in the Euclidean and in the Heisenberg group settings.

\smallskip

The organisation of the paper is as follows. In Section \ref{SEC:prelim} we briefly recall the necessary concepts of homogeneous Lie
groups and fix the notation. In Section \ref{SEC:Hardy} we introduce the weighted integral Hardy inequalities and in Section \ref{SEC:HLS} we apply them to obtain the Hardy-Littlewood-Sobolev inequalities on homogeneous groups. The Hardy, Rellich and Caffarelli-Kohn-Nirenberg inequalities on graded groups are established in Section \ref{SEC:Hardy_grad}. In Section \ref{SEC:weighted_Trudinger} we give the local and global weighted Trudinger-Moser inequalities with remainder terms, and show their equivalence with the critical Hardy type inequalities on graded groups. Finally, the weighted Gagliardo-Nirenberg type inequalities and their equivalence to the weighted Trudinger-Moser type inequalities with remainder terms are presented in Section \ref{SEC:GN}.

\smallskip

The authors would like to thank Fulvio Ricci for a valuable discussion.

\section{Preliminaries}
\label{SEC:prelim}

Following Folland and Stein \cite[Chapter 1]{FS-book} and the recent exposition in \cite[Chapter 3]{FR16} let us recall that a family of dilations of a Lie algebra $\mathfrak{g}$ is a family of linear mappings of the following form
$$D_{\lambda}={\rm Exp}(A \,{\rm ln}\lambda)=\sum_{k=0}^{\infty}
\frac{1}{k!}({\rm ln}(\lambda) A)^{k},$$
where $A$ is a diagonalisable linear operator on $\mathfrak{g}$ with positive eigenvalues. We also recall that $D_{\lambda}$ is a morphism of $\mathfrak{g}$ if it is a linear mapping from $\mathfrak{g}$ to itself satisfying the property
$$\forall X,Y\in \mathfrak{g},\, \lambda>0,\;
[D_{\lambda}X, D_{\lambda}Y]=D_{\lambda}[X,Y],$$
where $[X,Y]:=XY-YX$ is the Lie bracket. Then, a {\em homogeneous group} $\G$ is a connected simply connected Lie group whose Lie algebra is equipped with a morphism family of dilations. It induces the dilation structure on $\mathbb G$ which we denote by $D_{\lambda}x$ or just by $\lambda x$.

We call $\mathbb{G}$ a {\em graded Lie group} if its Lie algebra $\mathfrak{g}$ admits a gradation
$$\mathfrak{g}=\bigoplus_{i=1}^{\infty}\mathfrak{g}_{i},$$
where the $\mathfrak{g}_{1}, \mathfrak{g}_{2},...,$ are vector subspaces of the Lie algebra $\mathfrak{g}$, all but finitely
many equal to $\{0\}$, and satisfying
$$[\mathfrak{g}_{i},\mathfrak{g}_{j}]\subset \mathfrak{g}_{i+j} \;\;\forall i, j\in \mathbb{N}.$$

Every graded Lie group is also a homogeneous group with the dilation structure induced by the commutator relations.

The triple $\mathbb{G}=(\mathbb{R}^{n}, \circ, D_{\lambda})$ is called a {\em stratified group} if it satisfies the conditions:
\begin{itemize}
\item For some natural numbers $N=N_{1},N_{2},...,N_{r}$ with $N+N_{2}+\ldots+N_{r}=n$, the following decomposition $\mathbb{R}^{n}=\mathbb{R}^{N}\times\ldots\times\mathbb{R}^{N_{r}}$ is valid, and for each $\lambda>0$ the dilation $D_{\lambda}:\mathbb{R}^{n}\rightarrow\mathbb{R}^{n}$ defined by
    $$D_{\lambda}(x)=D_{\lambda}(x',x^{(2)},\ldots,x^{(r)}):=
    (\lambda x', \lambda^{2}x^{(2)},\ldots,\lambda^{r}x^{(r)})$$ is an automorphism of the stratified group $\mathbb{G}$. Here $x'\equiv x^{(1)}\in\mathbb{R}^{N}$ and $x^{(k)}\in\mathbb{R}^{N_{k}}$ for $k=2,\ldots,r$.
\item Let $N$ be as in above and let $X_{1}, \ldots, X_{N}$ be the left invariant vector fields on the stratified group $\mathbb{G}$ such that $X_{k}(0)=\frac{\partial}{\partial x_{k}}|_{0}$ for $k=1, \ldots, N$. Then
    $${\rm rank}({\rm Lie}\{X_{1}, \ldots, X_{N}\})=n,$$
for each $x\in\mathbb{R}^{n}$, that is, the iterated commutators of $X_{1}, \ldots, X_{N}$ span the Lie algebra of the stratified group $\mathbb{G}$.
\end{itemize}

Note that the left invariant vector fields $X_{1}, \ldots, X_{N}$ are called the (Jacobian) generators of the stratified group $\mathbb{G}$ and $r$ is called a step of this stratified group $\mathbb{G}$.
For the expressions for left invariant vector fields on $\G$ in terms of the usual (Euclidean) derivatives and further properties see e.g. \cite[Section 3.1.5]{FR16}.

As usual we always assume that $\G$ is connected and simply connected. If we fix a basis $\{X_{1},\ldots,X_{n}\}$ of
$\mathfrak{g}$ adapted to the gradation, then by the exponential mapping $\exp_{\mathbb{G}}:\mathfrak{g}\rightarrow\mathbb{G}$
we obtain points $x\in\mathbb{G}$:
$$x=\exp_{\mathbb{G}}(x_{1}X_{1}+\ldots+x_{n}X_{n}).$$
Let $A$ be a diagonalisable linear operator on the Lie algebra $\mathfrak{g}$ with positive eigenvalues. Then, a family of
linear mappings of the form
$$D_{r}={\rm Exp}(A \,{\rm ln}r)=\sum_{k=0}^{\infty}
\frac{1}{k!}({\rm ln}(r) A)^{k}$$
is a family of dilations of the Lie algebra $\mathfrak{g}$. Each $D_{r}$ is a morphism of $\mathfrak{g}$, that is, $D_{r}$ is a
linear mapping from the Lie algebra $\mathfrak{g}$ to itself with the following property
$$\forall X,Y\in \mathfrak{g},\, r>0,\;
[D_{r}X, D_{r}Y]=D_{r}[X,Y],$$
where $[X,Y]:=XY-YX$ is the Lie bracket. We can always extend these dilations through the exponential mapping to the group $\G$
by
\begin{equation}\label{dil_weight}
D_{r}(x)=rx:=(r^{\nu_{1}}x_{1},\ldots,r^{\nu_{n}}x_{n}), \;\;x=(x_{1},\ldots,x_{n})\in\mathbb{G},\;\;r>0,
\end{equation}
where $\nu_{1},\ldots,\nu_{n}$ are weights of the dilations. The sum of these weights
$$
Q:={\rm Tr}\, A=\nu_1+\cdots+\nu_n
$$
is called the homogeneous dimension of $\G$. Recall the fact that the standard Lebesgue measure $dx$ on $\mathbb{R}^{n}$ is the
Haar measure for $\mathbb{G}$ (see, e.g. \cite[Proposition 1.6.6]{FR16}). The continuous non-negative function
$$\mathbb{G}\ni x\mapsto |x|\in [0,\infty)$$
satisfying the following properties:
\begin{itemize}
\item   $|x^{-1}| = |x|$ for any $x\in \mathbb{G}$,
\item  $|\lambda x|=\lambda |x|$ for any
$x\in \mathbb{G}$ and $\lambda >0$,
\item  $|x|= 0$ if and only if $x=0$,
\end{itemize} is called a {homogeneous quasi-norm} on $\mathbb G$.

In the sequel we will need the following well-known facts, see e.g. \cite[Proposition 3.1.38 and Theorem 3.1.39]{FR16}:
\begin{prop}\label{triangle_euc}
Let $\G$ be a homogeneous Lie group and let $|\cdot|$ be an arbitrary homogeneous quasi-norm on $\G$. Then there exists a
constant $C_{0}$ such that
\begin{equation}\label{triangle}
|xy|\leq C_{0}(|x|+|y|)
\end{equation}
holds for all $x,y\in\G$.
At the same time, there always exists a homogeneous quasi-norm $|\cdot|$ on $\G$ which satisfies the triangle inequality
\begin{equation}\label{triangle2}
|xy|\leq |x|+|y|
\end{equation}
for all $x,y\in\G$.
\end{prop}

The quasi-ball centred at $x\in\mathbb{G}$ with radius $R > 0$ can be defined by
$$B(x,R):=\{y\in \mathbb{G}: |x^{-1}y|<R\}.$$

There exists a (unique) positive Borel measure $\sigma$ on the sphere
\begin{equation}\label{EQ:sphere}
\wp:=\{x\in \mathbb{G}:\,|x|=1\},
\end{equation}
such that for all $f\in L^{1}(\mathbb{G})$ there holds
\begin{equation}\label{EQ:polar}
\int_{\mathbb{G}}f(x)dx=\int_{0}^{\infty}
\int_{\wp}f(ry)r^{Q-1}d\sigma(y)dr.
\end{equation}

We denote by $\widehat{\mathbb{G}}$ the unitary dual of $\mathbb{G}$ and by $\mathcal{H}_{\pi}^{\infty}$ the space of smooth
 vectors for a representation $\pi\in\widehat{\mathbb{G}}$. If the left-invariant differential operator $\mathcal{R}$ on
 $\mathbb{G}$, which is homogeneous of positive degree, satisfies the following condition:

({\bf Rockland condition}) for every representation $\pi\in\widehat{\mathbb{G}}$, except for the trivial representation, the
operator $\pi(\R)$ is injective on $\mathcal{H}_{\pi}^{\infty}$, that is,
$$\forall \upsilon \in \mathcal{H}_{\pi}^{\infty}, \;\;\pi(\R)\upsilon=0\Rightarrow \upsilon=0,$$
then the left-invariant differential operator $\mathcal{R}$ is called a Rockland operator. Here, $\pi(\R):=d\pi(\R)$ is the
infinitesimal representation of the Rockland operator $\R$ as of an element of the universal enveloping algebra of $\G$.

Different characterisations of the Rockland operators have been obtained by Rockland \cite{Rockland} and Beals
\cite{Beals-Rockland}. We refer to \cite{FR:Sobolev} and \cite[Chapter 4]{FR16} for an extensive presentation about Rockland
operators and for the theory of Sobolev spaces on graded groups, and refer to \cite{CR17} for the Besov spaces on graded Lie
groups.

By Helffer and Nourrigat \cite{HN-79}, we know that one can also define {\em Rockland operators as left-invariant homogeneous
hypoelliptic differential operators on $\G$}, since this is equivalent to the Rockland condition.

Since we will deal with the Riesz and Bessel potentials, let us recall them on graded groups, and prove some useful estimates. Let $\R$ be a positive Rockland operator of homogeneous degree $\nu$. Then, the operators $\R^{-a/\nu}$ for $\{a\in\mathbb{R}, 0<a<Q\}$ and $(I+\R)^{-a/\nu}$ for $a\in \mathbb{R}_{+}$ are called Riesz and Bessel potentials, respectively. If we denote their kernels by $\mathcal{I}_{a}$ and $\mathcal{B}_{a}$, then we have
\begin{equation}\label{Rie_pot}
\mathcal{I}_{a}(x):=\frac{1}{\Gamma\left(\frac{a}{\nu}\right)}\int_{0}^{\infty}
t^{\frac{a}{\nu}-1}h_{t}(x)dt
\end{equation}
for $0<a<Q$ with $a\in\mathbb{R}$, and
\begin{equation}\label{Bes_pot}
\mathcal{B}_{a}(x):=\frac{1}{\Gamma\left(\frac{a}{\nu}\right)}\int_{0}^{\infty}
t^{\frac{a}{\nu}-1}e^{-t}h_{t}(x)dt
\end{equation}
for $a>0$, where $\Gamma$ denotes the Gamma function, and $h_{t}$ is the heat kernel associated to the positive Rockland operator $\R$. We refer for more details to \cite[Section 4.3.4]{FR16}.

Before using $\mathcal{I}_{a}(x)$ and $\mathcal{B}_{a}(x)$, we recall the following results:
\begin{thm}[{\cite[Theorem 4.2.7]{FR16}}]
\label{FR__Bes_thm} Let $\R$ be a positive Rockland operator on a graded Lie group $\G$. Let $|\cdot|$ be a fixed homogeneous quasi-norm. Let $h_{t}$ be a heat kernel associated with the Rockland operator. Then each $h_{t}$ is Schwartz and we have
\begin{equation}\label{ht_1}
\forall s,t>0\;\;h_{t}\ast h_{s}=h_{t+s},
\end{equation}
\begin{equation}\label{ht_2}
\forall x\in \G, r,t>0\;\;h_{r^{\nu}t}(rx)=r^{-Q}h_{t}(x),
\end{equation}
\begin{equation}\label{ht_3}
\forall x\in \G\;\;h_{t}(x)=\overline{h_{t}(x^{-1})},
\end{equation}
\begin{equation}\label{ht_4}
\int_{\G}h_{t}(x)dx=1.
\end{equation}
Moreover, we have
\begin{equation}\label{ht_5}
\exists C=C_{\alpha,N,\ell}>0\;\;\forall t\in(0,1]\;\;\sup_{|x|=1}|\partial_{t}^{\ell}X^{\alpha}h_{t}(x)|\leq C_{\alpha, N}t^{N}
\end{equation}
for any $N\in\mathbb{N}_{0}$, $\alpha\in \mathbb{N}_{0}^{n}$ and $\ell\in \mathbb{N}_{0}$.
\end{thm}
\begin{lem}[{\cite[Lemma 4.3.8]{FR16}}]
\label{FR__Rie_lem} Let $\R$ be a positive Rockland operator on graded group $\G$ and let $h_{t}$ be its heat kernel as in Theorem \ref{FR__Bes_thm}. Let $|\cdot|$ be a homogeneous quasi-norm and $\alpha\in \mathbb{N}^{n}_{0}$ be a multi-index. Then for any real number $a$ with $0<a<(Q+[\alpha])/\nu$ there exists a positive constant $C$ such that
\begin{equation}\label{FR__Rie_lem_eq1}
\int_{0}^{\infty}t^{a-1}|X^{\alpha}h_{t}(x)|dt\leq C|x|^{-Q-[\alpha]+\nu a}.
\end{equation}
\end{lem}
Replacing $a$ by $a/\nu$ and putting $\alpha=0$ in Lemma \ref{FR__Rie_lem}, and using the representation \eqref{Rie_pot} for $\mathcal{I}_{a}(x)$, we obtain
\begin{lem}\label{Rie_lem} Let $|\cdot|$ be a homogeneous quasi-norm. Let $0<a<Q$ and $a\in\mathbb{R}$. Then there exists a positive constant $C=C(Q,a)$ such that
\begin{equation}\label{Rie_lem1}
|\mathcal{I}_{a}(x)|\leq C|x|^{-(Q-a)}.
\end{equation}
\end{lem}
Now let us prove the following useful lemma for $\mathcal{B}_{a}$, which may be not optimal, but sufficient for our purposes.
\begin{lem}\label{Bes_lem} Let $|\cdot|$ be a homogeneous quasi-norm. Let $0<a<Q$ and $a\in\mathbb{R}$. Then there exists a positive constant $C=C(Q,a)$ such that
\begin{equation}\label{Bes_lem1}
|\mathcal{B}_{a}(x)|\leq
\begin{cases} C|x|^{-(Q-a)}, \text{\;for}\;x\in\mathbb{G}\backslash\{0\},\\
C|x|^{-Q},   \text{\;for}\;x\in\mathbb{G}\; \text{\;with}\;|x|\geq1.\end{cases}
\end{equation}
\end{lem}
\begin{proof}[Proof of Lemma \ref{Bes_lem}] We split the integral in \eqref{Bes_lem1} as follows
\begin{equation}\label{Bes_lem_eq1}
\begin{split}
\mathcal{B}_{a}(x)&=\frac{1}{\Gamma\left(\frac{a}{\nu}\right)}\int_{0}^{\infty}
t^{\frac{a}{\nu}-1}e^{-t}h_{t}(x)dt\\&=\frac{1}{\Gamma\left(\frac{a}{\nu}\right)}\int_{0}^{|x|^{\nu}}
t^{\frac{a}{\nu}-1}e^{-t}h_{t}(x)dt+\frac{1}{\Gamma\left(\frac{a}{\nu}\right)}\int_{|x|^{\nu}}^{\infty}
t^{\frac{a}{\nu}-1}e^{-t}h_{t}(x)dt\\&=:J_{1}+J_{2}.
\end{split}
\end{equation}
To estimate $J_{1}$ using the property of homogeneity of $h_{t}$ in \eqref{ht_2}, we calculate
\begin{equation}\label{Bes_lem_eq2}
\begin{split}
|J_{1}|&=\left|\frac{1}{\Gamma\left(\frac{a}{\nu}\right)}\int_{0}^{|x|^{\nu}}
t^{\frac{a}{\nu}-1}e^{-t}|x|^{-Q}h_{|x|^{-\nu}t}\left(\frac{x}{|x|}\right)dt\right|\\&
\leq \frac{1}{\Gamma\left(\frac{a}{\nu}\right)}|x|^{-Q}\left(\sup_{|y|=1, 0\leq t_{1}\leq 1}|h_{t_{1}}(y)|\right)\int_{0}^{|x|^{\nu}}
t^{\frac{a}{\nu}-1}dt\\&=
\frac{\nu}{a\Gamma\left(\frac{a}{\nu}\right)}|x|^{a-Q}\left(\sup_{|y|=1, 0\leq t_{1}\leq 1}|h_{t_{1}}(y)|\right)\\&
\leq C|x|^{a-Q},
\end{split}
\end{equation}
where we have used that $\underset{|y|=1, 0\leq t_{1}\leq 1}{\rm sup}|h_{t_{1}}(y)|$ is finite by \eqref{ht_5}.

Now we estimate $J_{2}$. A direct calculation gives that
\begin{equation}\label{Bes_lem_eq3}
\begin{split}
|J_{2}|&=\left|\frac{1}{\Gamma\left(\frac{a}{\nu}\right)}\int_{|x|^{\nu}}^{\infty}
t^{\frac{a}{\nu}-1}e^{-t}h_{t}(x)dt\right|\\&
\leq \frac{1}{\Gamma\left(\frac{a}{\nu}\right)}\int_{|x|^{\nu}}^{\infty}
t^{\frac{a}{\nu}-1}t^{-\frac{Q}{\nu}}|h_{1}(t^{-\frac{1}{\nu}}x)|dt\\&\leq
\frac{1}{\Gamma\left(\frac{a}{\nu}\right)}\|h_{1}\|_{L^{\infty}(\G)}\int_{|x|^{\nu}}^{\infty}
t^{\frac{a}{\nu}-1-\frac{Q}{\nu}}dt\\&
\leq C|x|^{a-Q},
\end{split}
\end{equation}
where we have used that $\|h_{1}\|_{L^{\infty}(\G)}$ is finite since $h_{1}$ is Schwartz. Combining \eqref{Bes_lem_eq1}, \eqref{Bes_lem_eq2} and \eqref{Bes_lem_eq3}, we obtain \eqref{Bes_lem1}.

On the other hand, when $|x|\geq1$, one has for $J_{1}$
\begin{equation}\label{Bes_lem_eq22}
\begin{split}
|J_{1}|&=\left|\frac{1}{\Gamma\left(\frac{a}{\nu}\right)}\int_{0}^{|x|^{\nu}}
t^{\frac{a}{\nu}-1}e^{-t}|x|^{-Q}h_{|x|^{-\nu}t}\left(\frac{x}{|x|}\right)dt\right|\\&
\leq \frac{1}{\Gamma\left(\frac{a}{\nu}\right)}|x|^{-Q}\left(\sup_{|y|=1, 0\leq t_{1}\leq 1}|h_{t_{1}}(y)|\right)\int_{0}^{|x|^{\nu}}
t^{\frac{a}{\nu}-1}e^{-t}dt\\&\leq
|x|^{-Q}\left(\sup_{|y|=1, 0\leq t_{1}\leq 1}|h_{t_{1}}(y)|\right)\\&
\leq C|x|^{-Q}.
\end{split}
\end{equation}
It remains to estimate $J_{2}$ for $|x|\geq1$. By a direct calculation, we obtain
\begin{equation}\label{Bes_lem_eq32}
\begin{split}
|J_{2}|&=\left|\frac{1}{\Gamma\left(\frac{a}{\nu}\right)}\int_{|x|^{\nu}}^{\infty}
t^{\frac{a}{\nu}-1}e^{-t}h_{t}(x)dt\right|\\&
= \left|\frac{1}{\Gamma\left(\frac{a}{\nu}\right)}\int_{|x|^{\nu}}^{\infty}
t^{\frac{a}{\nu}-1-\frac{Q}{\nu}}e^{-t}h_{1}(t^{-\frac{1}{\nu}}x)dt\right|\\&\leq
\frac{1}{\Gamma\left(\frac{a}{\nu}\right)}|x|^{-Q}\|h_{1}\|_{L^{\infty}(\G)}\Gamma\left(\frac{a}{\nu}\right)\\&
\leq C|x|^{-Q}.
\end{split}
\end{equation}
Combining \eqref{Bes_lem_eq1}, \eqref{Bes_lem_eq22} and \eqref{Bes_lem_eq32}, we obtain \eqref{Bes_lem1} for $|x|\geq1$.
\end{proof}

\section{Weighted integral Hardy inequalities on homogeneous groups}
\label{SEC:Hardy}
In this section we introduce various types of weighted $L^{p}-L^{q}$ inequalities for the Hardy operator on homogeneous groups for different ranges of indices $1<p,q<\infty$. We obtain necessary and sufficient condition on weights for such inequalities to be true. Subsequently, we apply them to obtain an integral Hardy inequality on general homogeneous groups which will be crucial for the further investigation of this paper.

\begin{thm}\label{high_Hardy_thm} Let $\mathbb{G}$ be a homogeneous group
of homogeneous dimension $Q$. Let $\{\phi_{i}\}_{i=1}^{2}$ and $\{\psi_{i}\}_{i=1}^{2}$ be positive functions on $\G$, and let $1<p\leq q<\infty$. Then the inequalities
\begin{equation}\label{high_Hardy1}
\left(\int_{\G}\left(\int_{B(0,|x|)}f(z)dz\right)^{q}\phi_{1}(x)dx\right)^{\frac{1}{q}}\leq C_{1}
\left(\int_{\G}(f(x))^{p}\psi_{1}(x)dx\right)^{\frac{1}{p}}
\end{equation}
and
\begin{equation}\label{high_Hardy2}
\left(\int_{\G}\left(\int_{\G\backslash B(0,|x|)}f(z)dz\right)^{q}\phi_{2}(x)dx\right)^{\frac{1}{q}}\leq C_{2}
\left(\int_{\G}(f(x))^{p}\psi_{2}(x)dx\right)^{\frac{1}{p}}
\end{equation}
hold for all $f\geq 0$ a.e. on $\G$ if and only if, respectively, we have
\begin{equation}\label{high_Hardy3}
A_{1}:=\sup_{R>0}\left(\int_{\{|x|\geq R\}}\phi_{1}(x)dx\right)^{\frac{1}{q}}
\left(\int_{\{|x|\leq R\}}(\psi_{1}(x))^{-(p'-1)}dx\right)^{\frac{1}{p'}}<\infty
\end{equation}
and
\begin{equation}\label{high_Hardy4}
A_{2}:=\sup_{R>0}\left(\int_{\{|x|\leq R\}}\phi_{2}(x)dx\right)^{\frac{1}{q}}
\left(\int_{\{|x|\geq R\}}(\psi_{2}(x))^{-(p'-1)}dx\right)^{\frac{1}{p'}}<\infty.
\end{equation}
Moreover, if $\{C_{i}\}_{i=1}^{2}$ are the smallest constants for which \eqref{high_Hardy1} and \eqref{high_Hardy2} hold, then
\begin{equation}\label{high_Hardy5}
A_{i}\leq C_{i}\leq (p')^{\frac{1}{p'}}p^{\frac{1}{q}}A_{i},\;\; i=1,2.
\end{equation}
\end{thm}
\begin{rem}\label{high_Hardy_thm_rem} In the abelian case $\mathbb{G}=(\Rn,+)$ and $Q=n$, if we take $p=q>1$ and $\phi_{1}(x)=|B(0,|x|)|^{-p}$ and $\psi_{1}(x)=1$ in \eqref{high_Hardy1}, then we have $A_{1}=(p-1)^{-1/p}$ and
\begin{equation}\label{high_Hardy1_rem}
\left(\int_{\Rn}\left|\frac{1}{|B(0,|x|)|}\int_{B(0,|x|)}f(z)dz\right|^{p}dx\right)^{\frac{1}{p}}\leq \frac{p}{p-1}
\left(\int_{\Rn}|f(x)|^{p}dx\right)^{\frac{1}{p}},
\end{equation}
where $|B(0,|x|)|$ is the volume of the ball $B(0,|x|)$. The inequality \eqref{high_Hardy1_rem} was obtained in \cite{CG95}.
\end{rem}
\begin{proof}[Proof of Theorem \ref{high_Hardy_thm}] We prove \eqref{high_Hardy1}$\Leftrightarrow$\eqref{high_Hardy3}, the case \eqref{high_Hardy2}$\Leftrightarrow$\eqref{high_Hardy4} can be proved similarly.

First, we show \eqref{high_Hardy3}$\Rightarrow$\eqref{high_Hardy1}. Then, using polar coordinates on $\G$ and denoting $r=|x|$, we write
\begin{equation}\label{high1}
\begin{split}
\int_{\G}&\phi_{1}(x)\left[\int_{B(0,r)}f(z)dz\right]^{q}dx\\&
=\int_{0}^{\infty}\int_{\wp}r^{Q-1}\phi_{1}(ry)\left[\int_{0}^{r}\int_{\wp}s^{Q-1}f(sy)d\sigma(y) ds\right]^{q}d\sigma(y) dr.
\end{split}
\end{equation}
Setting
\begin{equation}\label{g}
g(r)=\left\{\int_{\wp}\int_{0}^{r}s^{Q-1}(\psi_{1}(sy))^{1-p'}dsd\sigma(y)\right\}^{1/(pp')},
\end{equation}
and using H\"{o}lder's inequality, we calculate
\begin{equation}\label{high2}
\begin{split}
\int_{0}^{r}\int_{\wp}s^{Q-1}f(sy)d\sigma(y) ds&=\int_{\wp}\int_{0}^{r}s^{(Q-1)/p}f(sy)(\psi_{1}(sy))^{1/p}g(s)s^{(Q-1)/p'}\\&
\times \left((\psi_{1}(sy))^{1/p}g(s)\right)^{-1}dsd\sigma(y)\\&
\leq\left(\int_{\wp}\int_{0}^{r}s^{Q-1}\left[f(sy)(\psi_{1}(sy))^{1/p}g(s)\right]^{p}dsd\sigma(y)\right)^{1/p}\\&
\times\left(\int_{\wp}\int_{0}^{r}s^{Q-1}\left[(\psi_{1}(sy))^{1/p}g(s)\right]^{-p'}dsd\sigma(y)\right)^{1/p'}.
\end{split}
\end{equation}
If we define $U, V$ and $W_{1}$ by
\begin{equation}\label{U}
U(s)=\int_{\wp}s^{Q-1}\left(f(sy)(\psi_{1}(sy))^{1/p}g(s)\right)^{p}d\sigma(y),
\end{equation}
\begin{equation}\label{V}
V(r)=\int_{0}^{r}\int_{\wp}s^{Q-1}\left((\psi_{1}(sy))^{1/p}g(s)\right)^{-p'}d\sigma(y)ds,
\end{equation}
\begin{equation}\label{W}
W_{1}(r)=\int_{\wp}r^{Q-1}\phi_{1}(ry)d\sigma(y),
\end{equation}
for $s, r>0$, respectively, then plugging \eqref{high2} into \eqref{high1} we obtain
\begin{equation}\label{high3}
\int_{\G}\phi_{1}(x)\left(\int_{B(0,r)}f(z)dz\right)^{q}dx
\leq\int_{0}^{\infty}W_{1}(r)\left(\int_{0}^{r}U(s)ds\right)^{q/p}(V(r))^{q/p'}dr.
\end{equation}
Now we need to use the following continuous version of Minkowski's inequality (see e.g. \cite[Formula 2.1]{DHK97}): Let $\theta\geq1$. Then for all $f_{1}(x),f_{2}(x)\geq0$ on $(0,\infty)$, we have
\begin{equation}\label{Mink_for}
\int_{0}^{\infty}f_{1}(x)\left(\int_{0}^{x}f_{2}(z)dz\right)^{\theta}dx
\leq \left(\int_{0}^{\infty}f_{2}(z)\left(\int_{z}^{\infty}f_{1}(x)dx\right)^{1/\theta}dz\right)^{\theta}.
\end{equation}
Using this with $\theta=q/p\geq1$ in the right hand side of \eqref{high3}, we get
\begin{multline}\label{high4}
\int_{\G}\phi_{1}(x)\left(\int_{B(0,r)}f(z)dz\right)^{q}dx\\
\leq \left(\int_{0}^{\infty}U(s)\left(\int_{s}^{\infty}W_{1}(r)(V(r))^{q/p'}dr\right)^{p/q}ds\right)^{q/p}.
\end{multline}
In order to simplify the right hand side of above, denoting $$T(s):=\int_{\wp}s^{Q-1}(\psi_{1}(sy))^{1-p'}d\sigma(y),$$ and using \eqref{g}, \eqref{V}, the integration by parts, \eqref{high_Hardy3} and \eqref{W} we compute
\begin{equation*}
\begin{split}
V(r)&=\int_{\wp}\int_{0}^{r}s^{Q-1}(\psi_{1}(sy))^{1-p'}\left(\int_{0}^{s}\int_{\wp}t^{Q-1}(\psi_{1}(tw))^{1-p'}d\sigma(w) dt\right)^{-1/p}dsd\sigma(y)\\&
=\int_{0}^{r}T(s)\left(\int_{0}^{s}T(t)dt\right)^{-1/p}ds=p'\int^{r}_{0}\frac{d}{ds}\left(\int_{0}^{s}T(t)dt\right)^{1/p'}ds\\&
=p'\left(\int_{0}^{r}T(s)ds\right)^{1/p'}=p'\left(\int_{0}^{r}\int_{\wp}s^{Q-1}(\psi_{1}(sy))^{1-p'}d\sigma(y) ds\right)^{1/p'}\\&
\leq p' A_{1}\left(\int_{r}^{\infty}s^{Q-1}\int_{\wp}\phi_{1}(sw)d\sigma(w) ds\right)^{-1/q}=p'A_{1}\left(\int_{r}^{\infty}W_{1}(s)ds\right)^{-1/q}.
\end{split}
\end{equation*}
Similarly, applying the integration by parts and \eqref{high_Hardy3}, we have from above
\begin{equation}\label{high5}
\begin{split}\int_{s}^{\infty}&W_{1}(r)(V(r))^{q/p'}dr\\&
=(p'A_{1})^{q/p'}\int_{s}^{\infty}W_{1}(r)\left(\int_{r}^{\infty}W_{1}(s)ds\right)^{-1/p'}dr\\&
=(p'A_{1})^{q/p'}p\left(\int_{s}^{\infty}W_{1}(r)dr\right)^{1/p}\\&
=(p'A_{1})^{q/p'}p\left(\int_{s}^{\infty}\int_{\wp}r^{Q-1}\phi_{1}(ry)d\sigma(y) dr\right)^{1/p}\\&
\leq (p'A_{1})^{q/p'}pA_{1}^{q/p}\left(\int_{0}^{s}r^{Q-1}\int_{\wp}(\psi_{1}(ry))^{1-p'}d\sigma(y) dr\right)^{-q/(p'p)}\\&
=A_{1}^{q}(p')^{q/p'}p(g(s))^{-q},
\end{split}
\end{equation}
where we have used \eqref{g} in the last line. Putting \eqref{high5} in \eqref{high4} and recalling \eqref{U}, we obtain
\begin{equation}\label{high6}
\begin{split}
\int_{\G}\phi_{1}(x)\left(\int_{B(0,r)}f(z)dz\right)^{q}dx&
\leq \left(\int_{0}^{\infty}U(s)A_{1}^{p}(p')^{p-1}p^{p/q}(g(s))^{-p}ds\right)^{q/p}\\&
= A_{1}^{q}(p')^{q/p'}p\left(\int_{0}^{\infty}U(s)(g(s))^{-p}ds\right)^{q/p}\\&
=A_{1}^{q}(p')^{q/p'}p\left(\int_{0}^{\infty}\int_{\wp}s^{Q-1}(f(sy))^{p}\psi_{1}(sy)d\sigma(y)ds\right)^{q/p}\\&
=A_{1}^{q}(p')^{q/p'}p\left(\int_{\G}\psi_{1}(x)(f(x))^{p}dx\right)^{q/p},
\end{split}
\end{equation}
yielding \eqref{high_Hardy1} with $C_{1}=A_{1}(p')^{1/p'}p^{1/q}$.

Now it remains to show \eqref{high_Hardy1}$\Rightarrow$\eqref{high_Hardy3}. For that, we take $f(x)=(\psi_{1}(x))^{1-p'}\chi_{(0, R)}(|x|)$ with $R>0$ to get
\begin{multline}\label{high7}
\left(\int_{\G}\psi_{1}(x)(f(x))^{p}dx\right)^{1/p}\left(\int_{|x|\leq  R}(\psi_{1}(x))^{1-p'}dx\right)^{-1/p}\\=\left(\int_{|x|\leq R}(\psi_{1}(x))^{1-p'}dx\right)^{1/p}\left(\int_{|x|\leq  R}(\psi_{1}(x))^{1-p'}dx\right)^{-1/p}=1.
\end{multline}
Consequently, by \eqref{high_Hardy1} we have
\begin{multline}\label{high8}
C=C\left(\int_{\G}\psi_{1}(x)(f(x))^{p}dx\right)^{1/p}\left(\int_{|x|\leq  R}(\psi_{1}(x))^{1-p'}dx\right)^{-1/p}\\
\geq\left(\int_{\G}\phi_{1}(x)\left(\int_{|z|\leq|x|}f(z)dz\right)^{q}dx\right)^{1/q}\left(\int_{|x|\leq  R}(\psi_{1}(x))^{1-p'}dx\right)^{-1/p}\\
\geq\left(\int_{|x|\geq  R}\phi_{1}(x)\left(\int_{|z|\leq|x|}f(z)dz\right)^{q}dx\right)^{1/q}\left(\int_{|x|\leq  R}(\psi_{1}(x))^{1-p'}dx\right)^{-1/p}\\
=\left(\int_{|x|\geq R}\phi_{1}(x)dx\right)^{1/q}\left(\int_{|z|\leq R}(\psi_{1}(z))^{1-p'}dz\right)^{1/p'}.
\end{multline}
Combining \eqref{high7} and \eqref{high8}, we obtain \eqref{high_Hardy3} with $C\geq A_{1}$.
\end{proof}

Now we show the case $q<p$ of Theorem \ref{high_Hardy_thm}:
\begin{thm}\label{high2_Hardy_thm} Let $\mathbb{G}$ be a homogeneous group
of homogeneous dimension $Q$. Let $\{\phi_{i}\}_{i=3}^{4}$ and $\{\psi_{i}\}_{i=3}^{4}$ be positive functions on $\G$, and let $1<q<p<\infty$ with $1/\delta=1/q-1/p$. Then the inequalities
\begin{equation}\label{high2_Hardy1}
\left(\int_{\G}\left(\int_{B(0,|x|)}f(z)dz\right)^{q}\phi_{3}(x)dx\right)^{\frac{1}{q}}\leq C_{1}
\left(\int_{\G}(f(x))^{p}\psi_{3}(x)dx\right)^{\frac{1}{p}}
\end{equation}
and
\begin{equation}\label{high2_Hardy2}
\left(\int_{\G}\left(\int_{\G\backslash B(0,|x|)}f(z)dz\right)^{q}\phi_{4}(x)dx\right)^{\frac{1}{q}}\leq C_{2}
\left(\int_{\G}(f(x))^{p}\psi_{4}(x)dx\right)^{\frac{1}{p}}
\end{equation}
hold for all $f\geq 0$ if and only if, respectively, we have
\begin{equation}\label{high2_Hardy3}
A_{3}:=\int_{\G}\left(\int_{\G\backslash B(0,|x|)} \phi_{3}(z)dz\right)^{\delta/q}
\left(\int_{B(0,|x|)}(\psi_{3}(z))^{1-p'}dz\right)^{\delta/q'}(\psi_{3}(x))^{1-p'}dx<\infty
\end{equation}
and
\begin{equation}\label{high2_Hardy4}
A_{4}:=\int_{\G}\left(\int_{B(0,|x|)} \phi_{4}(z)dz\right)^{\delta/q}
\left(\int_{\G\backslash B(0,|x|)}(\psi_{4}(z))^{1-p'}dz\right)^{\delta/q'}(\psi_{4}(x))^{1-p'}dx<\infty.
\end{equation}
\end{thm}
\begin{proof}[Proof of Theorem \ref{high2_Hardy_thm}] We show \eqref{high2_Hardy1}$\Leftrightarrow$\eqref{high2_Hardy3}, the case \eqref{high2_Hardy2}$\Leftrightarrow$\eqref{high2_Hardy4} can be proved similarly.

First, we prove \eqref{high2_Hardy3}$\Rightarrow$\eqref{high2_Hardy1}. Denote
\begin{equation}\label{W2}
W_{2}(r):=\int_{\wp}r^{Q-1}\phi_{3}(ry)d\sigma(y)
\end{equation}
and
\begin{equation}\label{G}
G(s):=\int_{\wp}s^{Q-1}h(sy)(\psi_{3}(sy))^{1-p'}d\sigma(y)
\end{equation}
for $h\geq 0$ on $\G$. Then using polar coordinates on $\G$, we calculate
\begin{equation*}
\begin{split}
\int_{\G}&\phi_{3}(x)\left(\int_{B(0,|x|)}h(z)(\psi_{3}(z))^{1-p'}dz\right)^{q}dx\\&
=\int_{0}^{\infty}\int_{\wp}r^{Q-1}\phi_{3}(rw)d\sigma(w)
\left(\int_{0}^{r}\int_{\wp}s^{Q-1}h(sy)(\psi_{3}(sy))^{1-p'}d\sigma(y) ds\right)^{q}dr\\&
=\int_{0}^{\infty}W_{2}(r)\left(\int_{0}^{r}G(s)ds\right)^{q}dr\\&
=q\int_{0}^{\infty}G(s)\left(\int_{0}^{s}G(r)dr\right)^{q-1}\left(\int_{s}^{\infty}W_{2}(r)dr\right)ds\\&
=q\int_{\wp}\int_{0}^{\infty}s^{Q-1}h(sy)(\psi_{3}(sy))^{1-p'}
\left(\int_{0}^{s}\int_{\wp}r^{Q-1}h(rw)(\psi_{3}(rw))^{1-p'}d\sigma(w)dr\right)^{q-1}\\&
\times\left(\int_{s}^{\infty}W_{2}(r)dr\right)dsd\sigma(y)\\&
=q\int_{\wp}\int_{0}^{\infty}s^{Q-1}h(sy)(\psi_{3}(sy))^{(1-p')(\frac{1}{p}+\frac{q-1}{p}+\frac{p-q}{p})}
\\ &\times\left(\frac{\int_{\wp}\int_{0}^{s}r^{Q-1}h(rw)(\psi_{3}(rw))^{1-p'}drd\sigma(w)}
{\int_{\wp}\int_{0}^{s}r^{Q-1}(\psi_{3}(rw))^{1-p'}drd\sigma(w)}\right)^{q-1}\\&
\times
\left(\left(\int_{\wp}\int_{0}^{s}r^{Q-1}(\psi_{3}(rw))^{1-p'}drd\sigma(w)\right)^{q-1}
\left(\int_{s}^{\infty}W_{2}(r)dr\right)\right)dsd\sigma(y).
\end{split}
\end{equation*}
Here, using H\"{o}lder's inequality (with three factors) for $\frac{1}{p}+\frac{q-1}{p}+\frac{p-q}{p}=1$ we get
\begin{equation}\label{K123}\int_{\G}\phi_{3}(x)\left(\int_{B(0,|x|)}h(z)(\psi_{3}(z))^{1-p'}dz\right)^{q}dx
\leq qK_{1}K_{2}K_{3},
\end{equation}
where
\begin{equation}\label{K1}
K_{1}=\left(\int_{\wp}\int_{0}^{\infty}s^{Q-1}(h(sy))^{p}(\psi_{3}(sy))^{1-p'}dsd\sigma(y)\right)^{1/p}
=\left(\int_{\G}(h(x))^{p}(\psi_{3}(x))^{1-p'}dx\right)^{1/p},
\end{equation}
\begin{equation}\label{K2}
K_{2}=\left(\int_{\wp}\int_{0}^{\infty}s^{Q-1}(\psi_{3}(sy))^{1-p'}
\left(\frac{\int_{\wp}\int_{0}^{s}r^{Q-1}h(rw)(\psi_{3}(rw))^{1-p'}drd\sigma(w)}
{\int_{\wp}\int_{0}^{s}r^{Q-1}(\psi_{3}(rw))^{1-p'}drd\sigma(w)}\right)^{p}dsd\sigma(y)\right)^{\frac{q-1}{p}}
\end{equation}
and
\begin{multline}\label{K3}
K_{3}=\left(\int_{\wp}\int_{0}^{\infty}s^{Q-1}(\psi_{3}(sy))^{1-p'}\left(\int_{\wp}\int_{0}^{s}r^{Q-1}(\psi_{3}(rw))^{1-p'}dr d\sigma(w)\right)^{\frac{(q-1)p}{p-q}}\right.\\
\left.\times\left(\int_{s}^{\infty}W_{2}(r)dr\right)^{\frac{p}{p-q}}dsd\sigma(y)\right)^{\frac{p-q}{p}}.
\end{multline}

We have for $K_{2}$ that
$$K_{2}=
\left(\int_{\G}\frac{(\psi_{3}(x))^{1-p'}}{(\int_{B(0,|x|)}(\psi_{3}(z))^{1-p'}dz)^{p}}
\left(
\int_{B(0,|x|)}(\psi_{3}(z))^{1-p'}h(z)dz\right)^{p}dx\right)^{\frac{q-1}{p}}.$$
To apply \eqref{high_Hardy1} for $K_{2}$ with $p=q$, $f(x)=(\psi_{3}(x))^{1-p'}h(x)$ and
$$\phi_{1}(x)=\frac{(\psi_{3}(x))^{1-p'}}{(\int_{B(0,|x|)}(\psi_{3}(z))^{1-p'}dz)^{p}},\;\; \psi_{1}(x)=(\psi_{3}(x))^{(1-p')(1-p)},$$
we need to check the condition that
\begin{multline}\label{check1111}
A_{1}(R)=\left(\int_{|x|\geq R}(\psi_{3}(x))^{1-p'}\left(\int_{B(0,|x|)}(\psi_{3}(z))^{1-p'}dz\right)^{-p}dx\right)^{1/p}\\\times
\left(\int_{|x|\leq R }(\psi_{3}(x))^{1-p'}dx\right)^{1/p'}<\infty
\end{multline}
holds uniformly for all $R>0$. Indeed, once \eqref{check1111} has been established, the inequality \eqref{high_Hardy1} implies that
\begin{equation}\label{K2_2}
K_{2}\leq C\left(\int_{\G}(\psi_{3}(x))^{(1-p')(1-p+p)} (h(x))^{p}dx\right)^{\frac{q-1}{p}}=C\left(\int_{\G}(h(x))^{p}(\psi_{3}(x))^{1-p'}dx\right)^{\frac{q-1}{p}}.
\end{equation}
To check \eqref{check1111}, denoting $S(s)=\int_{\wp}s^{Q-1}(\psi_{3}(sw))^{1-p'}d\sigma(w)$ and using the integration by parts we compute
\begin{equation*}
\begin{split}
A_{1}(R)&=
\left(\int_{\wp}\int_{R}^{\infty}r^{Q-1}(\psi_{3}(rw))^{1-p'}
\left(\int_{0}^{r}S(s)ds\right)^{-p}drd\sigma(w)\right)^{1/p}
\left(\int_{0}^{R}S(s)ds\right)^{1/p'}\\&
=\left(\int_{R}^{\infty}\left(\int_{0}^{r}S(s)ds\right)^{-p}S(r)dr\right)^{1/p}
\left(\int_{0}^{R}S(s)ds\right)^{1/p'}\\&
\leq \left(\frac{1}{p-1}\left(\int_{0}^{R}S(s)ds\right)^{1-p}\right)^{1/p}
\left(\int_{0}^{ R}S(s)ds\right)^{1/p'}=(p-1)^{-1/p}<\infty.
\end{split}
\end{equation*}
Next, for $K_{3}$, taking into account $\frac{1}{\delta}=\frac{1}{q}-\frac{1}{p}=\frac{p-q}{pq}$ and using \eqref{high2_Hardy3}, we have
\begin{equation}\label{K3_3}
\begin{split}
K_{3}&=
\left(\int_{0}^{\infty}\int_{\wp}\left(\int_{s}^{\infty}W_{2}(r)dr\right)^{\delta/q}
\left(\int_{\wp}\int_{0}^{r}r^{Q-1}(\psi_{3}(rw))^{1-p'}drd\sigma(w)\right)^{\delta/q'}\right.\\&
\left. \times s^{Q-1}(\psi_{3}(sy))^{1-p'}d\sigma(y) ds\right)^{\frac{p-q}{p}}\\&
=\left(\int_{\G}\left(\int_{\G\backslash B(0,|x|)}\phi_{3}(z)dz\right)^{\delta/q}\left(\int_{B(0,|x|)}(\psi_{3}(z))^{1-p'}dz\right)^{\delta/q'}
(\psi_{3}(x))^{1-p'}dx\right)^{\frac{p-q}{p}}\\&
=A_{3}^{\frac{p-q}{p}}<\infty.
\end{split}
\end{equation}
Now, plugging \eqref{K1}, \eqref{K2_2} and \eqref{K3_3} into \eqref{K123}, we obtain
$$\int_{\G}\phi_{3}(x)\left(\int_{B(0,|x|)}h(z)(\psi_{3}(z))^{1-p'}dz\right)^{q}dx\leq CA_{3}^{\frac{p-q}{p}}\left(\int_{\G}(h(x))^{p}(\psi_{3}(x))^{1-p'}dx\right)^{\frac{1}{p}+\frac{q-1}{p}},$$
which implies \eqref{high2_Hardy1} after the setting $h:=f\psi_{3}^{p'-1}$.

To show \eqref{high2_Hardy1}$\Rightarrow$\eqref{high2_Hardy3}, putting the functions
 $$f_{k}(x)=
\left(\int_{|y|\geq |x|}\phi_{3}(z)dz\right)^{\delta/(pq)}
\left(\int_{\alpha_{k}\leq |z|\leq |x|} (\psi_{3}(z))^{1-p'}dz\right)^{\delta/(pq')}$$
$$\times (\psi_{3}(x))^{1-p'} \chi_{(\alpha_{k},\beta_{k})}(|x|), \;k=1,2,\ldots,$$
instead of $f(x)$ in \eqref{high2_Hardy1}, we get \eqref{high2_Hardy3}, where $0<\alpha_{k}<\beta_{k}$ with $\alpha_{k}\searrow 0$ and $\beta_{k}\nearrow \infty$ for $k\rightarrow \infty$.
\end{proof}
We also note another version of weighted Hardy inequalities with the radial derivative.
\begin{thm}\label{weight_hardy_thm} Let $\mathbb{G}$ be a homogeneous group
of homogeneous dimension $Q$. Let $\phi_{5},\psi_{5}$ be positive weight functions on $\G$ and let $1<p\leq q<\infty$. Then there exists a positive constant $C$ such that
\begin{equation}\label{weight_hardy1}
\left(\int_{\G}\phi_{5}(x)|f(x)|^{q}dx\right)^{1/q}\leq C\left(\int_{\G}\psi_{5}(x)|\R_{|x|} f(x)|^{p}dx\right)^{1/p}
\end{equation}
holds for all radial functions $f$ with $f(0)=0$ if and only if
\begin{equation}\label{weight_hardy2}
A_{5}:=\sup_{R>0}\left(\int_{|x|\geq R}\phi_{5}(x)dx\right)^{1/q}\left(\int_{0}^{ R}\left(\int_{\wp}r^{Q-1}\psi_{5}(ry)d\sigma(y)\right)^{1-p^{\prime}}dr\right)^{1/p^{\prime}}<\infty,
\end{equation}
where $\mathcal{R}_{|x|}:=\frac{d}{d|x|}$ is the radial derivative.
\end{thm}
In the abelian case $\mathbb{G}=(\Rn,+)$ and $Q=n$, \eqref{weight_hardy1} was obtained in \cite{DHK97} and in \cite{Saw84}.
\begin{proof}[Proof of Theorem \ref{weight_hardy_thm}] If we denote $\tilde{f}(r)=f(x)$ for $r=|x|$ and
$$\Phi(r)=\int_{\wp}r^{Q-1}\phi_{5}(ry)d\sigma(y),\ \ \ \ \ \ \ \Psi(r)=\int_{\wp}r^{Q-1}\psi_{5}(ry)d\sigma(y),$$
then using $\tilde{f}(0)=0$ we have
\begin{equation*}
\begin{split}
&\left(\int_{\G}\phi_{5}(x)|f(x)|^{q}dx\right)^{1/q}
=\left(\int_{\wp}\int_{0}^{\infty}r^{Q-1}\phi_{5}(ry)|\tilde{f}(r)|^{q}drd\sigma(y)\right)^{1/q}\\&
=\left(\int_{0}^{\infty}\Phi(r)|\tilde{f}(r)|^{q}dr\right)^{1/q}=
\left(\int_{0}^{\infty}\Phi(r)\left|\int_{0}^{r}\R_{r}\tilde{f}(r)dr\right|^{q}dr\right)^{1/q}\\&
\leq C\left(\int_{0}^{\infty}\Psi(r)\left|\R_{r}\tilde{f}(r)\right|^{p}dr\right)^{1/p}
=C\left(\int_{\G}\psi_{5}(x)|\R_{|x|} f(x)|^{p}dx\right)^{1/p}
\end{split}
\end{equation*}
if and only if the condition \eqref{weight_hardy2} holds by Theorem \ref{high_Hardy_thm}, namely by \eqref{high_Hardy1} and \eqref{high_Hardy3}.
\end{proof}
Now we introduce another integral Hardy inequality.
\begin{thm}\label{Hardy_thm_new} Let $\mathbb{G}$ be a homogeneous Lie group of homogeneous dimension $Q$. Let $|\cdot|$ be an arbitrary homogeneous quasi-norm. Let $1<p\leq q<\infty$ and $0<a<Q/p$. Let $0\leq b<Q$ and $\frac{a}{Q}=\frac{1}{p}-\frac{1}{q}+\frac{b}{qQ}$. Assume that $|T^{(1)}_{a}(x)|\leq C_{2} |x|^{a-Q}$ for some positive $C_{2}=C_{2}(a,Q)$. Then there exists a positive constant $C_{1}=C_{1}(p,q,a,b)$ such that
\begin{equation}\label{Hardy_new1}
\left\|\frac{f\ast T^{(1)}_{a}}{|x|^{\frac{b}{q}}}\right\|_{L^{q}(\G)}\leq
C_{1}\|f\|_{L^{p}(\G)}
\end{equation}
holds for all $f\in L^{p}(\G)$.
\end{thm}
\begin{proof}[Proof of Theorem \ref{Hardy_thm_new}] We split the integral into three parts:
\begin{equation}\label{K123_new}
\int_{\G}|(f\ast T^{(1)}_{a})(x)|^{q}\frac{dx}{|x|^{b}}\leq 3^{q}(M_{1}+M_{2}+M_{3}),
\end{equation}
where
$$M_{1}:=\int_{\G}\left(\int_{\{2|y|<|x|\}}|T^{(1)}_{a}(y^{-1}x)f(y)|dy\right)^{q}
\frac{dx}{|x|^{b}},$$
$$M_{2}:=\int_{\G}\left(\int_{\{|x|\leq2|y|<4|x|\}}|T^{(1)}_{a}(y^{-1}x)f(y)|dy\right)^{q}\frac{dx}
{|x|^{b}}$$
and
$$M_{3}:=\int_{\G}\left(\int_{\{|y|>2|x|\}}|T^{(1)}_{a}(y^{-1}x)f(y)|dy\right)^{q}\frac{dx}
{|x|^{b}}.$$
First, let us estimate $M_{1}$. We can assume that $|\cdot|$ is a norm
without loss of generality because of the existence of a homogeneous norm (Proposition \ref{triangle_euc}) and since replacing the seminorm by an equivalent one only changes the appearing constants. Although we could give a proof without this hypothesis, it simplifies the arguments below. Then, by the reverse triangle inequality and $2|y|<|x|$ we have
\begin{equation}
\label{quasi_Euc_norm_new}
|y^{-1}x|\geq |x|-|y|>|x|-\frac{|x|}{2}=\frac{|x|}{2},
\end{equation}
which is $|x|<2|y^{-1}x|$. Taking into account this and that $T^{(1)}_{a}(x)$ is bounded by a radial function which is non-increasing with respect to $|x|$, we calculate
\begin{equation}\label{Log_HardyK1_1_new}
\begin{split}
M_{1}&\leq \int_{\G}\left(\int_{\{2|y|<|x|\}}|f(y)|dy\right)^{q}\left(\sup_{\{|x|<2|z|\}}|T^{(1)}_{a}(z)|\right)^{q}
\frac{dx}{|x|^{b}}\\&
\leq C\int_{\G}\left(\int_{\{2|y|<|x|\}}|f(y)|dy\right)^{q} \left(\frac{|x|}{2}\right)^{(a-Q)q}
\frac{dx}{|x|^{b}}.
\end{split}
\end{equation}
In order to apply \eqref{high_Hardy1} for $M_{1}$, let us check the condition \eqref{high_Hardy3}, that is, that
\begin{equation}\label{check1_new}
\left(\int_{\{2R<|x|\}}\left(\frac{|x|}{2}\right)^{(a-Q)q}
\frac{dx}{|x|^{b}}\right)^{\frac{1}{q}}
\left(\int_{\{|x|<R\}}dx\right)^{\frac{1}{p^{\prime}}}\leq A_{1}
\end{equation}
holds for all $R>0$. To check this, we consider two cases: $R\geq1$ and $0<R<1$. Then, we compute for $R\geq1$
\begin{equation}\label{check1_1_new}
\begin{split}\left(\int_{\{2R<|x|\}}\left(\frac{|x|}{2}\right)^{(a-Q)q}
\frac{dx}{|x|^{b}}\right)^{\frac{1}{q}}&
\left(\int_{\{|x|<R\}}dx\right)^{\frac{1}{p^{\prime}}}
\\&\leq C R^{\frac{Q}{p'}}
\left(\int_{\{2R<|x|\}}\left(\frac{|x|}{2}\right)^{(a-Q)q}
\frac{dx}{|x|^{b}}\right)^{\frac{1}{q}}\\&\leq CR^{\frac{Q}{p'}}\left(\int_{\{2R<|x|\}}
|x|^{(a-Q)q-b}dx\right)^{\frac{1}{q}}\\&\leq CR^{\frac{Q}{p'}}R^{\frac{(a-Q)q-b+Q}{q}}\\&\leq C,
\end{split}
\end{equation}
since $\frac{a}{Q}=\frac{1}{p}-\frac{1}{q}+\frac{b}{qQ}$ and $(a-Q)q-b+Q=-\frac{Qq}{p'}\neq 0$.
Now we check the condition \eqref{check1_new} for $0<R<1$. Here, taking into account $(a-Q)q-b+Q=-\frac{Qq}{p'}\neq 0$ we have
\begin{equation}\label{check2_3_01_new}
\int_{\{2R<|x|\}}\left(\frac{|x|}{2}\right)^{(a-Q)q}
\frac{dx}{|x|^{b}}\leq C R^{(a-Q)q-b+Q}.
\end{equation}
It follows with $\frac{a}{Q}=\frac{1}{p}-\frac{1}{q}+\frac{b}{qQ}$ that
\begin{equation}\label{check2_3_new}
\begin{split}\left(\int_{\{2R<|x|\}}\left(\frac{|x|}{2}\right)^{(a-Q)q}
\frac{dx}{|x|^{b}}\right)^{\frac{1}{q}}&
\left(\int_{\{|x|<R\}}dx\right)^{\frac{1}{p^{\prime}}}
\\&\leq CR^{a-Q-\frac{b}{q}+\frac{Q}{q}}R^{Q/p'}\leq C
\end{split}
\end{equation}
for any $0<R<1$. Thus, we have checked \eqref{check1_new}, then we can apply \eqref{high_Hardy1} for $M_{1}$ to obtain
\begin{equation}\label{Log_Hardy_K1_2_new}
M_{1}^{\frac{1}{q}}\leq(p^{\prime})^{\frac{1}{p^{\prime}}}p^{\frac{1}{q}}{A}_{1}\|f\|_{L^{p}(\G)}.
\end{equation}
Now let us estimate $M_{3}$. Without loss of generality, we may assume again $|\cdot|$ is the norm. Then, similarly to \eqref{quasi_Euc_norm_new} we note that $2|x|<|y|$ implies $|y|<2|y^{-1}x|$. Taking into account this we obtain for $M_{3}$ that
$$M_{3}\leq \int_{\G}\left(\int_{\{|y|>2|x|\}}\left(\frac{|y|}{2}\right)^{(a-Q)}|f(y)|dy\right)^{q}\frac{dx}
{|x|^{b}}.$$
To apply \eqref{high_Hardy2} for $M_{3}$, we check the following condition:
\begin{equation}\label{check2_new}
\left(\int_{\{|x|<R\}}\frac{dx}{|x|^{b}}\right)^{\frac{1}{q}}
\left(\int_{\{2R<|x|\}}\left(\frac{|x|}{2}\right)^{(a-Q)p^{\prime}}dx\right)^{\frac{1}{p^{\prime}}}\leq
A_{2}.
\end{equation}
To check this, we consider two cases: $R\geq1$ and $0<R<1$. Then, for $R\geq1$ using $|T^{(1)}_{a}(x)|\leq C |x|^{a-Q}$ and $Q\neq ap$, one gets
\begin{equation}\label{check2_1_new}
\left(\int_{\{2R<|x|\}}\left(\frac{|x|}{2}\right)^{(a-Q)p^{\prime}}dx\right)^{\frac{1}{p^{\prime}}}
\leq C\left(\int_{\{2R<|x|\}}|x|^{(a-Q)p^{\prime}}dx\right)^{\frac{1}{p^{\prime}}}\leq CR^{a-\frac{Q}{p}}.
\end{equation}
It follows for $R\geq1$ that
\begin{equation*}
\left(\int_{\{|x|<R\}}\frac{dx}{|x|^{b}}\right)^{\frac{1}{q}}
\left(\int_{\{2R<|x|\}}\left(\frac{|x|}{2}\right)^{(a-Q)p^{\prime}}dx\right)^{\frac{1}{p^{\prime}}}\leq
CR^{a-\frac{Q}{p}+\frac{Q-b}{q}}\leq C,
\end{equation*}
since $b< Q$ and $\frac{a}{Q}=\frac{1}{p}-\frac{1}{q}+\frac{b}{qQ}$. Now we check the condition \eqref{check2_new} for $0<R<1$. In this case, noting $ap-Q<0$ we have
\begin{equation}\label{check2_3_02_new}
\int_{\{2R<|x|\}}\left(\frac{|x|}{2}\right)^{(a-Q)p^{\prime}}dx
\leq C\int_{\{2R<|x|\}}|x|^{(a-Q)p'}dx\leq C R^{(a-Q)p'+Q}.
\end{equation}
Then, it gives with
$$\int_{\{ |x|<R\}}\frac{dx}{|x|^{b}}\leqslant CR^{Q-b}$$
that
\begin{equation}\label{check2_4_new}
\begin{split}
\left(\int_{\{|x|<R\}}\frac{dx}{|x|^{b}}\right)^{\frac{1}{q}}&\left(\int_{\{2R<|x|\}} \left(\frac{|x|}{2}\right)^{(a-Q)p^{\prime}}dx\right)^{\frac{1}{p^{\prime}}}\\&\leqslant CR^{\frac{Q-b}{q}}R^{\frac{(a-Q)p'+Q}{p'}}\\&
\leqslant C,
\end{split}
\end{equation}
since $Q>b$ and $\frac{a}{Q}=\frac{1}{p}-\frac{1}{q}+\frac{b}{qQ}$. Thus, we have checked \eqref{check2_new}, then we can apply \eqref{high_Hardy2} for $M_{3}$ to get
\begin{equation}\label{K3_new}
M_{3}^{\frac{1}{q}}\leq(p^{\prime})^{\frac{1}{p^{\prime}}}p^{\frac{1}{q}}{A}_{2}\|f\|_{L^{p}(\G)}.
\end{equation}
Finally, we estimate $M_{2}$. We write
$$
M_{2}=\sum_{k\in\mathbb{Z}}\int_{\{2^{k}\leqslant |x|<2^{k+1}\}}\left(\int_{\{|x|\leqslant 2|y|\leqslant 4|x|\}}|T^{(1)}_{a}(y^{-1}x)f(y)|dy\right)^{q}\frac{dx}{|x|^{b}}.
$$
Since $|x|\leqslant 2|y|\leqslant 4|x|$ and $2^{k}\leqslant |x|<2^{k+1}$, we have $2^{k-1}\leqslant |y|<2^{k+2}$.  As in \eqref{quasi_Euc_norm_new}, assuming $|\cdot|$ is the norm and using the triangle inequality, we have
\begin{equation}
\label{quasi_Euc_norm2_new} 3|x|=|x|+2|x|\geq |x|+|y|\geq |y^{-1}x|,
\end{equation}
which implies $0\leq |y^{-1}x|\leq3|x|<3\cdot 2^{k+1}$. If we denote $\widetilde{I_{a}}(x):=C_{2}|x|^{a-Q}$, then $|T^{(1)}_{a}(x)|\leq \widetilde{I_{a}}(x)$. Taking into account these, applying Young's inequality (well-known, see e.g. \cite[Proposition 1.5.2]{FR16}) for $1+\frac{1}{q}=\frac{1}{r}+\frac{1}{p}$ with $r\in [1,\infty]$ we estimate $M_{2}$ by
\begin{equation}\label{K2_new}
\begin{split}
M_{2}&\leq \sum_{k\in\mathbb{Z}}2^{-kb}\int_{\G}(([f\cdot \chi_{\{2^{k-1}\leqslant |\cdot|<2^{k+2}\}}]\ast \widetilde{I}_{a})(x))^{q}dx\\&
= \sum_{k\in\mathbb{Z}}2^{-kb}\|[f\cdot \chi_{\{2^{k-1}\leqslant |\cdot|<2^{k+2}\}}]\ast \widetilde{I}_{a}\|^{q}_{L^{q}(\G)}\\&
\leq \sum_{k\in\mathbb{Z}}2^{-kb}\|\widetilde{I}_{a}\cdot \chi_{\{0\leqslant |\cdot|<3\cdot2^{k+1}\}}\|^{q}_{L^{r}(\G)}\|f\cdot \chi_{\{2^{k-1}\leqslant |\cdot|<2^{k+2}\}}\|^{q}_{L^{p}(\G)}\\&
= C_{2}\sum_{k\in\mathbb{Z}}2^{-kb}\left(\int_{|x|<3\cdot2^{k+1}}|x|^{(a-Q)r}dx\right)^{\frac{q}{r}}\|f\cdot \chi_{\{2^{k-1}\leqslant |x|<2^{k+2}\}}\|^{q}_{L^{p}(\G)}\\&
\leq C\sum_{k\in\mathbb{Z}}2^{-kb}(3\cdot2^{k+1})^
{\left(\frac{(a-Q)pq}{pq+p-q}+Q\right)\frac{pq+p-q}{p}}\|f\cdot \chi_{\{2^{k-1}\leqslant |x|<2^{k+2}\}}\|^{q}_{L^{p}(\G)}\\&
=C\sum_{k\in\mathbb{Z}}2^{-kb}(3\cdot2^{k+1})^
{b}\|f\cdot \chi_{\{2^{k-1}\leqslant |x|<2^{k+2}\}}\|^{q}_{L^{p}(\G)}\\&
\leq C\sum_{k\in\mathbb{Z}}\|f\cdot \chi_{\{2^{k-1}\leqslant |x|<2^{k+2}\}}\|^{q}_{L^{p}(\G)}\\&
\leq C\|f\|^{q}_{L^{p}(\G)},
\end{split}
\end{equation}
since $\frac{(a-Q)pq}{pq+p-q}+Q=\frac{bp}{pq+p-q}>0$ and $q\geq p$.

Thus, \eqref{Log_Hardy_K1_2_new}, \eqref{K3_new} and \eqref{K2_new} complete the proof of Theorem \ref{Hardy_thm_new}.
\end{proof}
\begin{rem}\label{Schurtest_proof_hom}  In the case $p=q$, we can also prove Theorem \ref{Hardy_thm_new} by using Schur's test \cite{FR74_Schur}. Since $p=q$, we have $b=ap$ from $\frac{a}{Q}=\frac{1}{p}-\frac{1}{q}+\frac{b}{qQ}$. Let $S_{a}f:=|x|^{-b/p}(f\ast |x|^{a-Q})$, then $S^{*}_{a}g:=(|x|^{-b/p}g)\ast |x|^{a-Q}$, where $(f,S^{*}_{a}g)=(S_{a}f,g)$. Since the integral kernel of $S_{a}$ is positive, by Schur's test we see that instead of proving the estimate
$$\|S_{a}f\|_{L^{p}(\G)}\leq A_{a,p}^{1/p'}B_{a,p}^{1/p}\|f\|_{L^{p}(\G)}$$
for all $f\in L^{p}(\G)$, it is enough to exhibit a positive function $h$ and constants $A_{a,p}$ and $B_{a,p}$ such that
$$S_{a}(h^{p'})(x)\leq A_{a,p} (h(x))^{p'}\;\;\text{and}\;\;S^{*}_{a}(h^{p})(x)\leq B_{a,p} (h(x))^{p}$$
for almost all $x\in \G$.

Let us take $h_{c}(x):=|x|^{c-Q}$ with $c>0$ and consider the convolution integrals
$$h_{c}^{p'}\ast |x|^{a-Q}\;\;\text{and}\;\;(|x|^{-b/p}h_{c}^{p})\ast |x|^{a-Q},$$
which arise in the computation of $S_{a}(h_{c}^{p'})$ and $S^{*}_{a}(h_{c}^{p})$. We see that the homogeneity orders of $h_{c}^{p'}$ and $|x|^{-b/p}h_{c}^{p}$ are $(c-Q)p'$ and $(c-Q)p-b/p$, respectively. Then, the homogeneity of $h_{c}^{p'}\ast |x|^{a-Q}$ and $(|x|^{-b/p}h_{c}^{p})\ast |x|^{a-Q}$ are $a-Q+(c-Q)p'$ and $a-Q+(c-Q)p-b/p$, respectively. Therefore, these convolution integrals converge absolutely in $\G\backslash\{0\}$ if and only if $0<(c-Q)p'+Q<Q-a$ and $0<(c-Q)p-b/p+Q<Q-a$, that is,
$$\max\left(\frac{Q}{p},\frac{a}{p}+\frac{Q}{p'}\right)<c<Q-\frac{a}{p'}$$
since $b=ap$. This condition is true if $0<a<Q/p$.

Thus, we have obtained
$$\||x|^{-b/p}(f\ast |x|^{a-Q})\|_{L^{p}(\G)}\leq A_{a,p}^{1/p'}B_{a,p}^{1/p}\|f\|_{L^{p}(\G)},$$
where $0<a<Q/p$, $1<p<\infty$, $f\in L^{p}(\G)$ and $b=ap$.

Taking into account this and $|T^{(1)}_{a}(x)|\leq C |x|^{a-Q}$, we obtain
\begin{multline}\label{Schur1_hom}\left\|\frac{f\ast T^{(1)}_{a}}{|x|^{\frac{b}{p}}}\right\|_{L^{p}(\G)}\leq C
\left\|\frac{|f|\ast |T^{(1)}_{a}|}{|x|^{\frac{b}{p}}}\right\|_{L^{p}(\G)}\\ \leq
C \||x|^{-b/p}(|f|\ast |x|^{a-Q})\|_{L^{p}(\G)} \leq C \|f\|_{L^{p}(\G)}.
\end{multline}
\end{rem}

Now we also show the critical case $b=Q$ of Theorem \ref{Hardy_thm_new}.
\begin{thm}\label{Log_Hardy_thm} Let $\mathbb{G}$ be a homogeneous Lie group of homogeneous dimension $Q$. Let $|\cdot|$ be an arbitrary homogeneous quasi-norm and let $1<p<r<\infty$ and $p<q<(r-1)p'$, where $1/p+1/p'=1$. Assume that for $a=Q/p$ we have
\begin{equation}\label{T2_B}
|T^{(2)}_{a}(x)|\leq
C_{2}\begin{cases} |x|^{a-Q}, \text{\;for}\;x\in\mathbb{G}\backslash\{0\},\\
|x|^{-Q},   \text{\;for}\;x\in\mathbb{G}\; \text{\;with}\;|x|\geq1,\end{cases}
\end{equation} for some positive $C_{2}=C_{2}(a,Q)$. Then there exists a positive constant $C_{1}=C_{1}(p, q, r, Q)$ such that
\begin{equation}\label{Hardy_wholeG2_log}
\left\|\frac{f\ast T^{(2)}_{Q/p}}{\left(\log\left(e+\frac{1}{|x|}\right)\right)^{\frac{r}{q}}|x|^{\frac{Q}{q}}}\right\|_{L^{q}(\G)}\leq
C_{1}\|f\|_{L^{p}(\G)}
\end{equation}
holds for all $f\in L^{p}(\G)$.
\end{thm}
\begin{proof}[Proof of Theorem \ref{Log_Hardy_thm}] Let us split the integral into three parts
\begin{equation}\label{N123}
\int_{\G}|(f\ast T^{(2)}_{Q/p})(x)|^{q}\frac{dx}{\left|\log\left(e+\frac{1}{|x|}\right)\right|^{r}|x|^{Q}}\leq 3^{q}(N_{1}+N_{2}+N_{3}),
\end{equation}
where
$$N_{1}:=\int_{\G}\left(\int_{\{2|y|<|x|\}}| T^{(2)}_{Q/p}(y^{-1}x)f(y)|dy\right)^{q}
\frac{dx}{\left|\log\left(e+\frac{1}{|x|}\right)\right|^{r}|x|^{Q}},$$
$$N_{2}:=\int_{\G}\left(\int_{\{|x|\leq2|y|<4|x|\}}| T^{(2)}_{Q/p}(y^{-1}x)f(y)|dy\right)^{q}\frac{dx}
{\left|\log\left(e+\frac{1}{|x|}\right)\right|^{r}|x|^{Q}}$$
and
$$N_{3}:=\int_{\G}\left(\int_{\{|y|>2|x|\}}| T^{(2)}_{Q/p}(y^{-1}x)f(y)|dy\right)^{q}\frac{dx}
{\left|\log\left(e+\frac{1}{|x|}\right)\right|^{r}|x|^{Q}}.$$
First, let us estimate $N_{1}$. Similar to \eqref{quasi_Euc_norm_new} from $2|y|<|x|$ we get
\begin{equation}
\label{quasi_Euc_norm}
|y^{-1}x|\geq |x|-|y|>|x|-\frac{|x|}{2}=\frac{|x|}{2},
\end{equation}
which is $|x|<2|y^{-1}x|$. Denote
\begin{equation}\label{T2_B_denote}
|T^{(2)}_{a}(x)|\leq \widetilde{B}_{a}(x):=
C_{2}\begin{cases} |x|^{a-Q}, \text{\;for}\;x\in\mathbb{G}\backslash\{0\},\\
|x|^{-Q},   \text{\;for}\;x\in\mathbb{G}\; \text{\;with}\;|x|\geq1.\end{cases}
\end{equation} Since $ T^{(2)}_{Q/p}(x)$ is bounded by $\widetilde{B}_{Q/p}(x)$ which is non-increasing with respect to $|x|$, then using \eqref{quasi_Euc_norm} we get
\begin{equation}\label{Log_HardyK1_1}
\begin{split}
N_{1}&\leq \int_{\G}\left(\int_{\{2|y|<|x|\}}|f(y)|dy\right)^{q}\left(\sup_{\{|x|<2|z|\}}| T^{(2)}_{Q/p}(z)|\right)^{q}
\frac{dx}{\left|\log\left(e+\frac{1}{|x|}\right)\right|^{r}|x|^{Q}}\\&
\leq \int_{\G}\left(\int_{\{2|y|<|x|\}}|f(y)|dy\right)^{q} \left(\widetilde{B}_{Q/p}\left(\frac{x}{2}\right)\right)^{q}
\frac{dx}{\left|\log\left(e+\frac{1}{|x|}\right)\right|^{r}|x|^{Q}}.
\end{split}
\end{equation}
To apply \eqref{high_Hardy1} for $N_{1}$, we need to check the condition \eqref{high_Hardy3}, that is, that
\begin{equation}\label{check1}
\left(\int_{\{2R<|x|\}}\left(\widetilde{B}_{Q/p}\left(\frac{x}{2}\right)\right)^{q}
\frac{dx}{\left|\log\left(e+\frac{1}{|x|}\right)\right|^{r}|x|^{Q}}\right)^{\frac{1}{q}}
\left(\int_{\{|x|<R\}}dx\right)^{\frac{1}{p^{\prime}}}\leq A_{1}
\end{equation}
holds for all $R>0$. In order to check this, let us consider two cases: $R\geq1$ and $0<R<1$. Then, for $R\geq1$ using the second equality in \eqref{T2_B_denote}, one calculates
\begin{equation}\label{check1_1}
\begin{split}\left(\int_{\{2R<|x|\}}\left(\widetilde{B}_{Q/p}\left(\frac{x}{2}\right)\right)^{q}
\right.&\left.\frac{dx}{\left|\log\left(e+\frac{1}{|x|}\right)\right|^{r}|x|^{Q}}\right)^{\frac{1}{q}}
\left(\int_{\{|x|<R\}}dx\right)^{\frac{1}{p^{\prime}}}
\\&\leq C R^{\frac{Q}{p'}}
\left(\int_{\{2R<|x|\}}\left(\widetilde{B}_{Q/p}\left(\frac{x}{2}\right)\right)^{q}
\frac{dx}{|x|^{Q}}\right)^{\frac{1}{q}}\\&=CR^{\frac{Q}{p'}}\left(\int_{\{2R<|x|\}}
|x|^{-Qq-Q}dx\right)^{\frac{1}{q}}\\&\leq CR^{-Q}R^{\frac{Q}{p'}}\\&\leq C.
\end{split}
\end{equation}
Now let us check \eqref{check1} for $0<R<1$. We write
\begin{multline}\label{check2_3_01}
\int_{\{2R<|x|\}}\left(\widetilde{B}_{Q/p}\left(\frac{x}{2}\right)\right)^{q}
\frac{dx}{\left|\log\left(e+\frac{1}{|x|}\right)\right|^{r}|x|^{Q}}
\\=
\int_{\{2R<|x|<2\}}\left( \widetilde{B}_{Q/p}\left(\frac{x}{2}\right)\right)^{q}\frac{dx}{\left|\log\left(e+\frac{1}{|x|}\right)\right|^{r}|x|^{Q}}
\\+\int_{\{|x|\geqslant 2\}}\left( \widetilde{B}_{Q/p}\left(\frac{x}{2}\right)\right)^{q}\frac{dx}{\left|\log\left(e+\frac{1}{|x|}\right)\right|^{r}|x|^{Q}}.
\end{multline}
We note that the second integral in the right hand side of \eqref{check2_3_01} is integrable by the second equality in \eqref{T2_B_denote}. Then, using the first equality in \eqref{T2_B_denote} we get for the first integral that
\begin{equation*}
\begin{split}
\int_{\{2R<|x|<2\}}\left| \widetilde{B}_{Q/p}\left(\frac{x}{2}\right)\right|^{q}&
\frac{dx}{\left|\log\left(e+\frac{1}{|x|}\right)\right|^{r}|x|^{Q}}\\&
\leq \int_{\{2R<|x|<2\}}\left| \widetilde{B}_{Q/p}\left(\frac{x}{2}\right)\right|^{q}
\frac{dx}{|x|^{Q}}\\&\leq C
\int_{\{2R<|x|<2\}}|x|^{-Qq/p'-Q}dx\\&
\leq C R^{-Qq/p'}.
\end{split}
\end{equation*}
It implies with \eqref{check2_3_01} that
\begin{equation*}
\begin{split}\left(\int_{\{2R<|x|\}}\left| \widetilde{B}_{Q/p}\left(\frac{x}{2}\right)\right|^{q}
\frac{dx}{\left|\log\left(e+\frac{1}{|x|}\right)\right|^{r}|x|^{Q}}\right)^{\frac{1}{q}}&
\left(\int_{\{|x|<R\}}dx\right)^{\frac{1}{p^{\prime}}}
\\&\leq C(R^{-Q/p'}+1)R^{Q/p'}\leq C
\end{split}
\end{equation*}
for any $0<R<1$. Thus, we have checked \eqref{check1}, then applying \eqref{high_Hardy1} for $N_{1}$ one gets
\begin{equation}\label{Log_Hardy_K1_2}
N_{1}^{\frac{1}{q}}\leq(p^{\prime})^{\frac{1}{p^{\prime}}}p^{\frac{1}{q}}{A}_{1}\|f\|_{L^{p}(\G)}.
\end{equation}
Now we estimate $N_{3}$. Without loss of generality, we may assume again that $|\cdot|$ is the norm. Similarly to \eqref{quasi_Euc_norm} we obtain $|y|<2|y^{-1}x|$ from $2|x|<|y|$. Then, we have for $N_{3}$ that
$$N_{3}\leq \int_{\G}\left(\int_{\{|y|>2|x|\}}\left| \widetilde{B}_{Q/p}\left(\frac{y}{2}\right)\right||f(y)|dy\right)^{q}\frac{dx}
{\left|\log\left(e+\frac{1}{|x|}\right)\right|^{r}|x|^{Q}}.$$
In order to apply \eqref{high_Hardy2} for $N_{3}$, we need to check the following condition:
\begin{equation}\label{check2}
\left(\int_{\{|x|<R\}}\frac{dx}{\left|\log\left(e+\frac{1}{|x|}\right)\right|^{r}|x|^{Q}}\right)^{\frac{1}{q}}
\left(\int_{\{2R<|x|\}}\left| \widetilde{B}_{Q/p}\left(\frac{x}{2}\right)\right|^{p^{\prime}}dx\right)^{\frac{1}{p^{\prime}}}\leq
A_{2}.
\end{equation}
To check this, let us consider the cases: $R\geq1$ and $0<R<1$. Then, for $R\geq1$ by the second equality in \eqref{T2_B_denote}, we get
\begin{equation}\label{check2_1}
\left(\int_{\{2R<|x|\}}\left| \widetilde{B}_{Q/p}\left(\frac{x}{2}\right)\right|^{p^{\prime}}dx\right)^{\frac{1}{p^{\prime}}}
\leq C\left(\int_{\{2R<|x|\}}|x|^{-Qp^{\prime}}dx\right)^{\frac{1}{p^{\prime}}}\leq CR^{-\frac{Q}{p}}.
\end{equation}
Moreover, we have
\begin{multline*}
\int_{\{|x|<R\}}\frac{dx}{\left|\log\left(e+\frac{1}{|x|}\right)\right|^{r}|x|^{Q}}
=\int_{\{|x|<\frac{1}{2}\}}\frac{dx}{\left|\log\left(e+\frac{1}{|x|}\right)\right|^{r}|x|^{Q}}\\+
\int_{\left\{\frac{1}{2}\leqslant |x|<R\right\}}\frac{dx}{\left|\log\left(e+\frac{1}{|x|}\right)\right|^{r}|x|^{Q}},
\end{multline*}
and we note that the first summand in the right hand side of above is integrable since $r>1$. For the second term, we get
\begin{equation}\label{check2_2}
\int_{\left\{\frac{1}{2}\leq |x|<R\right\}}\frac{dx}{\left|\log\left(e+\frac{1}{|x|}\right)\right|^{r}|x|^{Q}}
\leq \int_{\left\{\frac{1}{2}\leq |x|<R \right\}}\frac{dx}{|x|^{Q}}\leq C(1+\log R).
\end{equation}
Combining \eqref{check2_1} and \eqref{check2_2}, we have for $R\geq1$ that
\begin{multline*}
\left(\int_{\{|x|<R\}}\frac{dx}{\left|\log\left(e+\frac{1}{|x|}\right)\right|^{r}|x|^{Q}}\right)^{\frac{1}{q}}
\left(\int_{\{2R<|x|\}}\left| \widetilde{B}_{Q/p}\left(\frac{x}{2}\right)\right|^{p^{\prime}}dx\right)^{\frac{1}{p^{\prime}}}\\ \leq
CR^{-\frac{Q}{p}}(1+\log R)^{\frac{1}{q}}\leq C.
\end{multline*}
Now let us check the condition \eqref{check2} for $0<R<1$. We split the integral
\begin{equation}\label{check2_3_02}
\int_{\{2R<|x|\}}\left| \widetilde{B}_{Q/p}\left(\frac{x}{2}\right)\right|^{p^{\prime}}dx=
\int_{\{2R<|x|<2\}}\left| \widetilde{B}_{Q/p}\left(\frac{x}{2}\right)\right|^{p^{\prime}}dx+\int_{\{|x|\geqslant 2\}}\left| \widetilde{B}_{Q/p}\left(\frac{x}{2}\right)\right|^{p^{\prime}}dx.
\end{equation}
We note that the second integral in the right hand side of above is integrable by the second equality in \eqref{T2_B_denote}. Then, using the first equality in \eqref{T2_B_denote} we get for the first integral that
\begin{equation*}
\int_{\{2R<|x|<2\}}\left| \widetilde{B}_{Q/p}\left(\frac{x}{2}\right)\right|^{p'}dx\leq C
\int_{\{2R<|x|<2\}}|x|^{-Q}dx
\leq C \log\left(\frac{1}{R}\right),
\end{equation*}
which implies with \eqref{check2_3_02} that
\begin{equation}\label{check2_3}
\int_{\{2R<|x|\}}\left| \widetilde{B}_{Q/p}\left(\frac{x}{2}\right)\right|^{p^{\prime}}dx\leq C\left(1+\log\left(\frac{1}{R}\right)\right).
\end{equation}
Since
$$\int_{\{ |x|<R\}}\frac{dx}{\left|\log\left(e+\frac{1}{|x|}\right)\right|^{r}|x|^{Q}}\leqslant C\left(\log\left(e+\frac{1}{R}\right)\right)^{-(r-1)},$$
and \eqref{check2_3}, and taking into account $r>1$ and $q<(r-1)p^{\prime}$ we obtain that
\begin{equation}\label{check2_4}
\begin{split}
\left(\int_{\{|x|<R\}}\frac{dx}{\left|\log\left(e+\frac{1}{|x|}\right)\right|^{r}|x|^{Q}}\right)^{\frac{1}{q}}&\left(\int_{\{2R<|x|\}} \left| \widetilde{B}_{Q/p}\left(\frac{x}{2}\right)\right|^{p^{\prime}}dx\right)^{\frac{1}{p^{\prime}}}\\&\leqslant C\left(\log \left(e+\frac{1}{R}\right)\right)^{-\frac{r-1}{q}}\left(1+\left(\log \left(\frac{1}{R}\right)\right)^{\frac{1}{p^{\prime}}}\right)\\&
\leqslant C.
\end{split}
\end{equation}
Thus, we have checked \eqref{check2}, then applying \eqref{high_Hardy2} for $N_{3}$ we obtain
\begin{equation}\label{N3}
N_{3}^{\frac{1}{q}}\leq(p^{\prime})^{\frac{1}{p^{\prime}}}p^{\frac{1}{q}}{A}_{2}\|f\|_{L^{p}(\G)}.
\end{equation}
Now let us estimate $N_{2}$. We write
$$N_{2}=\sum_{k\in\mathbb{Z}}\int_{\{2^{k}\leqslant |x|<2^{k+1}\}}\left(\int_{\{|x|\leqslant 2|y|\leqslant 4|x|\}}| T^{(2)}_{Q/p}(y^{-1}x)f(y)|dy\right)^{q}\frac{dx}{\left|\log\left(e+\frac{1}{|x|}\right)\right|^{r}|x|^{Q}}.$$
Since the function $\left(\log\left(\frac{1}{|x|}\right)\right)^{r}|x|^{Q}$ is non-decreasing with respect to $|x|$ near the origin, there exists an integer $k_{0}\in\mathbb{Z}$ with $k_{0}\leqslant -3$ such that this function is non-decreasing in $|x|\in(0,2^{k_{0}+1})$. We decompose $N_{2}$ with $k_{0}$ as follows
\begin{equation}\label{N2}
N_{2}=N_{21}+N_{22},
\end{equation}
where
$$N_{21}:=\sum_{k=-\infty}^{k_{0}}\int_{\{2^{k}\leqslant |x|<2^{k+1}\}}\left(\int_{\{|x|\leqslant 2|y|\leqslant 4|x|\}}| T^{(2)}_{Q/p}(y^{-1}x)f(y)|dy\right)^{q}\frac{dx}{\left|\log\left(e+\frac{1}{|x|}\right)\right|^{r}|x|^{Q}}$$
and
$$N_{22}:=\sum_{k=k_{0}+1}^{\infty}\int_{\{2^{k}\leqslant |x|<2^{k+1}\}}\left(\int_{\{|x|\leqslant 2|y|\leqslant 4|x|\}}| T^{(2)}_{Q/p}(y^{-1}x)f(y)|dy\right)^{q}\frac{dx}{\left|\log\left(e+\frac{1}{|x|}\right)\right|^{r}|x|^{Q}}.$$
Let us first estimate $N_{22}$. Since $|x|\leqslant 2|y|\leqslant 4|x|$ and $2^{k}\leqslant |x|<2^{k+1}$, we have $2^{k-1}\leqslant |y|<2^{k+2}$. Before starting to estimate $N_{22}$, using \eqref{T2_B} and $q>p$, let us show that
\begin{equation}\label{Bes_est}
\begin{split}
\int_{\G}| T^{(2)}_{Q/p}(x)|^{\tilde{r}}dx&=\int_{|x|<1}| T^{(2)}_{Q/p}(x)|^{\tilde{r}}dx+
\int_{|x|\geq1}| T^{(2)}_{Q/p}(x)|^{\tilde{r}}dx\\&\leq C_{2}\left(
\int_{|x|<1}|x|^{-\frac{Qq(p-1)}{pq+p-q}}dx+
\int_{|x|\geq1}|x|^{-\frac{Qpq}{pq+p-q}}dx\right)<\infty,
\end{split}
\end{equation}
where $\tilde{r}\in [1,\infty]$ is such that $1+\frac{1}{q}=\frac{1}{\tilde{r}}+\frac{1}{p}$.

Then, \eqref{Bes_est} and Young's inequality (e.g. \cite[Proposition 1.5.2]{FR16}) for $1+\frac{1}{q}=\frac{1}{\tilde{r}}+\frac{1}{p}$ with $\tilde{r}\in [1,\infty]$ imply that
\begin{equation}\label{N22}
\begin{split}
N_{22}&\leqslant C \sum_{k=k_{0}+1}^{\infty}\int_{\{2^{k}\leqslant |x|<2^{k+1}\}}\left(\int_{\{|x|\leqslant 2|y|\leqslant 4|x|\}}| T^{(2)}_{Q/p}(y^{-1}x)f(y)|dy\right)^{q}dx\\&
\leqslant C \|[f\cdot\chi_{\{2^{k-1}\leqslant |\cdot|<2^{k+2}\}}]\ast  T^{(2)}_{Q/p}\|^{q}_{L^{q}(\G)}\\& \leqslant C\| T^{(2)}_{Q/p}\| ^{q}_{L^{\tilde{r}}(\G)}\sum_{k=k_{0}+1}^{\infty}\|f\cdot\chi_{\{2^{k-1}\leqslant |\cdot|<2^{k+2}\}}\|^{q}_{L^{p}(\G)}\\&
=C \sum_{k=k_{0}+1}^{\infty}\left(\int_{\{2^{k}\leqslant |x|<2^{k+1}\}}|f(x)|^{p}dx\right)^{\frac{q}{p}}\\& \leqslant C\left(\sum_{k\in\mathbb{Z}}\int_{\{2^{k}\leqslant |x|<2^{k+1}\}}|f(x)|^{p}dx\right)^{\frac{q}{p}}\\&
=C\|f \|^{q}_{L^{p}(\G)}.
\end{split}
\end{equation}

To complete the proof it is left to estimate $N_{21}$. As in \eqref{quasi_Euc_norm}, assuming $|\cdot|$ is the norm and using the triangle inequality, we have
\begin{equation}
\label{quasi_Euc_norm2}3|x|=|x|+2|x|\geq |x|+|y|\geq |y^{-1}x|,
\end{equation}
where we have used $|y|\leqslant 2|x|$. Since $\left(\log\left(\frac{1}{|x|}\right)\right)^{r}|x|^{Q}$ is non-decreasing in $|x| \in (0,2^{k_{0}+1})$ and $3|x|\geqslant |y^{-1}x|$, we have $$\left(\log\left(\frac{1}{|x|}\right)\right)^{r}|x|^{Q}\geq \left(\log\left(\frac{1}{\left|\frac{y^{-1}x}{3}\right|}\right)\right)^{r}\left|\frac{y^{-1}x}{3}\right|^{Q}.$$
Then, these and \eqref{T2_B} yield
$$N_{21}\leq C\sum_{k=-\infty}^{k_{0}}\int_{\{2^{k}\leq|x|<2^{k+1}\}}\left(\int_{\{|x|\leq2|y|\leq4|x|\}}
|y^{-1}x|^{-\frac{Q}{p^{\prime}}}|f(y)|dy\right)^{q}\frac{dx}{\left(\log\left(\frac{1}{|x|}\right)\right)^{r}|x|^{Q}}$$
$$= C\sum_{k=-\infty}^{k_{0}}\int_{\{2^{k}\leq|x|<2^{k+1}\}}\left(\int_{\{|x|\leq2|y|\leq4|x|\}}\frac{|y^{-1}x|
^{-\frac{Q}{p^{\prime}}}|f(y)|}{\left(\left(\log\left(\frac{1}{|x|}\right)\right)^{r}|x|^{Q}\right)^{\frac{1}{q}}}dy\right)^{q}dx$$
$$\leq C\sum_{k=-\infty}^{k_{0}}\int_{\{2^{k}\leq|x|<2^{k+1}\}}\left(\int_{\{|x|\leq2|y|\leq4|x|\}}\frac{|y^{-1}x|
^{-\frac{Q}{p^{\prime}}}|f(y)|}{\left(\left(\log\left(\frac{1}{|(y^{-1}x)/3|}\right)\right)^{r}|(y^{-1}x)/3|^{Q}\right)^{\frac{1}{q}}}dy\right)^{q}dx.$$
Since $|x|\leq2|y|\leq4|x|$ and $2^{k}\leq|x|<2^{k+1}$ with $k\leq k_{0}$, we get $2^{k-1}\leq|y|<2^{k+2}$ and $|y^{-1}x|\leq3|x|<3\cdot 2^{k_{0}+1}\leq3/4$ by \eqref{quasi_Euc_norm2} and $k_{0}\leq-3$. Taking into account these and setting $$g(x):=\frac{\chi_{B_{\frac{3}{4}}(0)}(x)}{\left(\log\left(\frac{1}{|x|}\right)\right)^{\frac{r}{q}}|x|^{\frac{Q}{q}+\frac{Q}{p'}}},$$
we have for $N_{21}$ that
$$N_{21}\leq C\sum_{k=-\infty}^{k_{0}}\int_{\{2^{k}\leq|x|<2^{k+1}\}}\left(\int_{\{|x|\leq2|y|\leq4|x|\}}\frac{|f(y)|}{
{\left(\log\left(\frac{1}{|y^{-1}x|}\right)\right)^{\frac{r}{q}}|y^{-1}x|^{\frac{Q}{q}+\frac{Q}{p'}}}}dy\right)^{q}dx$$
$$\leq C\sum_{k=-\infty}^{k_{0}}\|[f\cdot\chi_{\{2^{k-1}\leq |\cdot|<2^{k+2}\}}]\ast g\|^{q}_{L^{q}(\G)}.$$
Since $p<q<(r-1)p'$, we use Young's inequality for $1+\frac{1}{q}=\frac{1}{\tilde{r}}+\frac{1}{p}$ with $\tilde{r}\in[1,\infty)$ to get
\begin{equation}\label{N21}
N_{21}\leq C \|g\|^{q}_{L^{\tilde{r}}(\G)}\sum_{k=-\infty}^{k_{0}}\|f\cdot\chi_{\{2^{k-1}\leq|\cdot|<2^{k+2}\}}\|^{q}_{L^{p}(\G)}\leq C\|f\|^{q}_{L^{p}(\G)},
\end{equation}
provided that $g\in L^{\tilde{r}}(\G)$. Since $\left(\frac{Q}{q}+\frac{Q}{p'}\right)\tilde{r}=Q$, $\frac{r\tilde{r}}{q}=\frac{rp'}{p'+q}$ and $q<(r-1)p'$, then changing variables, we obtain
$$\|g\|^{\tilde{r}}_{L^{\tilde{r}}(\G)}=
\int_{B(0,3/4)}\frac{dx}{\left(\log\left(\frac{1}{x}\right)\right)^{\frac{rp'}{p'+q}}|x|^{Q}}
=C\int_{\log\left(\frac{4}{3}\right)}^{\infty}\frac{dt}{t^{\frac{rp'}{p'+q}}}<\infty.$$
Thus, \eqref{Log_Hardy_K1_2}, \eqref{N3}, \eqref{N2}, \eqref{N22}, \eqref{N21} and \eqref{N123} complete the proof of Theorem \ref{Log_Hardy_thm}.
\end{proof}

\section{Hardy-Littlewood-Sobolev inequalities on homogeneous groups}
\label{SEC:HLS}

In this section we apply the integral Hardy inequality from the previous section to obtain the Hardy-Littlewood-Sobolev inequality on homogeneous groups. We also discuss the reversed Hardy-Littlewood-Sobolev inequalities on general homogeneous groups.

Now we start with the Hardy-Littlewood-Sobolev inequality (see \cite{HL28}, \cite{HL30} and \cite{Sob38}). We also refer to \cite{FS74} for the case of the Heisenberg group and to \cite{Lie83} and \cite{FL12} for sharp constants of the Hardy-Littlewood-Sobolev inequality. Here, we investigate the weighted Hardy-Littlewood-Sobolev inequality on general homogeneous groups.

\begin{thm}\label{HLS_thm} Let $\mathbb{G}$ be a homogeneous Lie group of homogeneous dimension $Q$ and let $|\cdot|$ be an arbitrary homogeneous quasi-norm. Let $0<\lambda<Q$ and $1<p,q<\infty$ be such that $1/p+1/q+(\alpha+\lambda)/Q=2$ with $0\leq \alpha <Q/p'$ and $\alpha+\lambda\leq Q$, where $1/p+1/p'=1$. Then there exists a positive constant $C=C(Q,\lambda, p, \alpha)$ such that
\begin{equation}\label{HLS_ineq1}
\left|\int_{\G}\int_{\G}\frac{\overline{f(x)}g(y)}{|x|^{\alpha}|y^{-1}x|^{\lambda}}dxdy\right|\leq C\|f\|_{L^{p}(\G)}\|g\|_{L^{q}(\G)}
\end{equation}
holds for all $f\in L^{p}(\G)$ and $g\in L^{q}(\G)$.
\end{thm}
\begin{proof}[Proof of Theorem \ref{HLS_thm}] Let $T^{(3)}_{a}(x):=|x|^{a-Q}$ with $0<a<Q/r$ for some $1<r<\infty$. Then, using H\"{o}lder's inequality we calculate
\begin{equation}\label{HLS_ineq2}
\begin{split}
\left|\int_{\G}\int_{\G}\frac{\overline{f(x)}g(y)}{|x|^{\alpha}|y^{-1}x|^{\lambda}}dxdy\right|
&=\left|\int_{\G}\overline{f(x)}\frac{(g\ast T^{(3)}_{Q-\lambda})(x)}{|x|^{\alpha}}dx\right| \\&
\leq \|f\|_{L^{p}(\G)}\left\|\frac{g\ast T^{(3)}_{Q-\lambda}}{|x|^{\alpha}}\right\|_{L^{p'}(\G)}.
\end{split}
\end{equation}
Note that the conditions $\alpha+\lambda\leq Q$ and $1/p+1/q+(\alpha+\lambda)/Q=2$ imply $q\leq p'$, while  $0<\lambda<Q$, $\alpha<Q/p'$ and $1/p+1/q+(\alpha+\lambda)/Q=2$ give
$$0<Q-\lambda=Q-Q\left(2-\frac{1}{p}-\frac{1}{q}\right)+\alpha<Q-Q\left(2-\frac{1}{p}-\frac{1}{q}\right)+\frac{Q}{p'}=Q/q.$$ Since we have $1<q\leq p'<\infty$, $0\leq \alpha p'<Q$, $0<Q-\lambda<Q/q$ and $(Q-\lambda)/Q=1/q-1/p'+\alpha/Q$, using Theorem \ref{Hardy_thm_new} in \eqref{HLS_ineq2} we obtain \eqref{HLS_ineq1}.
\end{proof}
\begin{rem}\label{HLS_rem_inv} Let us make some remarks concerning the reversed Hardy-Littlewood-Sobolev inequality on homogeneous groups (see \cite{DZ14}, \cite{NN17} and \cite{DFH18} for the recent Euclidean analysis of such inequalities). Namely, let us look at the validity of the inequality
\begin{equation}\label{rev_HLS1}
\int_{\G}\int_{\G} f(x)|y^{-1}x|^{\lambda}f(y)dxdy\geq C_{Q,\lambda, p}\|f\|^{\theta}_{L^{1}(\G)}\|f\|^{2-\theta}_{L^{p}(\G)}
\end{equation}
for any $0\leq f\in L^{1}\cap L^{p}(\G)$ with $f\not\equiv 0$ and $0<p<1$, where $\lambda>0$ and $\theta:=(2Q-p(2Q+\lambda))/(Q(1-p))$. When $\G=(\Rn,+)$, hence $Q=n$, the case $p=2n/(2n+\lambda)$ is investigated in \cite{DZ14} and \cite{NN17}, and the case $p>n/(n+\lambda)$ is studied in \cite{DFH18}.

We show that in the case $0<p\leq Q/(Q+\lambda)$ the inequality \eqref{rev_HLS1} is not valid, namely we show that \eqref{rev_HLS1} fails for any $C_{Q,\lambda, p}>0$. This is showed in the Euclidean case in \cite{CDP18} when $p<n/(n+\lambda)$ and in \cite{DFH18} when $p\leq n/(n+\lambda)$.

We consider
$$f_{\varepsilon}(x):=f(x)+A \varepsilon^{-Q}h(x/\varepsilon),$$
for a non-negative function $f$ with compact support and for a non-negative smooth fuction $h$ with the property $\int_{\G}h(x)dx=1$, and for some $A>0$. Suppose \eqref{rev_HLS1} holds for some $C_{Q,\lambda, p}>0$. Putting this $f_{\varepsilon}$ in the inequality \eqref{rev_HLS1}, we obtain
\begin{multline}\label{rev_HLS2}
C_{Q,\lambda,p}\leq \frac{\int_{\G}\int_{\G} f_{\varepsilon}(x)|y^{-1}x|^{\lambda}f_{\varepsilon}(y)dxdy}
{\|f_{\varepsilon}\|^{\theta}_{L^{1}(\G)}\|f_{\varepsilon}\|^{2-\theta}_{L^{p}(\G)}} \\
\rightarrow\frac{\int_{\G}\int_{\G} f(x)|y^{-1}x|^{\lambda}f(y)dxdy+2A\int_{\G}|x|^{\lambda}f(x)dx}
{(\int_{\G}f(x)dx+A)^{\theta}(\int_{\G}(f(x))^{p}dx)^{(2-\theta)/p}}
\end{multline}
as $\varepsilon \rightarrow 0_{+}$, where we have used $\int_{\G}f_{\varepsilon}(x)dx=\int_{\G}f(x)dx+A$, and when $\varepsilon \rightarrow 0_{+}$ the following facts
$$\int_{\G}(f_{\varepsilon}(x))^{p}dx\rightarrow \int_{\G}(f(x))^{p}dx$$
and
\begin{multline*}
\int_{\G}\int_{\G} f_{\varepsilon}(x)|y^{-1}x|^{\lambda}f_{\varepsilon}(y)dxdy=
\int_{\G}\int_{\G} f(x)|y^{-1}x|^{\lambda}f(y)dxdy\\+2A\int_{\G}\int_{\G} f(x)|(\varepsilon^{-1}y)^{-1}x|^{\lambda}
h(y)dxdy
+A^{2}\varepsilon^{-2Q}\int_{\G}\int_{\G}h\left(\frac{x}{\varepsilon}\right)h\left(\frac{y}{\varepsilon}\right)dxdy\\
\rightarrow \int_{\G}\int_{\G} f(x)|y^{-1}x|^{\lambda}f(y)dxdy+2A\int_{\G}|x|^{\lambda}f(x)dx,
\end{multline*}
since $\int_{\G}h(x)dx=1$.
Note that we can take the limit as $A\rightarrow +\infty$ in \eqref{rev_HLS2}, since it is valid for all $A>0$. Then, when $\theta>1$, i.e., $p<Q/(Q+\lambda)$, taking $A\rightarrow +\infty$ in \eqref{rev_HLS2} we see that $C_{Q,\lambda,p}=0$. In the case $\theta=1$, that is, $p=Q/(Q+\lambda)$, taking again the limit as $A\rightarrow +\infty$ in \eqref{rev_HLS2} we get
\begin{equation}\label{rev_HLS3}
C_{Q,\lambda,p}\leq \frac{2\int_{\G}|x|^{\lambda}f(x)dx}
{(\int_{\G}(f(x))^{p}dx)^{1/p}}.
\end{equation}
Now we show that the right-hand side of \eqref{rev_HLS3} goes to zero when $R\rightarrow\infty$ if we put there the function
\begin{equation}\label{rev_HLS4}
f_{R}(x)=
\begin{cases} |x|^{-(Q+\lambda)}, \text{\;for}\;1\leq |x|\leq R,\\
0,   \text{\;otherwise},\end{cases}
\end{equation}
for any $R>1$. Indeed, taking into account $p=Q/(Q+\lambda)$ we obtain from \eqref{rev_HLS3} that
\begin{equation}\label{rev_HLS5}
C_{Q,\lambda,p}\leq \frac{2\int_{\G}|x|^{\lambda}f_{R}(x)dx}
{(\int_{\G}(f_{R}(x))^{p}dx)^{1/p}}=2(|\wp|\log R)^{-\lambda/Q}\rightarrow 0
\end{equation}
as $R\rightarrow\infty$, where $|\wp|$ is a $Q-1$ dimensional surface measure of the unit quasi-sphere in $\G$.

Thus, we have proved that the reversed Hardy-Littlewood-Sobolev inequality \eqref{rev_HLS1} is not valid with any positive constant $C_{Q,\lambda,p}$ for $0<p\leq Q/(Q+\lambda)$.
\end{rem}

\section{Hypoelliptic Hardy, Rellich, Caffarelli-Kohn-Nirenberg and Hardy-Littlewood-Sobolev inequalities}
\label{SEC:Hardy_grad}

In this section we obtain Hardy inequality on graded groups. Actually, we obtain a more general inequality, which implies Hardy, Sobolev and Rellich inequalities on graded groups. Moreover, we show the relations between Hardy and weighted Trudinger-Moser inequalities, which imply the critical case of \eqref{Ricci1} when $\gamma=Q/p$. Furthermore, Caffarelli-Kohn-Nirenberg and Hardy-Littlewood-Sobolev inequalities, and uncertainty type principle are established.

Since we have \eqref{Rie_lem1} for the Riesz kernel $\mathcal{I}_{\alpha}$ from \eqref{Rie_pot}, taking $T^{(1)}_{a}(x)=\mathcal{I}_{a}(x)$ in Theorem \ref{Hardy_thm_new} and noting that $\R^{-\frac{a}{\nu}}f=f\ast \mathcal{I}_{a}$ by \cite[Corollary 4.3.11]{FR16}, we obtain
\begin{thm}\label{Hardy_thm_grad} Let $\mathbb{G}$ be a graded Lie group of homogeneous dimension $Q$ and let $\mathcal{R}$ be a positive Rockland operator of homogeneous degree $\nu$. Let $|\cdot|$ be an arbitrary homogeneous quasi-norm. Let $1<p\leq q<\infty$ and $0<a<Q/p$. Let $0\leq b<Q$ and $\frac{a}{Q}=\frac{1}{p}-\frac{1}{q}+\frac{b}{qQ}$. Then there exists a positive constant $C$ such that
\begin{equation}\label{Hardy_grad1}
\left\|\frac{f}{|x|^{\frac{b}{q}}}\right\|_{L^{q}(\G)}\leq
C\|\R^{\frac{a}{\nu}}f\|_{L^{p}(\G)}
\end{equation}
holds for all $f\in \dot{L}^{p}_{a}(\G)$.
\end{thm}
\begin{rem} In the case $b=0$, the inequality \eqref{Hardy_grad1} gives the Sobolev inequality on graded groups \cite[Proposition 4.4.13, (5)]{FR16}: Let $1<p<q<\infty$ and $0<a<Q/p$ with $\frac{a}{Q}=\frac{1}{p}-\frac{1}{q}$. Then there exists a positive constant $C$ such that
\begin{equation}\label{Hardy_grad1_11}
\left\|f\right\|_{L^{q}(\G)}\leq
C\|\R^{\frac{a}{\nu}}f\|_{L^{p}(\G)}
\end{equation}
holds for all $f\in \dot{L}^{p}_{a}(\G)$.
\end{rem}
\begin{rem} In the case $q=p$ and $a=1$, the inequality \eqref{Hardy_grad1} gives the Hardy inequality on graded groups
\begin{equation}\label{Hardy_grad1_1}
\left\|\frac{f}{|x|}\right\|_{L^{p}(\G)}\leq
C\|\R^{\frac{1}{\nu}}f\|_{L^{p}(\G)}, \;\;1<p<Q,
\end{equation}
for all $f\in \dot{L}^{p}_{1}(\G)$.
\end{rem}
\begin{rem} In the case $q=p$ and $a=2$, the inequality \eqref{Hardy_grad1} gives the Rellich inequality on graded groups
\begin{equation}\label{Rellich_grad1}
\left\|\frac{f}{|x|^{2}}\right\|_{L^{p}(\G)}\leq
C\|\R^{\frac{2}{\nu}}f\|_{L^{p}(\G)}, \;\;1<p<\frac{Q}{2},
\end{equation}
for all $f\in \dot{L}^{p}_{2}(\G)$.
\end{rem}
\begin{rem}\label{Schurtest_proof}  We can also prove Theorem \ref{Hardy_thm_grad} using the Schur's test argument from Remark \ref{Schurtest_proof_hom} and the analysis on graded Lie groups developed in \cite{FR16}. We first prove it for $p=q$. By \eqref{Hardy_new1}, we have
\begin{equation}\label{Schur1}\left\|\frac{f\ast T^{(1)}_{b/q}}{|x|^{\frac{b}{q}}}\right\|_{L^{q}(\G)}\leq C
\|f\|_{L^{q}(\G)}
\end{equation}
for all $f\in L^{q}(\G)$, where $0<b<Q$ and $1<q<\infty$.
Here, since we have \eqref{Rie_lem1}, taking $T^{(1)}_{b/q}(x)=\mathcal{I}_{b/q}(x)$ and noting that $\R^{-\frac{b}{\nu q}}f=f\ast \mathcal{I}_{b/q}$ by \cite[Corollary 4.3.11]{FR16}, we obtain from \eqref{Schur1} that
\begin{equation}\label{Schur2}
\left\|\frac{f}{|x|^{\frac{b}{q}}}\right\|_{L^{q}(\G)}\leq
C\|\R^{\frac{b}{\nu q}}f\|_{L^{q}(\G)}
\end{equation}
holds for all $f\in \dot{L}^{q}_{b/q}(\G)$, where $0<b<Q$ and $1<q<\infty$.

Now for $q>p$, we use the Sobolev inequality \cite[Proposition 4.4.13, (5)]{FR16} in \eqref{Schur2} to get
\begin{equation}\label{Schur3}
\left\|\frac{f}{|x|^{\frac{b}{q}}}\right\|_{L^{q}(\G)}\leq
C\|\R^{\frac{b}{\nu q}}f\|_{L^{q}(\G)}\leq C\|\R^{\frac{a}{\nu}}f\|_{L^{p}(\G)},
\end{equation}
where $0<a<Q/p$ and $\frac{a}{Q}=\frac{1}{p}-\frac{1}{q}+\frac{b}{qQ}$.
\end{rem}
Similarly, putting $T^{(2)}_{a}(x)=\mathcal{B}_{a}(x)$ in Theorem \ref{Log_Hardy_thm} and using \eqref{Bes_lem1} with the Bessel kernel $\mathcal{B}_{a}$ from \eqref{Bes_pot}, by noting $(I+\R)^{-\frac{a}{\nu}}f=f\ast \mathcal{B}_{a}$ by \cite[Corollary 4.3.11]{FR16}, we obtain the critical case $b=Q$ of Theorem \ref{Hardy_thm_grad}:
\begin{thm}\label{Log_Hardy_grad_thm} Let $\mathbb{G}$ be a graded Lie group of homogeneous dimension $Q$ and let $\mathcal{R}$ be a positive Rockland operator of
homogeneous degree $\nu$. Let $|\cdot|$ be an arbitrary homogeneous quasi-norm and let $1<p<r<\infty$ and $p<q<(r-1)p'$, where $1/p+1/p'=1$. Then there exists a positive constant $C_{6}=C_{6}(p, q, r, Q)$ such that
\begin{equation}\label{Hardy_wholeG2_log_grad}
\left\|\frac{f}{\left(\log\left(e+\frac{1}{|x|}\right)\right)^{\frac{r}{q}}|x|^{\frac{Q}{q}}}\right\|_{L^{q}(\G)}\leq
C_{6}\|f\|_{L^{p}_{Q/p}(\G)}
\end{equation}
holds for all $f\in L^{p}_{Q/p}(\G)$.
\end{thm}
The Hardy inequality \eqref{Hardy_grad1} implies the following uncertainty type principle:
\begin{cor}\label{uncer_thm}
Let $\mathbb{G}$ be a graded Lie group of homogeneous dimension $Q$ and let $\mathcal{R}$ be a positive Rockland operator of homogeneous degree $\nu$. Let $|\cdot|$ be an arbitrary homogeneous quasi-norm. Let $1<p\leq q<\infty$ and $0<a<Q/p$. Let $0\leq b<Q$ and $\frac{a}{Q}=\frac{1}{p}-\frac{1}{q}+\frac{b}{qQ}$. Then there exists a positive constant $C$ such that
\begin{equation}\label{uncer_grad1}
\|\R^{\frac{a}{\nu}}f\|_{L^{p}(\G)}\||x|^{\frac{b}{q}}f\|_{L^{q'}(\G)}\geq
C\int_{\G}|f(x)|^{2}dx
\end{equation}
holds for all $f\in \dot{L}^{p}_{a}(\G)$, where $1/q+1/q'=1$.
\end{cor}
\begin{proof}[Proof of Theorem \ref{uncer_thm}] Using H\"{o}lder's inequality and \eqref{Hardy_grad1}, we have
\begin{equation*} \|\R^{\frac{a}{\nu}}f\|_{L^{p}(\G)}\||x|^{\frac{b}{q}}f\|_{L^{q'}(\G)}
\geq C\left\|\frac{f}{|x|^{b/q}}\right\|_{L^{q}(\G)}\||x|^{\frac{b}{q}}f\|_{L^{q'}(\G)} \geq
C\int_{\G}|f(x)|^{2}dx,
\end{equation*}
which is \eqref{uncer_grad1}.
\end{proof}

Now we discuss the Caffarelli-Kohn-Nirenberg inequalities. First, let us recall the classical Caffarelli-Kohn-Nirenberg inequality \cite{CKN84}:
\begin{thm}\label{clas_CKN}
Let $n\in\mathbb{N}$ and let $p$, $q$, $r$, $a$, $b$, $d$, $\delta\in \mathbb{R}$ such that $p,q\geq1$, $r>0$, $0\leq\delta\leq1$, and
\begin{equation}\label{clas_CKN0}
\frac{1}{p}+\frac{a}{n},\, \frac{1}{q}+\frac{b}{n},\, \frac{1}{r}+\frac{c}{n}>0
\end{equation}
where $c=\delta d + (1-\delta) b$. Then there exists a positive constant $C$ such that
\begin{equation}\label{clas_CKN1}
\||x|^{c}f\|_{L^{r}(\Rn)}\leq C \||x|^{a}|\nabla f|\|^{\delta}_{L^{p}(\Rn)} \||x|^{b}f\|^{1-\delta}_{L^{q}(\Rn)}
\end{equation}
holds for all $f\in C_{0}^{\infty}(\Rn)$, if and only if the following conditions hold:
\begin{equation}\label{clas_CKN2}
\frac{1}{r}+\frac{c}{n}=\delta \left(\frac{1}{p}+\frac{a-1}{n}\right)+(1-\delta)\left(\frac{1}{q}+\frac{b}{n}\right),
\end{equation}
\begin{equation}\label{clas_CKN3}
a-d\geq 0 \quad {\rm if} \quad \delta>0,
\end{equation}
\begin{equation}\label{clas_CKN4}
a-d\leq 1 \quad {\rm if} \quad \delta>0 \quad {\rm and} \quad \frac{1}{r}+\frac{c}{n}=\frac{1}{p}+\frac{a-1}{n}.
\end{equation}
\end{thm}
As another consequence of Theorem \ref{Hardy_thm_grad}, we also obtain a family of extended Caffarelli-Kohn-Nirenberg inequalities on graded groups.
\begin{thm}\label{CKN_thm}
Let $\mathbb{G}$ be a graded Lie group of homogeneous dimension $Q$ and let $\mathcal{R}$ be a positive Rockland operator of homogeneous degree $\nu$. Let $|\cdot|$ be an arbitrary homogeneous quasi-norm. Let $1<p,q<\infty$, $\delta\in(0,1]$ and $0<r<\infty$ with $r\leq  \frac{q}{1-\delta}$ for $\delta\neq1$. Let $0<a<Q/p$ and $\beta$, $\gamma\in\mathbb{R}$ with $\delta r (Q-ap-\beta p)\leq p(Q+r\gamma-r\beta)$ and $\beta (1-\delta)-\delta a \leq \gamma \leq \beta(1-\delta)$. Assume that $\frac{r(\delta Q+p(\beta(1-\delta)-\gamma-a\delta))}{pQ}+\frac{(1-\delta)r}{q}=1$. Then there exists a positive constant $C$ such that
\begin{equation}\label{CKN_thm2}
\||x|^{\gamma}f\|_{L^{r}(\mathbb{G})}
\leq C \left\|\R^{\frac{a}{\nu}}f\right\|^{\delta}_{L^{p}(\mathbb{G})}
\left\||x|^{\beta}f\right\|^{1-\delta}_{L^{q}(\mathbb{G})}
\end{equation}
holds for all $f\in \dot{L}^{p}_{a}(\G)$.
\end{thm}
In the Euclidean case $\G=(\Rn,+)$ with $Q=n$, if the conditions \eqref{clas_CKN0} are not satisfied, then the inequality \eqref{CKN_thm2} is not covered by Theorem \ref{clas_CKN}. So, this actually also gives an extension of Theorem \ref{clas_CKN} with respect to the range of parameters. Let us give an example:
\begin{exa} If $1<p=q=r<n$, $a=1$, $\R=-\Delta$ and $\gamma=\beta(1-\delta)-\delta$, then \eqref{CKN_thm2} takes the form
\begin{equation}\label{exa_CKN}
\begin{split}
\||x|^{\gamma}f\|_{L^{p}(\Rn)}
&\leq C \left\|(-\Delta)^{\frac{1}{2}}f\right\|^{\delta}_{L^{p}(\Rn)}
\left\||x|^{\beta}f\right\|^{1-\delta}_{L^{p}(\Rn)}.
\end{split}
\end{equation}
Here, we can take e.g. $\beta\leq-n/p$ or $\gamma\leq-n/p$ so that the conditions \eqref{clas_CKN0} are not satisfied, then the inequality \eqref{exa_CKN} is not covered by Theorem \ref{clas_CKN}.
\end{exa}
\begin{rem} We note that the conditions $\beta=\gamma=0$, $a>0$, $1<p<Q/a$, $1< q\leq r\leq pQ/(Q-ap)$ and $\delta=(1/q-1/r)(a/Q+1/q-1/p)^{-1}$ satisfy all the conditions of Theorem \ref{CKN_thm}. Indeed, $\delta=(1/q-1/r)(a/Q+1/q-1/p)^{-1}$, $r\geq q$ and $Q-ap>0$ imply $r\leq \frac{q}{1-\delta}$, while $r\leq pQ/(Q-ap)$ gives $\delta r (Q-ap-\beta p)\leq p(Q+r\gamma-r\beta)$ since $\beta=\gamma=0$ and $\delta\leq 1$. In this case, $\delta=(1/q-1/r)(a/Q+1/q-1/p)^{-1}$ and $\beta (1-\delta)-\delta a \leq \gamma \leq \beta(1-\delta)$ are equivalent to $\frac{r(\delta Q+p(\beta(1-\delta)-\gamma-a\delta))}{pQ}+\frac{(1-\delta)r}{q}=1$ and $a\delta\geq0$, respectively. Thus, \eqref{CKN_thm2} recovers also the Gagliardo-Nirenberg inequality previously obtained in \cite{RT17} and \cite{RTY17} on graded groups
\begin{equation}\label{CKN_thm2_exa_GN}
\|f\|_{L^{r}(\mathbb{G})}
\leq C \left\|\R^{\frac{a}{\nu}}f\right\|^{\delta}_{L^{p}(\mathbb{G})}
\left\|f\right\|^{1-\delta}_{L^{q}(\mathbb{G})}
\end{equation}
for all $f\in \dot{L}^{p}_{a}(\G)\cap L^{q}(\G)$.

We also note that when $\G=(\Rn,+)$, $Q=n$ and $\R=-\Delta$, in the special case $p=q=2$ and $a=1$, the inequality \eqref{CKN_thm2_exa_GN} essentially gives the classical Gagliardo-Nirenberg inequality \cite{Gag59} and \cite{Nir59}.
\end{rem}
Note that another type of Garliardo-Nirenberg inequality involving Besov norms on graded groups was obtained in \cite{BFG12}.
\begin{proof}[Proof of Theorem \ref{CKN_thm}]
{\bf Case $\delta=1$}. Notice that in this case, $\frac{r(\delta Q+p(\beta(1-\delta)-\gamma-a\delta))}{pQ}+\frac{(1-\delta)r}{q}=1$ gives $\frac{a}{Q}=\frac{1}{p}-\frac{1}{r}-\frac{\gamma}{Q}$, which implies that the condition $\delta r (Q-ap-\beta p)\leq p(Q+r\gamma-r\beta)$ is equivalent to the trivial estimate $pQ\leq pQ$. The condition $\beta (1-\delta)-\delta a \leq \gamma \leq \beta(1-\delta)$ gives $-a\leq \gamma \leq 0$, which implies $r\geq p$ with $\frac{a}{Q}=\frac{1}{p}-\frac{1}{r}-\frac{\gamma}{Q}$. Taking into account these we see that \eqref{CKN_thm2} is equivalent to \eqref{Hardy_grad1}.

{\bf Case $\delta\in(0,1)$}. We write $$\||x|^{\gamma}f\|_{L^{r}(\mathbb{G})}=
\left(\int_{\mathbb{G}}|x|^{\gamma r}|f(x)|^{r}dx\right)^{\frac{1}{r}}
=\left(\int_{\mathbb{G}}\frac{|f(x)|^{\delta r}}{|x|^{r (\beta(1-\delta)-\gamma)}}\cdot \frac{|f(x)|^{(1-\delta)r}}{|x|^{-\beta r(1-\delta)}}dx\right)^{\frac{1}{r}}.$$
Note that $\delta>0$, $Q>ap$ and $\beta(1-\delta)-\gamma\geq0$ imply $r(\delta Q+p(\beta(1-\delta)-\gamma-a\delta))>0$, while $\delta r (Q-ap-\beta p)\leq p(Q+r\gamma-r\beta)$, $\delta<1$ and $r\leq \frac{q}{1-\delta}$ give $\frac{pQ}{r(\delta Q+p(\beta(1-\delta)-\gamma-a\delta))}\geq 1$ and $\frac{q}{(1-\delta)r}\geq 1$, respectively. Then by using H\"{o}lder's inequality for $\frac{r(\delta Q+p(\beta(1-\delta)-\gamma-a\delta))}{pQ}+\frac{(1-\delta)r}{q}=1$, we obtain
$$\||x|^{\gamma }f\|_{L^{r}(\mathbb{G})}
\leq \left(\int_{\mathbb{G}}\frac{|f(x)|^{\frac{\delta pQ}{\delta Q+p(\beta(1-\delta)-\gamma-a\delta)}}}{|x|^{\frac{pQ(\beta(1-\delta)-\gamma)}{\delta Q+p(\beta(1-\delta)-\gamma-a\delta)}}}dx\right)
^{\frac{\delta Q+p(\beta(1-\delta)-\gamma-a\delta)}{pQ}}
\left(\int_{\mathbb{G}}\frac{|f(x)|^{q}}{|x|^{-\beta q}}dx\right)^{\frac{1-\delta}{q}}$$
\begin{equation}\label{CKN_thm1_1}=\left\|\frac{f}{|x|^{\frac{\beta(1-\delta)-\gamma}{\delta}}}\right\|^{\delta}_{L^{\frac{\delta pQ}{\delta Q+p(\beta(1-\delta)-\gamma-a\delta)}}(\mathbb{G})}
\left\|\frac{f}{|x|^{-\beta}}\right\|^{1-\delta}_{L^{q}(\mathbb{G})}.
\end{equation}
We also note that the conditions $\frac{\delta pQ}{\delta Q+p(\beta(1-\delta)-\gamma-a\delta)}\geq \delta r>0$ and $\beta(1-\delta)-\gamma\geq0$ imply \begin{equation}\label{CKN_grad1}\frac{\delta pQ}{\delta Q+p(\beta(1-\delta)-\gamma-a\delta)}\cdot \frac{\beta(1-\delta)-\gamma}{\delta}\geq0,
\end{equation}
while $Q>ap$ and $\delta>0$ give
\begin{equation}\label{CKN_grad2}\frac{\delta pQ}{\delta Q+p(\beta(1-\delta)-\gamma-a\delta)}\cdot \frac{\beta(1-\delta)-\gamma}{\delta}<Q.
\end{equation}
Then, \eqref{CKN_grad1}, \eqref{CKN_grad2} and
$$\frac{a}{Q}=\frac{1}{p}-\frac{1}{\frac{\delta pQ}{\delta Q+p(\beta(1-\delta)-\gamma-a\delta)}}+\frac{\frac{\beta(1-\delta)-\gamma}{\delta}}{Q}$$
with $\gamma\geq \beta (1-\delta)-\delta a$ imply $\frac{\delta pQ}{\delta Q+p(\beta(1-\delta)-\gamma-a\delta)}\geq p$, so that we can use Theorem \ref{Hardy_thm_grad} in \eqref{CKN_thm1_1} to obtain \eqref{CKN_thm2}.
\end{proof}
Now we show the weighted improved Hardy-Littlewood-Sobolev inequality on graded groups.
\begin{thm}\label{HLS_thm_grad} Let $\mathbb{G}$ be a graded Lie group of homogeneous dimension $Q$ and let $|\cdot|$ be an arbitrary homogeneous quasi-norm. Let $1<p,q<\infty$, $0\leq a<Q/p$ and $0\leq b<Q/q$. Let $0<\lambda<Q$, $0\leq \alpha <a+Q/p'$ and $0\leq \beta\leq b$ be such that $(Q-ap)/(pQ)+(Q-q(b-\beta))/(qQ)+(\alpha+\lambda)/Q=2$ and $\alpha+\lambda\leq Q$, where $1/p+1/p'=1$. Then there exists a positive constant $C=C(Q,\lambda, p, \alpha, \beta, a, b)$ such that
\begin{equation}\label{HLS_ineq1_grad}
\left|\int_{\G}\int_{\G}\frac{\overline{f(x)}g(y)}{|x|^{\alpha}|y^{-1}x|^{\lambda}|y|^{\beta}}dxdy\right|\leq C\|f\|_{\dot{L}^{p}_{a}(\G)}\|g\|_{\dot{L}^{q}_{b}(\G)}
\end{equation}
holds for all $f\in \dot{L}^{p}_{a}(\G)$ and $g\in \dot{L}^{q}_{b}(\G)$.
\end{thm}
\begin{proof}[Proof of Theorem \ref{HLS_thm_grad}] We first prove it for $a\neq 0$ and $b\neq 0$. We want to use Theorem \ref{HLS_thm} in the left hand side of \eqref{HLS_ineq1_grad} to get
\begin{equation}\label{HLS_ineq1_grad_1}
\left|\int_{\G}\int_{\G}\frac{\overline{f(x)}g(y)}{|x|^{\alpha}|y^{-1}x|^{\lambda}|y|^{\beta}}dxdy\right|\leq C\|f\|_{L^{p_{1}}(\G)}\left\|\frac{g}{|y|^{\beta}}\right\|_{L^{q_{1}}(\G)},
\end{equation}
where $p_{1}:=\frac{pQ}{Q-ap}$ and $q_{1}:=\frac{qQ}{Q-q(b-\beta)}$. For this, let us check conditions of Theorem \ref{HLS_thm}. Note that $0<a<Q/p\Rightarrow 1<p_{1}<\infty$, while $0<b<Q/q$ and $0\leq \beta\leq b$ imply $1<q_{1}<\infty$. We also note that $0\leq \alpha <a+Q/p'\Rightarrow 0\leq \alpha <Q/p'_{1}$ with $p_{1}^{\prime}=p_{1}/(p_{1}-1)$ and $2=(Q-ap)/(pQ)+(Q-q(b-\beta))/(qQ)+(\alpha+\lambda)/Q=1/p_{1}+1/q_{1}+(\alpha+\lambda)/Q$. Thus, since we also have $0<\lambda<Q$ and $\alpha+\lambda\leq Q$, we obtain \eqref{HLS_ineq1_grad_1}.

We have $1<p<p_{1}<\infty$, $0<a<Q/p$ and $\frac{a}{Q}=\frac{1}{p}-\frac{1}{p_{1}}$ since $p_{1}:=\frac{pQ}{Q-ap}$, then applying the Sobolev inequality \eqref{Hardy_grad1_11} on graded groups (or \cite[Proposition 4.4.13, (5)]{FR16}) we get
\begin{equation}\label{HLS_ineq1_grad_2}\|f\|_{L^{p_{1}}(\G)}\leq C \|f\|_{\dot{L}^{p}_{a}(\G)}.
\end{equation}
Since $Q-q(b-\beta)>0$ and $Q-qb>0$ we have $0\leq \frac{\beta qQ}{Q-q(b-\beta)}<Q$, that is, $0\leq \beta q_{1}<Q$ since $q_{1}:=\frac{qQ}{Q-q(b-\beta)}$. We also have $b/Q=1/q-1/q_{1}+\beta/Q$ since $q_{1}:=\frac{qQ}{Q-q(b-\beta)}$ and $1<q\leq q_{1}<\infty$. Then we can use \eqref{Hardy_grad1}, i.e.
\begin{equation}\label{HLS_ineq1_grad_3}\left\|\frac{g}{|y|^{\beta}}\right\|_{L^{q_{1}}(\G)}\leq C\|g\|_{\dot{L}^{q}_{b}(\G)}.
\end{equation}

Finally, putting \eqref{HLS_ineq1_grad_2} and \eqref{HLS_ineq1_grad_3} in \eqref{HLS_ineq1_grad_1}, we obtain \eqref{HLS_ineq1_grad}.

In the case $a=0$, the inequalities \eqref{HLS_ineq1_grad_3} and \eqref{HLS_ineq1_grad_1} give \eqref{HLS_ineq1_grad}.

When $b=0$, we have $\beta=0$ since $0\leq \beta\leq b$, then \eqref{HLS_ineq1_grad_1} implies \eqref{HLS_ineq1_grad}.
\end{proof}

\section{Weighted Trudinger-Moser inequalities with remainder terms}
\label{SEC:weighted_Trudinger} In this section we show local and global weighted Trudinger-Moser inequalities on general graded groups
$\G$, noting that they are new already on the Heisenberg groups.

Let $\Omega$ be a bounded domain in $\G$ with smooth boundary. Let $a\geq 0$ and $1<p<\infty$. Let $L_{a}^{p}(\Omega)$ be the completion of $C_{0}^{\infty}(\Omega)$ with respect to the norm
\begin{equation}\label{def_space}
\|f\|_{L_{a}^{p}(\Omega)}=\left(\int_{\Omega}(|\R^{\frac{a}{\nu}}f(x)|^{p}+|f(x)|^{p})dx\right)^{1/p}.
\end{equation}
We note that the powers $\R^{\frac{a}{\nu}}$ are well-defined on  $\G$ , see e.g. \cite[Chapter 4.3]{FR16} or \cite{FR:Sobolev}, we then integrate them in \eqref{def_space} over $\Omega\subset\G$.
Let us recall the following results:
\begin{thm}[{\cite[Theorem 3.3]{RY17}}]
\label{crit_GN_thm}
Let $\mathbb{G}$ be a graded Lie group of homogeneous dimension $Q$ and let $\mathcal{R}$ be a positive Rockland operator of
homogeneous degree $\nu$. Then there exists some constant $\widetilde{C_{1}}$ depending only on $p$ and $Q$ such that
\begin{equation}\label{crit_GN_ineq}
\|f\|_{L^{q}(\G)}\leq \widetilde{C_{1}}q^{1-1/p}\|\R^{\frac{Q}{\nu
p}}f\|_{L^{p}(\G)}^{1-p/q}\|f\|_{L^{p}(\G)}^{p/q},\;\;1<p<\infty,
\end{equation}
holds for every $q$ with $p\leq q<\infty$ and for every function $f\in L_{Q/p}^{p}(\G)$.
\end{thm}
\begin{thm}[{\cite[Theorem 3.5]{RY17}}]
\label{Trud_thm} Let $\mathbb{G}$ be a graded group of homogeneous dimension $Q$ and let $\mathcal{R}$ be a positive Rockland
operator of homogeneous degree $\nu$. Then there exist positive constants $\alpha$ and $\widetilde{C_{2}}$ such that
\begin{equation}\label{Trud_ineq}
\int_{\G}\left(\exp(\alpha|f(x)|^{p'})-\sum_{0\leq k<p-1,\;k\in\mathbb{N}}\frac{1}{k!}(\alpha|f(x)|^{p'})^{k}\right)dx\leq
\widetilde{C_{2}}\|f\|^{p}_{L^{p}(\G)},\;\;1<p<\infty,
\end{equation}
holds for all functions $f\in L_{Q/p}^{p}(\G)$ with $\|\R^{\frac{Q}{\nu p}}f\|_{L^{p}(\G)}\leq1$, where $1/p+1/p'=1$.
\end{thm}
\begin{rem}[{\cite[Remark 3.6]{RY17}}]
\label{rem_Trud_thm} The constant $\widetilde{C_{2}}$ in \eqref{Trud_ineq} can be expressed in terms of the constant
$\widetilde{C_{1}}=\widetilde{C_{1}}(p,Q)$ in \eqref{crit_GN_ineq} as follows
$$\widetilde{C_{2}}=\widetilde{C_{2}}(\alpha)=\sum_{k\geq
p-1,\;k\in\mathbb{N}}\frac{k^{k}}{k!}(p'\widetilde{C_{1}}^{p'}\alpha)^{k}.$$
Therefore, \eqref{Trud_ineq} is valid for all $\alpha\in(0, (ep'\widetilde{C_{1}}^{p'})^{-1})$ and $\widetilde{C_{2}}(\alpha)$.
\end{rem}

Theorems \ref{crit_GN_thm} and \ref{Trud_thm} imply the following corollary:
\begin{cor}\label{cor1} Let $\mathbb{G}$ be a graded group of homogeneous dimension $Q$ and let $\mathcal{R}$ be a positive
Rockland operator of homogeneous degree $\nu$. Let $1<p<\infty$. Then there exist some positive constants $\alpha$ and
$\widetilde{C_{3}}$ such that
\begin{equation}\label{Trud1}
\int_{\Omega}\exp(\alpha|f(x)|^{p'})dx\leq \widetilde{C_{3}}
\end{equation}
holds for all bounded smooth domains $\Omega\subset\G$, and for all functions $f\in L_{Q/p}^{p}(\Omega)$ with
$\|f\|_{L^{p}_{Q/p}(\Omega)}\leq1$, where $1/p+1/p'=1$ and $\widetilde{C_{3}}=\widetilde{C_{3}}(p,Q,\alpha)$.
\end{cor}
\begin{proof}[Proof of Corollary \ref{cor1}] Using \eqref{crit_GN_ineq} in \eqref{Trud_ineq} for all $f\in L_{Q/p}^{p}(\Omega)$ with $\|f\|_{L^{p}_{Q/p}(\Omega)}\leq1$ we get
\begin{equation*}
\begin{split}\int_{\Omega}\exp(\alpha|f(x)|^{p'})dx&\leq
\int_{\Omega}\sum_{0\leq k<p-1,\;k\in\mathbb{N}}\frac{1}{k!}(\alpha|f(x)|^{p'})^{k}dx+\widetilde{C_{2}}\|f\|^{p}_{L^{p}(\Omega)}\\&
\leq \sum_{0\leq k<p-1,\;k\in\mathbb{N}}\frac{\alpha^{k}}{k!}\widetilde{C_{1}}^{kp'}(kp')^{kp'-k/(p-1)}\|f\|_{L^{p}(\Omega)}^{p}
+\widetilde{C_{2}}\|f\|^{p}_{L^{p}(\Omega)}\leq C,
\end{split}
\end{equation*}
since $\|f\|_{L^{p}(\Omega)}\leq1$.
\end{proof}
\begin{rem} From the proof of Corollary \ref{cor1}, we see that actually we have \eqref{Trud1} with remainder terms for all bounded smooth domains $\Omega\subset\G$, and for all functions $f\in L_{Q/p}^{p}(\Omega)$ with
$\|\R^{\frac{Q}{\nu p}} f\|_{L^{p}(\Omega)}\leq1$, in the form of
\begin{equation}\label{Trud1_rem}
\int_{\Omega}\exp(\alpha|f(x)|^{p'})dx\leq \widetilde{C_{3}}\|f\|_{L^{p}(\Omega)}.
\end{equation}
\end{rem}
\begin{rem}\label{rem2} By Remark \ref{rem_Trud_thm}, we see that \eqref{Trud1} and \eqref{Trud1_rem} are actually valid for all $\alpha \in[0, (ep'$
$\widetilde{C_{1}}^{p'})^{-1})$, where $\widetilde{C_{1}}$ is defined in \eqref{crit_GN_ineq}. We also note that the smallest
constant $\widetilde{C_{1}}$ (and hence also $\widetilde{C_{2}}$) can be expressed in the variational form as well as in terms of the ground state solutions of the
nonlinear Schr\"{o}dinger type equations (see \cite[Section 5]{RY17}).
\end{rem}

Let us also recall the subelliptic $Q$-Laplacian for $Q\geq 3$
$$-\Delta_{Q,H}f:=-{\rm div}_{H}(|\nabla_{H}f|^{Q-2}\nabla_{H}f),$$
where $\nabla_{H}$ is the horizontal gradient. We refer to \cite{BMT03} for more details.
Finally, we recall the following results:
\begin{thm}[{\cite[Theorem 2.6]{BMT03}}]
\label{Tyson_thm}
Let $\G$ be a stratified group with homogeneous dimension $Q\geq 3$ and let $u_{Q}$ be a
singular solution for the subelliptic $Q$-Laplacian with pole at $0\in \G$. Then there exists a positive constant $a_{Q}$ such that the function
\begin{equation}\label{Tyson_hom}
N(x)=\exp(-a_{Q}u_{Q}(x))
\end{equation}
is a homogeneous norm on $\G$.
\end{thm}
\begin{thm}[{\cite[Theorem 4.1]{BMT03}}]
\label{Tyson_thm2} Let $\G$ be a stratified group with homogeneous dimension $Q\geq 3$ and let $N(x)$ is a homogeneous norm on $\G$ as in Theorem \ref{Tyson_thm}. Let $\alpha_{Q}=Qc_{Q}^{Q'-1}$ with $Q'=Q/(Q-1)$ and $c_{Q}=\int_{\wp}|\nabla_{H}N(y)|^{Q}d\sigma(y)$, where $\wp:=\{x\in \mathbb{G}:\,|x|=1\}$ is the unit sphere with respect to the homogeneous norm $N$. Then there exists a positive constant $C$ such that
\begin{equation}\label{Tyson_thm_eq1}
\frac{1}{|\Omega|}\int_{\Omega}\exp \left(\alpha_{Q}\left(\frac{|f(x)|}{\|\nabla_{H} f\|_{L^{Q}(\G)}}\right)^{Q'}\right)dx\leq C
\end{equation}
holds for any domain $\Omega\subset \G$, $|\Omega|<\infty$ and all $f\in L^{1}_{Q}(\Omega)$. Moreover, if $\alpha_{Q}$ is replaced by any greater number the statement is false.
\end{thm}
\begin{rem}\label{explic_aQ_rem} Also, it is known that when $\G$ is the H-type group we have the following explicit representation (see e.g. {\cite[Remark on p.48]{BMT03}})
\begin{equation}\label{explic_aQ}
\alpha_{Q}=Q\left(\frac{2\pi^{(k+\ell)/2}\Gamma((Q-\ell)/2)}{4^{\ell}\Gamma(k/2)\Gamma(Q/2)}\right)^{Q'-1},
\end{equation}
where $k=\dim V_{1}$ and $\ell=\dim V_{2}$ are the dimensions of the horizontal space and the dimension of the center of $\G$, respectively.
\end{rem}
Now we are ready to state the local Trudinger-Moser inequalities:
\begin{thm}\label{locweightedTrud_thm} Let $\mathbb{G}$ be a graded Lie group of homogeneous dimension $Q$ and let $\mathcal{R}$
be a positive Rockland operator of homogeneous degree $\nu$. Let $|\cdot|$ be an arbitrary homogeneous quasi-norm. Let
$\beta\in[0,Q)$ and $1<p<\infty$. Let $r>0$ be given and let $x_{0}$ be any point of $\G$. Then there exist some positive
constants $C_{1}$ and $C_{2}$ such that
\begin{multline}\label{weightedTrud1}
\int_{B(x_{0},r)}\frac{1}{|x|^{\beta}}\left(\exp(\alpha|f(x)|^{p'})-\sum_{0\leq
k<p-1,\;k\in\mathbb{N}}\frac{\alpha^{k}|f(x)|^{kp'}}{k!}\right)dx
\\ \leq C_{1}\|f\|^{p}_{L^{p}_{Q/p}(B(x_{0},r))}
\end{multline}
holds for any $\alpha\in[0,C_{2})$ and for all $f\in L^{p}_{Q/p}(B(x_{0},r))$ with $\|f\|_{L^{p}_{Q/p}(B(x_{0},r))}\leq1$, where the space $L^{p}_{Q/p}(B(x_{0},r))$ is defined in \eqref{def_space}. Moreover, for any $\mu>Q/(Q-\beta)$ we can take $C_{2}=C_{2}(p,Q,\beta, \mu)=(e\widetilde{C_{1}}^{p'}\mu
p')^{-1}$ and
\begin{multline}\label{C1}
C_{1}=C_{1}(p,Q,\alpha,\beta, r, \mu)\\
=\max\left(\widetilde{C_{3}}r^{-\beta},
\widetilde{C_{3}}^{1/\mu}\frac{|\wp|^{1/\mu'}(C_{0}(2C_{0}+1))^{Q/\mu'-\beta}}{(Q-\beta\mu')^{1/\mu'}}r^{Q/\mu'-\beta}\right),
\end{multline}
where $1/\mu+1/\mu'=1$, $1/p+1/p'=1$, and $\widetilde{C_{1}}$ and $C_{0}$ are constants from \eqref{crit_GN_ineq} and \eqref{triangle}, respectively. Here, $|\wp|$ is a $Q-1$ dimensional surface measure of the unit quasi-sphere and $\widetilde{C_{3}}$ is the constant from
\eqref{Trud1}.
\end{thm}
When $\G$ is a stratified group and $p=Q$, we can obtain Trudinger-Moser inequalities with sharp constant. Note that we have the following facts on the stratified group $\G$:
\begin{itemize}
\item The triangle inequality for homogeneous quasi-norms (see Proposition \ref{triangle_euc});
\item The covering lemma \cite[Lemma 7.14]{FS-book} (or see \cite[Lemma 5.7.5]{FR16});
 \item Let $x_{0}$ be any point of $\G$. Then there exist a positive constant $C$ such that $|\nabla_{H}|x_{0}^{-1}x||\leq C$ holds for $x\neq x_{0}$, since $\nabla_{H}|x_{0}^{-1}x|$ is homogeneous of order zero, where $\nabla_{H}$ is the horizontal gradient in $\G$.
 \end{itemize}
Then, taking into account these facts, applying Theorem \ref{Tyson_thm2} with $\|\nabla_{H} f\|_{L^{Q}(B(x_{0},r))}\leq1$ and using the strategy developed in \cite{Yang14}, we can also obtain the following Theorems \ref{locweightedTrud_thm_Qp} and \ref{locweightedTrud_thm_Qp2} on stratified groups:
\begin{thm}\label{locweightedTrud_thm_Qp} Let $\mathbb{G}$ be a stratified group of homogeneous dimension $Q\geq 3$ and let $|\cdot|$ be a homogeneous quasi-norm on $\G$. Let $\beta\in[0,Q)$, and let $r>0$ be given and let $x_{0}$ be any point of $\G$. Let $\alpha_{Q}$ be as in Theorem \ref{Tyson_thm2}. Then there exists a positive
constant $C=C(Q,r,\beta)$ such that
\begin{multline}\label{weightedTrud1_Qp}
\int_{B(x_{0},r)}\frac{1}{|x|^{\beta}}\left(\exp(\alpha|f(x)|^{Q'})-\sum_{k=0}^{Q-2}\frac{\alpha^{k}|f(x)|^{kQ'}}{k!}\right)dx
\leq C\|\nabla_{H} f\|_{L^{Q}(B(x_{0},r))}^{Q}
\end{multline}
holds for all $f\in L^{Q}_{1}(B(x_{0},r))$ satisfying $\|\nabla_{H} f\|_{L^{Q}(B(x_{0},r))}\leq1$ and any $\alpha\in [0,\alpha_{Q}(1-\beta/Q)]$, where $Q'=Q/(Q-1)$ and the space $L^{Q}_{1}(B(x_{0},r))$ is defined in \eqref{def_space_str}.
\end{thm}
\begin{thm}\label{locweightedTrud_thm_Qp2} Let $\mathbb{G}$ be a stratified group of homogeneous dimension $Q\geq 3$ and let $|\cdot|$ be a homogeneous quasi-norm on $\G$. Let $\alpha_{Q}$ be as in Theorem \ref{Tyson_thm2}. Then we have
\begin{equation}\label{weightedTrud1_Qp2}
\sup_{\|f\|_{L^{Q}_{1}(\G)}\leq1}\int_{\G}\frac{1}{|x|^{\beta}}\left(\exp(\alpha|f(x)|^{Q'})
-\sum_{k=0}^{Q-2}\frac{\alpha^{k}|f(x)|^{kQ'}}{k!}\right)dx
< \infty
\end{equation}
for any $\beta\in[0,Q)$ and $\alpha\in (0,\alpha_{Q}(1-\beta/Q))$, where $Q'=Q/(Q-1)$. When $\alpha>\alpha_{Q}(1-\beta/Q)$, the integral in \eqref{weightedTrud1_Qp2} is still finite for any $f\in L^{Q}_{1}(\G)$, but the supremum is infinite.
\end{thm}
\begin{rem}\label{explic_aQ_rem2} In Theorems \ref{locweightedTrud_thm_Qp} and \ref{locweightedTrud_thm_Qp2}, we note that if we have more information on the homogeneous norm $N$ the constant $\alpha_{Q}=Qc_{Q}^{Q'-1}$ can be explicitly calculated. For example, when $\G$ is the H-type group we have \eqref{explic_aQ}.
\end{rem}
\begin{rem} In the case when $\mathbb{G}$ is the Heisenberg group, and $|\cdot|$ is the Kaplan distance, the obtained Theorems \ref{locweightedTrud_thm_Qp} and \ref{locweightedTrud_thm_Qp2} were established in \cite{Yang14}.
\end{rem}
\begin{proof}[Proof of Theorem \ref{locweightedTrud_thm}] By \eqref{triangle}, we have
\begin{equation}\label{loc_TM1}
|x_{0}|\leq C_{0}(|x_{0}^{-1}x|+|x|)
\end{equation}
for any $x\in\G$.

Let us first consider the case $|x_{0}|>2C_{0}r$. Then, from \eqref{loc_TM1} we have
\begin{equation}\label{1case}
|x|\geq \frac{|x_{0}|}{C_{0}}-|x_{0}^{-1}x|>r,\;\;\forall x\in B(x_{0},r).
\end{equation}
When $\|f\|_{L^{p}_{Q/p}(B(x_{0},r))}=0$, we have $f\equiv0$, that is, \eqref{weightedTrud1} is trivial. Therefore, we can assume that
$\|f\|_{L^{p}_{Q/p}(B(x_{0},r))}\neq0$. Setting $\widetilde{f}:=f/\|f\|_{L^{p}_{Q/p}(B(x_{0},r))}$
and taking into account that $\|f\|_{L^{p}_{Q/p}(B(x_{0},r))}$ $\leq1$, we calculate
\begin{equation}\label{1case_1}
\begin{split}\exp(\alpha|f(x)|^{p'})-\sum_{0\leq k<p-1,\;k\in\mathbb{N}}\frac{\alpha^{k}|f(x)|^{kp'}}{k!}&=\sum_{k\geq
p-1}\frac{\alpha^{k}|f(x)|^{p'k}}{k!}\\
&=\sum_{k\geq p-1}\frac{\alpha^{k}\|f\|_{L^{p}_{Q/p}(B(x_{0},r))}^{p'k}|\widetilde{f}(x)|^{p'k}}{k!}\\
&\leq\|f\|_{L^{p}_{Q/p}(B(x_{0},r))}^{p}\sum_{k\geq p-1}\frac{\alpha^{k}|\widetilde{f}(x)|^{p'k}}{k!}.
\end{split}
\end{equation}
It follows from this and \eqref{1case} that
\begin{equation}\label{loc_TM2}
\begin{split}
\int_{B(x_{0},r)}\frac{1}{|x|^{\beta}}&\left(\exp(\alpha|f(x)|^{p'})-\sum_{0\leq
k<p-1,\;k\in\mathbb{N}}\frac{\alpha^{k}|f(x)|^{kp'}}{k!}\right)dx\\
&\leq r^{-\beta}\int_{B(x_{0},r)}\left(\exp(\alpha|f(x)|^{p'})-\sum_{0\leq
k<p-1,\;k\in\mathbb{N}}\frac{\alpha^{k}|f(x)|^{kp'}}{k!}\right)dx\\&
\leq r^{-\beta}\|f\|_{L^{p}_{Q/p}(B(x_{0},r))}^{p}\int_{B(x_{0},r)}\sum_{k\geq
p-1}\frac{\alpha^{k}|\widetilde{f}(x)|^{p'k}}{k!}dx.
\end{split}
\end{equation}
Since $\|\widetilde{f}\|_{L^{p}_{Q/p}(B(x_{0},r))}\leq1$ and \eqref{1case}, then using \eqref{Trud1} in
\eqref{loc_TM2}, we get that
\begin{equation*}
\begin{split}
\int_{B(x_{0},r)}\frac{1}{|x|^{\beta}}&\left(\exp(\alpha|f(x)|^{p'})-\sum_{0\leq
k<p-1,\;k\in\mathbb{N}}\frac{\alpha^{k}|f(x)|^{kp'}}{k!}\right)dx\\&\leq
\widetilde{C_{3}}r^{-\beta}\|f\|_{L^{p}_{Q/p}(B(x_{0},r))}^{p},
\end{split}
\end{equation*}
for $|x_{0}|>2C_{0}r$ and $\alpha\in[0, (ep'\widetilde{C_{1}}^{p'})^{-1})$. Thus, we have proved \eqref{weightedTrud1} for
$|x_{0}|>2C_{0}r$.

Now let us consider the case $|x_{0}|\leq2C_{0}r$. If $x\in B(x_{0},r)$, then by \eqref{triangle} we get
$$|x|\leq C_{0}(|x_{0}^{-1}x|+|x_{0}|)<C_{0}(2C_{0}+1)r.$$
Then, by H\"{o}lder's inequality, \eqref{Trud1} and Remark \ref{rem2}, one calculates
\begin{equation}\label{2case_1}
\begin{split}
\int_{B(x_{0},r)}\frac{1}{|x|^{\beta}}&\sum_{k\geq p-1}\frac{\alpha^{k}|\widetilde{f}(x)|^{p'k}}{k!}dx\\&\leq
\int_{|x|\leq C_{0}(2C_{0}+1)r}\frac{1}{|x|^{\beta}}\sum_{k\geq p-1}\frac{\alpha^{k}|\widetilde{f}(x)|^{p'k}}{k!}dx\\&
\leq \left(\int_{|x|\leq C_{0}(2C_{0}+1)r}\frac{1}{|x|^{\beta\mu'}}dx\right)^{\frac{1}{\mu'}}
\left(\int_{|x|\leq C_{0}(2C_{0}+1)r}\exp(\alpha \mu|\widetilde{f}(x)|^{p'})dx\right)^{\frac{1}{\mu}}
\\&=\frac{|\wp|^{1/\mu'}(C_{0}(2C_{0}+1))^{Q/\mu'-\beta}}{(Q-\beta\mu')^{1/\mu'}}r^{Q/\mu'-\beta}
\left(\int_{|x|\leq C_{0}(2C_{0}+1)r}\exp(\alpha \mu|\widetilde{f}(x)|^{p'})dx\right)^{\frac{1}{\mu}}
\\&\leq
\widetilde{C_{3}}^{1/\mu}\frac{|\wp|^{1/\mu'}(C_{0}(2C_{0}+1))^{Q-\beta}}{(Q-\beta\mu')^{1/\mu'}}r^{Q/\mu'-\beta},
\end{split}
\end{equation}
for all $\alpha\in[0, (e\widetilde{C_{1}}^{p'}\mu p')^{-1})$, where $|\wp|$ is a $Q-1$ dimensional surface measure of the unit
quasi-sphere, $1/\mu+1/\mu'=1$ and $1<\mu'<Q/\beta$.
Combining \eqref{2case_1} with \eqref{1case_1}, one has
\begin{equation*}
\begin{split}
\int_{B(x_{0},r)}&\frac{1}{|x|^{\beta}}\left(\exp(\alpha|f(x)|^{p'})-\sum_{0\leq
k<p-1,\;k\in\mathbb{N}}\frac{\alpha^{k}|f(x)|^{kp'}}{k!}\right)dx\\&
\leq
\|f\|_{L^{p}_{Q/p}(B(x_{0},r))}^{p}
\int_{B(x_{0},r)}\frac{1}{|x|^{\beta}}\sum_{k\geq p-1}\frac{\alpha^{k}|\widetilde{f}(x)|^{p'k}}{k!}dx
\\&
\leq \widetilde{C_{3}}^{1/\mu}\frac{|\wp|^{1/\mu'}(C_{0}(2C_{0}+1))^{Q/\mu'-\beta}}{(Q-\beta\mu')^{1/\mu'}}r^{Q/\mu'-\beta}
\|f\|_{L^{p}_{Q/p}(B(x_{0},r))}^{p},
\end{split}
\end{equation*}
which implies \eqref{weightedTrud1} for $|x_{0}|\leq2C_{0}r$. Thus, we have completed the proof of Theorem
\ref{locweightedTrud_thm}.
\end{proof}

Now we introduce the weighted Trudinger-Moser inequality with remainder terms on the entire graded group:
\begin{thm}\label{Tru-Mos_thm}
Let $\mathbb{G}$ be a graded group of homogeneous dimension $Q$ and let $\mathcal{R}$ be a positive Rockland operator of
homogeneous degree $\nu$. Let $|\cdot|$ be an arbitrary homogeneous quasi-norm. Let $1<p<\infty$ and $\beta\in[0,Q)$ with $\mu>Q/(Q-\beta)$. Then there
exist positive constants $C_{2}$ and $C_{3}$ such that
\begin{equation}\label{Tru-Mos1}
\int_{\G}\frac{1}{|x|^{\beta}}\left(\exp(\alpha|f(x)|^{p'})-\sum_{0\leq
k<p-1,\;k\in\mathbb{N}}\frac{\alpha^{k}|f(x)|^{kp'}}{k!}\right)dx\leq C_{3}(\|f\|^{p}_{L^{p}(\G)}+\|f\|^{p/\mu}_{L^{p}(\G)})
\end{equation}
holds for all $\alpha\in(0,C_{2})$, and for all functions $f\in L_{Q/p}^{p}(\G)$ with $\|\R^{\frac{Q}{\nu
p}}f\|_{L^{p}(\G)}\leq1$, where $1/p+1/p'=1$. Moreover, we can take
$C_{2}$ $=C_{2}(p,Q,\beta,\mu)=(e\widetilde{C_{1}}^{p'}\mu p')^{-1}$ and
$$C_{3}=C_{3}(p,Q,\alpha, \beta, \mu)=\max\left(\frac{|\wp|^{1/\mu'}}{(Q-\beta\mu')^{1/\mu'}}\sum_{k\geq
p-1,\;k\in\mathbb{N}}\frac{\alpha^{k}}{k!}(\widetilde{C_{1}}^{p'}kp'\mu)^{k},\widetilde{C_{2}}\right),$$
where $1/\mu+1/\mu'=1$, $\widetilde{C_{1}}=\widetilde{C_{1}}(p,Q)$ and $\widetilde{C_{2}}=\widetilde{C_{2}}(p,Q,\alpha)$ are from
\eqref{crit_GN_ineq} and \eqref{Trud_ineq}, respectively.
\end{thm}
If we take supremum over $\|f\|_{L^{p}_{Q/p}(\G)}\leq1$ in \eqref{Tru-Mos1}, then we obtain the weighted Trudinger-Moser
inequality:
\begin{cor}\label{Tru-Mos_cor} Let $\mathbb{G}$ be a graded group of homogeneous dimension $Q$ and let $\mathcal{R}$ be a
positive Rockland operator of homogeneous degree $\nu$. Let $|\cdot|$ be an arbitrary homogeneous quasi-norm. Let $1<p<\infty$
and $\beta\in[0,Q)$. Then there exist positive constants $C_{2}$ and $C_{3}$ such that
\begin{equation}\label{Tru-Mos1_cor}
\sup_{\|f\|_{L^{p}_{Q/p}(\G)}\leq1}\int_{\G}\frac{1}{|x|^{\beta}}\left(\exp(\alpha|f(x)|^{p'})-\sum_{0\leq
k<p-1,\;k\in\mathbb{N}}\frac{\alpha^{k}|f(x)|^{kp'}}{k!}\right)dx\leq C_{3},
\end{equation}
holds for all $\alpha\in(0,C_{2})$, where $1/p+1/p'=1$. Moreover, the constants
$C_{2}=C_{2}(p,Q,\beta)$ and $C_{3}=C_{3}(p,Q,\alpha,\beta)$ can be given as in Theorem \ref{Tru-Mos_thm}.
\end{cor}
\begin{rem} We note that when $\mathbb{G}$ is the Heisenberg group and $|\cdot|$ is the Kaplan distance, in the special case
$p=Q$, Corollary \ref{Tru-Mos_cor} was obtained in \cite[Theorem 1.1]{Yang14}. In the case $\beta=0$, the unweighted
Trudinger-Moser inequality similar to \eqref{Tru-Mos1} was investigated in \cite{Oz95} on $\G=(\Rn,+)$ with $\R=-\triangle$ the
Laplacian, and in \cite[Theorem 3.5]{RY17} on the graded groups $\mathbb{G}$. On graded groups, we also note that in the special case when $p=2$ and $\beta=0$, the unweighted Trudinger-Moser inequality \eqref{Tru-Mos1_cor} was obtained in \cite{BFG12}.
\end{rem}
\begin{proof}[Proof of Theorem \ref{Tru-Mos_thm}]
By a direct calculation, we have
$$
\int_{\G}\frac{1}{|x|^{\beta}}\sum_{k\geq p-1,\;k\in\mathbb{N}}\frac{1}{k!}(\alpha|f(x)|^{p'})^{k}dx$$$$=
\int_{|x|\leq1}\frac{1}{|x|^{\beta}}\sum_{k\geq p-1,\;k\in\mathbb{N}}\frac{1}{k!}(\alpha|f(x)|^{p'})^{k}dx+
\int_{|x|>1}\frac{1}{|x|^{\beta}}\sum_{k\geq p-1,\;k\in\mathbb{N}}\frac{1}{k!}(\alpha|f(x)|^{p'})^{k}dx$$
$$
\leq\int_{|x|\leq1}\frac{1}{|x|^{\beta}}\sum_{k\geq p-1,\;k\in\mathbb{N}}\frac{1}{k!}(\alpha|f(x)|^{p'})^{k}dx+
\int_{\G}\sum_{k\geq p-1,\;k\in\mathbb{N}}\frac{1}{k!}(\alpha|f(x)|^{p'})^{k}dx$$
\begin{equation}\label{weight_Trud1}
=:\sum_{k\geq p-1,\;k\in\mathbb{N}}I_{1,k}+\sum_{k\geq p-1,\;k\in\mathbb{N}}I_{2,k}=:I_{1}+I_{2}.
\end{equation}
We use H\"{o}lder's inequality for each term of the first summand $I_{1}$ in the right hand side above to have
\begin{equation*}
\begin{split}
I_{1,k}=\frac{\alpha^{k}}{k!}\int_{|x|\leq1}\frac{1}{|x|^{\beta}}&|f(x)|^{kp'}dx
\\&\leq\frac{\alpha^{k}}{k!}\left(\int_{|x|\leq1}\frac{1}{|x|^{\beta \mu'}}dx\right)^{\frac{1}{\mu'}}
\left(\int_{|x|\leq1}|f(x)|^{kp'\mu}dx\right)^{\frac{1}{\mu}}\\&\leq
\frac{|\wp|^{1/\mu'}}{(Q-\beta\mu')^{1/\mu'}}\cdot\frac{\alpha^{k}}{k!}\left(\int_{\G}|f(x)|^{kp'\mu}dx\right)^{\frac{1}{\mu}},
\end{split}
\end{equation*}
for $k\geq p-1$ with $k\in\mathbb{N}$, where $1/\mu+1/\mu'=1$ with $1<\mu'<Q/\beta$. Since $kp'\mu>p$ and $\|\R^{\frac{Q}{\nu
p}}f\|_{L^{p}(\G)}\leq1$, it follows using \eqref{crit_GN_ineq} that
$$I_{1,k}=\frac{\alpha^{k}}{k!}\int_{|x|\leq1}\frac{1}{|x|^{\beta}}|f(x)|^{kp'}dx
\leq
\widetilde{C_{1}}^{kp'}\frac{|\wp|^{1/\mu'}}{(Q-\beta\mu')^{1/\mu'}}\cdot\frac{\alpha^{k}}{k!}(kp'\mu)^{k}\|f\|^{p/\mu}_{L^{p}(\G)},
$$
where $\mu>Q/(Q-\beta)$. By this we have the estimate for $I_{1}$, namely
\begin{equation}\label{weight_Trud2}I_{1}\leq \frac{|\wp|^{1/\mu'}}{(Q-\beta\mu')^{1/\mu'}}\sum_{k\geq
p-1,\;k\in\mathbb{N}}\frac{\alpha^{k}}{k!}(\widetilde{C_{1}}^{p'}kp'\mu)^{k}\|f\|^{p/\mu}_{L^{p}(\G)}.
\end{equation}
Let us check the convergence of the series in the right hand side of \eqref{weight_Trud2}:
\begin{equation*}
\begin{split}
\lim_{k\rightarrow \infty}\frac{\alpha^{k+1}(\widetilde{C_{1}}^{p'}(k+1)p'\mu)^{k+1}}{(k+1)!}\cdot
\frac{k!}{\alpha^{k}(\widetilde{C_{1}}^{p'}kp'\mu)^{k}}&=\alpha\widetilde{C_{1}}^{p'}\mu p'\lim_{k\rightarrow \infty}
\left(1+\frac{1}{k}\right)^{k}\\&=\alpha\widetilde{C_{1}}^{p'}\mu p'e,
\end{split}
\end{equation*}
that is, the series in the right hand side of \eqref{weight_Trud2} converges when $0<\alpha<(e\widetilde{C_{1}}^{p'}\mu p')^{-1}$. For $I_{2}$, in view of \eqref{Trud_ineq} one has
\begin{equation}\label{weight_Trud3}I_{2}=\int_{\G}\sum_{k\geq p-1,\;k\in\mathbb{N}}\frac{1}{k!}(\alpha|f(x)|^{p'})^{k}dx\leq
\widetilde{C_{2}}\|f\|^{p}_{L^{p}(\G)},
\end{equation}
for $\alpha\in(0, (ep'\widetilde{C_{1}}^{p'})^{-1})$, see Remark \ref{rem_Trud_thm}, where
$\widetilde{C_{1}}=\widetilde{C_{1}}(p,Q)$ is the constant from \eqref{crit_GN_ineq}.
Finally, plugging \eqref{weight_Trud2} and \eqref{weight_Trud3} into \eqref{weight_Trud1}, we obtain \eqref{Tru-Mos1} for all
$\alpha\in(0,C_{2})$.

Thus, we have completed the proof of Theorem \ref{Tru-Mos_thm}.
\end{proof}

Now we discuss critical Hardy type inequalities, namely the case $a=Q/p$ of Theorem \ref{Hardy_thm_grad}. Moreover, we show the equivalence of critical Hardy type and local weighted Trudinger-Moser type inequalities.
\begin{thm}\label{Hardy_Rock_thm}
Let $\mathbb{G}$ be a graded Lie group of homogeneous dimension $Q$ and let $\mathcal{R}$ be a positive Rockland operator of
homogeneous degree $\nu$. Let $|\cdot|$ be an arbitrary homogeneous quasi-norm. Let $1<p<\infty$ and $\beta\in[0,Q)$. Let $r>0$ be given and let
$x_{0}$ be any point of $\G$. Then for any $p\leq q<\infty$ there exists a positive constant $C_{4}=C_{4}(p, Q, \beta, r, q)$ such that
\begin{equation}\label{Hardy_Rock1}
\left\|\frac{f}{|x|^{\frac{\beta}{q}}}\right\|_{L^{q}(B(x_{0},r))}\leq C_{4}q^{1-1/p}\|f\|_{L^{p}_{Q/p}(B(x_{0},r))}
\end{equation}
holds for all $f\in L^{p}_{Q/p}(B(x_{0},r))$, where the space $L^{p}_{Q/p}(B(x_{0},r))$ is defined in \eqref{def_space}. Furthermore, we have
\begin{equation}\label{equiv_identity_Hardy}
\frac{1}{\widehat{\alpha} p'e}=A^{p'}=B^{p'},
\end{equation}
where
\begin{equation*}
\begin{split}
\widehat{\alpha}=\sup\{\alpha>0; \exists C_{1}=C_{1}(\alpha):\eqref{weightedTrud1} \textrm{ holds }\forall f\in
&L_{Q/p}^{p}(B(x_{0},r))\\&
\textrm{ with } \|f\|_{L^{p}_{Q/p}(B(x_{0},r))}\leq1\},
\end{split}
\end{equation*}
\begin{equation*}
\begin{split}
A=\inf\{C_{4}>0; \exists t=t(C_{4}) \textrm{ with } t\geq p:\eqref{Hardy_Rock1}\textrm{ holds }\forall f\in
&L_{Q/p}^{p}(B(x_{0},r)),\\&
\forall q
\textrm{ with } t\leq q<\infty\},
\end{split}
\end{equation*}
\begin{equation}\label{alphaAB_Hardy}
B=\limsup_{q\rightarrow \infty}\sup_{f\in L^{p}_{Q/p}(B(x_{0},r))\backslash\{0\}}\frac{\left\|\frac{f}{|x|^{\frac{\beta}{q}}}\right\|_{L^{q}(B(x_{0},r))}}
{q^{1-1/p}\|f\|_{L^{p}_{Q/p}(B(x_{0},r))}}.
\end{equation}
The critical Hardy type inequalities \eqref{Hardy_Rock1} are equivalent to the weighted Trudinger-Moser inequalities \eqref{weightedTrud1}.
\end{thm}
\begin{rem}\label{rem_F} By \eqref{equiv_identity_Hardy} and \eqref{alphaAB_Hardy}, we see that the constant $B$ is asymptotically sharp for \eqref{Hardy_Rock1}, i.e. \eqref{Hardy_Rock1} does not hold for $0<C_{4}<B$.
\end{rem}
\begin{rem}\label{Ricci_compare1} When $\G$ is a stratified group and $\R=-\L$ is the sub-Laplacian, the critical Hardy type inequality
\eqref{Hardy_Rock1} takes the form
\begin{equation}\label{Hardy_Rock1_strat}
\left\|\frac{f}{|x|^{\frac{\beta}{q}}}\right\|_{L^{q}(B(x_{0},r))}\leq C_{4}q^{1-1/p}(\|f\|_{L^{p}(B(x_{0},r))}+\|(-\L)^{\frac{Q}{2
p}}f\|_{L^{p}(B(x_{0},r))})
\end{equation}
for any $\beta\in[0,Q)$, $p\leq q<\infty$ and for all $f\in L^{p}_{Q/p}(B(x_{0},r))$. We note that when $q=p$, the inequality
\eqref{Hardy_Rock1_strat} gives the critical case of \eqref{Ricci1} for $f\in L^{p}_{Q/p}(B(x_{0},r))$.
\end{rem}
\begin{proof}[Proof of Theorem \ref{Hardy_Rock_thm}]
Taking into account $B\leq A$, we see that to show \eqref{equiv_identity_Hardy} it is enough to prove \eqref{weightedTrud1}$\Rightarrow$\eqref{Hardy_Rock1} with $\widehat{\alpha}\leq(ep'A^{p'})^{-1}$ and
\eqref{Hardy_Rock1}$\Rightarrow$\eqref{weightedTrud1} with $1/\widehat{\alpha}\leq p'eB^{p'}$. Let us start with the first. Here, we also can assume that $\|f\|_{L^{p}_{Q/p}(B(x_{0},r))} \neq0$, otherwise we have $f\equiv0$, that is, \eqref{Hardy_Rock1} is trivial. Replacing $f$ by $f/\|f\|_{L^{p}_{Q/p}(B(x_{0},r))}$ in \eqref{weightedTrud1},
we have
$$
\int_{B(x_{0},r)}\frac{1}{|x|^{\beta}}\sum_{k\geq p-1,
\;k\in\mathbb{N}}\frac{1}{k!}\left(\frac{\alpha|f(x)|^{p'}}{\|f\|_{L^{p}_{Q/p}(B(x_{0},r))}^{p'}}\right)^{k}dx\leq
C_{1}.$$
It follows that for any $\varepsilon$ with $0<\varepsilon<\widehat{\alpha}$ there exists $C_{\varepsilon}$ such that
$$
\int_{B(x_{0},r)}\frac{1}{|x|^{\beta}}\sum_{k\geq p-1,
\;k\in\mathbb{N}}\frac{1}{k!}\left(\frac{(\widehat{\alpha}-\varepsilon)|f(x)|^{p'}}{\|f\|_{L^{p}_{Q/p}(B(x_{0},r))}^{p'}}\right)^{k}dx\leq
C_{\varepsilon}.$$
Here, we can take $C_{\varepsilon}=C_{\varepsilon}(p,Q,\beta,r)=C_{1}(p,Q,\widehat{\alpha}-\varepsilon, \beta, r, \mu)$, where $C_{1}$ is the constant from \eqref{weightedTrud1}.
In particular, it implies that
\begin{equation}\label{Hardy_ineq_3}
\left\|\frac{f}{|x|^{\frac{\beta}{p'k}}}\right\|_{L^{p'k}(B(x_{0},r))}\leq (C_{\varepsilon}k!)^{1/p'k}(\widehat{\alpha}-\varepsilon)^{-1/p'}
\|f\|_{L^{p}_{Q/p}(B(x_{0},r))}
\end{equation}
holds for all $k\in\mathbb{N}$ with $k\geq p-1$. Moreover, for any $q>p$, there exists an integer $k\geq p-1$ satisfying $p'k\leq q <p'(k+1)$. Then, using H\"{o}lder's inequality for $\frac{\theta q}{p'k}+\frac{(1-\theta)
q}{p'(k+1)}=1$ with $0<\theta\leq1$ we calculate
\begin{equation*}
\begin{split}
\int_{\G}\frac{|f(x)|^{q}}{|x|^{\beta}}dx&=\int_{\G}\frac{|f(x)|^{\theta q}}{|x|^{\frac{\beta\theta
q}{p'k}}}\cdot\frac{|f(x)|^{(1-\theta)q}}{|x|^{\frac{\beta(1-\theta) q}{p'(k+1)}}}dx\\&
\leq \left(\int_{\G}\frac{|f(x)|^{p'k}}{|x|^{\beta}}dx\right)^{\frac{\theta q}{p'k}}
\left(\int_{\G}\frac{|f(x)|^{p'(k+1)}}{|x|^{\beta}}dx\right)^{\frac{(1-\theta) q}{p'(k+1)}}\\&
=\left\|\frac{f}{|x|^{\frac{\beta}{p'k}}}\right\|_{L^{p'k}(\G)}^{\theta q}
\left\|\frac{f}{|x|^{\frac{\beta}{p'(k+1)}}}\right\|_{L^{p'(k+1)}(\G)}^{(1-\theta)q},
\end{split}
\end{equation*}
that is,
\begin{equation}\label{Hardy_GN_ineq_4}
\left\|\frac{f}{|x|^{\frac{\beta}{q}}}\right\|_{L^{q}(\G)}
\leq\left\|\frac{f}{|x|^{\frac{\beta}{p'k}}}\right\|_{L^{p'k}(\G)}^{\theta}
\left\|\frac{f}{|x|^{\frac{\beta}{p'(k+1)}}}\right\|_{L^{p'(k+1)}(\G)}^{(1-\theta)}.
\end{equation}
Then, by \eqref{Hardy_GN_ineq_4} and \eqref{Hardy_ineq_3}, we have
\begin{equation}\label{Hardy_ineq_4}
\left\|\frac{f}{|x|^{\frac{\beta}{q}}}\right\|_{L^{q}(B(x_{0},r))}\leq C_{\varepsilon}^{\frac{1}{q}}(\widehat{\alpha}-\varepsilon)^{-\frac{1}{p'}}
(\Gamma(q/p'+2))^{\frac{1}{q}}
\|f\|_{L^{p}_{Q/p}(B(x_{0},r))},
\end{equation}
where we have used $(k+1)!\leq \Gamma(q/p'+2)$.
Since the constant $C_{\varepsilon}$ does not depend on $q$, so that the constant in the right hand side of above does not have singularity at $q=p$, then taking the limit $q\rightarrow p$, we see that \eqref{Hardy_ineq_4} also holds true for $q=p$. Thus, we have obtained \eqref{Hardy_ineq_4} for any $q\geq p$ and for all $f\in L^{p}_{Q/p}(B(x_{0},r))$, which is \eqref{Hardy_Rock1}.

To analyse the constant, we rewrite \eqref{Hardy_ineq_4} for $q\rightarrow\infty$ as follows
\begin{equation}\label{Hardy_ineq_4_1}
\lim_{q\rightarrow\infty}\frac{\left\|\frac{f}{|x|^{\frac{\beta}{q}}}\right\|_{L^{q}(B(x_{0},r))}}
{(\Gamma(q/p'+2))^{\frac{1}{q}}
\|f\|_{L^{p}_{Q/p}(B(x_{0},r))}}\leq \lim_{q\rightarrow\infty}C_{\varepsilon}^{\frac{1}{q}}(\widehat{\alpha}-\varepsilon)^{-\frac{1}{p'}}
=(\widehat{\alpha}-\varepsilon)^{-\frac{1}{p'}}.
\end{equation}
On the other hand, we have
\begin{equation}\label{Hardy_ineq_4_2}
\lim_{q\rightarrow\infty}\frac{q^{1-1/p}}{(\Gamma(q/p'+2))^{\frac{1}{q}}}=
\lim_{q\rightarrow\infty}\frac{q^{1-1/p}}{(1+o(1))\left(\frac{q}{ep'}\right)^{1/p'}}=(ep')^{1/p'},
\end{equation}
where we have used for $q\rightarrow+\infty$ that
\begin{equation}\label{Gamma1}
\begin{split}
\Gamma(q/p'+2)^{1/q}&=\left((1+o(1))\sqrt{2\pi\left(q/p'+1\right)}\left(\frac{q/p'+1}{e}\right)^{q/p'+1}\right)^{1/q}
\\&=(1+o(1))\left(\frac{q}{ep'}\right)^{1/p'}.
\end{split}
\end{equation}
Now, dividing \eqref{Hardy_ineq_4_1} by \eqref{Hardy_ineq_4_2}, we note that for any $\delta>0$ there exists $t\geq p$ such that
\begin{equation}\label{crit_Hardy_ineq2_1}
\left\|\frac{f}{|x|^{\frac{\beta}{q}}}\right\|_{L^{q}(B(x_{0},r))}\leq
((p'e(\widehat{\alpha}-\varepsilon))^{-1/p'}+\delta)q^{1-1/p}\|f\|_{L^{p}_{Q/p}(B(x_{0},r))}
\end{equation}
holds for all $f\in L_{Q/p}^{p}(B(x_{0},r))$ and $q$ with $t\leq q<\infty$.

Thus, we obtain $A\leq (p'e(\widehat{\alpha}-\varepsilon))^{-1/p'}+\delta$, which implies $\widehat{\alpha}\leq(ep'A^{p'})^{-1}$
since $\varepsilon$ and $\delta$ are arbitrary.

Now let us show that \eqref{Hardy_Rock1}$\Rightarrow$\eqref{weightedTrud1} with $1/\widehat{\alpha}\leq p'eB^{p'}$. By
\eqref{Hardy_Rock1}, for any $q$ with $p\leq q<\infty$ there exists a positive constant $C_{4}=C_{4}(p,Q,\beta,r,q)$ such that
\begin{equation}\label{crit_Hardy_1}
\left\|\frac{f}{|x|^{\frac{\beta}{q}}}\right\|_{L^{q}(B(x_{0},r))}\leq C_{4}q^{1-1/p}\|f\|_{L^{p}_{Q/p}(B(x_{0},r))}
\end{equation}
holds for all $f\in L^{p}_{Q/p}(B(x_{0},r))$. Since we have $\|f\|_{L^{p}_{Q/p}(B(x_{0},r))}\leq1$, \eqref{crit_Hardy_1} implies that
\begin{equation}\label{crit_Hardy_2}
\begin{split}
\int_{B(x_{0},r)}\frac{1}{|x|^{\beta}}&\left(\exp(\alpha|f(x)|^{p'})-\sum_{0\leq k<p-1,\;k\in\mathbb{N}}
\frac{1}{k!}(\alpha|f(x)|^{p'})^{k}\right)dx\\&
\leq\sum_{p'k\geq p,\;k\in\mathbb{N}}\frac{(\alpha p'kC_{4}^{p'})^{k}}{k!}.
\end{split}
\end{equation}
The condition $0\leq\alpha<1/(p'eC_{4}^{p'})$ provides the convergence of the last series in \eqref{crit_Hardy_2}. Thus, we have obtained \eqref{weightedTrud1} with  $0\leq\alpha<1/(p'eC_{4}^{p'})$, that is,
$\widehat{\alpha}\geq1/(p'eC_{4}^{p'})$ for all $C_{4}\geq B$, which implies $\widehat{\alpha}\geq1/(p'eB^{p'})$.

Thus, we have completed the proof of Theorem \ref{Hardy_Rock_thm}.
\end{proof}
In the case $p=Q$, since we have Theorem \ref{locweightedTrud_thm_Qp}, then similarly to the proof of Theorem \ref{Hardy_Rock_thm}, we can obtain the following improved version of Theorem \ref{Hardy_Rock_thm} on stratified groups:
\begin{thm}\label{Hardy_Rock_thm_pQ}
Let $\mathbb{G}$ be a stratified group of homogeneous dimension $Q\geq3$ and let $|\cdot|$ be a homogeneous norm on $\G$. Let $\beta\in[0,Q)$. Let $r>0$ be given and let $x_{0}$ be any point of $\G$. Let $\alpha_{Q}$ be as in Theorem \ref{Tyson_thm2}. Then for any $Q\leq q<\infty$ there exists a positive constant $C_{5}=C_{5}(Q, \beta, r, q)$ such that
\begin{equation}\label{Hardy_Rock1_pQ}
\left\|\frac{f}{|\cdot|^{\frac{\beta}{q}}}\right\|_{L^{q}(B(x_{0},r))}\leq C_{5}q^{1-1/Q}\|\nabla_{H}f\|_{L^{Q}(B(x_{0},r))}
\end{equation}
holds for all functions $f\in L^{Q}_{1}(B(x_{0},r))$. Moreover, we have
\begin{equation}\label{equiv_identity_Hardy_pQ}
\frac{1}{\alpha_{\beta} Q'e}=\widetilde{A}^{Q'}=\widetilde{B}^{Q'},
\end{equation}
where
$$\alpha_{\beta}=\alpha_{Q}(1-\beta/Q),$$
\begin{equation*}
\begin{split}
\widetilde{A}=\inf\{C_{5}>0; &\exists t=t(C_{5}) \textrm{ with } t\geq Q:\\& \eqref{Hardy_Rock1_pQ}\textrm{ holds }\forall f\in
L^{Q}_{1}(B(x_{0},r)),
\forall q
\textrm{ with } t\leq q<\infty\},
\end{split}
\end{equation*}
\begin{equation}\label{alphaAB_Hardy_pQ}
\widetilde{B}=\limsup_{q\rightarrow \infty}\sup_{f\in L^{Q}_{1}(B(x_{0},r))\backslash\{0\}}
\frac{\left\|\frac{f}{|\cdot|^{\frac{\beta}{q}}}\right\|_{L^{q}(B(x_{0},r))}}{q^{1-1/Q}\|\nabla_{H}f\|_{L^{Q}(B(x_{0},r))}}.
\end{equation}
The weighted Trudinger-Moser inequalities \eqref{weightedTrud1_Qp} are equivalent to the critical Hardy type inequalities \eqref{Hardy_Rock1_pQ} with relation \eqref{equiv_identity_Hardy_pQ}.
\end{thm}
\begin{rem}\label{rem_B7} By \eqref{equiv_identity_Hardy_pQ} and \eqref{alphaAB_Hardy_pQ}, we see that the constant
$$\widetilde{B}=(\alpha_{Q}(1-\beta/Q)Q'e)^{-1/Q'}$$
is asymptotically sharp for \eqref{Hardy_Rock1_pQ}, i.e. \eqref{Hardy_Rock1_pQ} does not hold for $0<C_{5}<\widetilde{B}$.

In fact, \eqref{Hardy_Rock1_pQ} implies \eqref{weightedTrud1_Qp} for $0<\alpha<\alpha_{1}$ for some $\alpha_{1}>0$, while \eqref{Hardy_Rock1_pQ} and \eqref{equiv_identity_Hardy_pQ} together imply \eqref{weightedTrud1_Qp} for all $0<\alpha<\alpha_{\beta}$.
\end{rem}

\section{Weighted Gagliardo-Nirenberg inequalities}
\label{SEC:GN}
In this section we show weighted Gagliardo-Nirenberg inequalities assoicated
with the positive Rockland operators. Moreover, we show the equivalence of weighted Trudinger-Moser inequalities with remainder
terms \eqref{Tru-Mos1} and weighted Gagliardo-Nirenberg type inequalities \eqref{Hardy_GN_Rock1}, and establish an asymptotic relation between their best constants.

\begin{thm}\label{Hardy_GN_Rock_thm}
Let $\mathbb{G}$ be a graded Lie group of homogeneous dimension $Q$ and let $\mathcal{R}$ be a positive Rockland operator of
homogeneous degree $\nu$. Let $|\cdot|$ be an arbitrary homogeneous quasi-norm. Let $1<p<\infty$ and $Q/(Q-\beta)<\mu<\infty$ with $\beta\in[0,Q)$. Then for any $p\leq q<\infty$ there exists a positive constant $C_{7}=C_{7}(p,Q,\beta,\mu,q)$ such that
\begin{equation}\label{Hardy_GN_Rock1}
\left\|\frac{f}{|x|^{\frac{\beta}{q}}}\right\|_{L^{q}(\G)}\leq C_{7}q^{1-1/p}
(\|\R^{\frac{Q}{\nu p}}f\|_{L^{p}(\G)}^{1-p/q}\|f\|_{L^{p}(\G)}^{p/q}+\|\R^{\frac{Q}{\nu
p}}f\|_{L^{p}(\G)}^{1-p/(q\mu)}\|f\|_{L^{p}(\G)}^{p/(q\mu)})
\end{equation}
holds for all $f\in L^{p}_{Q/p}(\G)$.

Furthermore, we have
\begin{equation}\label{equiv_identity}
\frac{1}{\widetilde{\alpha} p'e}=D^{p'}=F^{p'},
\end{equation}
where
\begin{equation*}
\begin{split}
\widetilde{\alpha}=\sup\{\alpha>0; \exists C_{3}=C_{3}(\alpha):\eqref{Tru-Mos1} \textrm{ holds }&\forall f\in
L_{Q/p}^{p}(\G)
\\&\textrm{ with } \|\R^{\frac{Q}{\nu p}}f\|_{L^{p}(\G)}\leq1\},
\end{split}
\end{equation*}
\begin{equation*}
\begin{split}
D=\inf\{C_{7}>0; \exists t=t(C_{7}) \textrm{ with } t\geq p:\eqref{Hardy_GN_Rock1}\textrm{ holds }&
\forall f\in L_{Q/p}^{p}(\G), \\& \forall q \textrm{ with } t\leq q<\infty\},
\end{split}
\end{equation*}
\begin{equation}\label{alphaAB}
F=\limsup_{q\rightarrow \infty}\sup_{f\in L^{p}_{Q/p}(\G)\backslash\{0\}}\frac{\left\|\frac{f}{|x|^{\frac{\beta}{q}}}\right\|_{L^{q}(\G)}}{q^{1-1/p}
(\|\R^{\frac{Q}{\nu p}}f\|_{L^{p}(\G)}^{1-p/q}\|f\|_{L^{p}(\G)}^{p/q}+\|\R^{\frac{Q}{\nu
p}}f\|_{L^{p}(\G)}^{1-p/(q\mu)}\|f\|_{L^{p}(\G)}^{p/(q\mu)})}.
\end{equation}
Moreover, the weighted Gagliardo-Nirenberg inequalities \eqref{Hardy_GN_Rock1} are
equivalent to the weighted Trudinger-Moser inequalities with remainder terms \eqref{Tru-Mos1}.
\end{thm}
\begin{rem}\label{rem_B} By \eqref{equiv_identity} and \eqref{alphaAB}, we see that the constant $F$ is asymptotically sharp for \eqref{Hardy_GN_Rock1}, i.e. \eqref{Hardy_GN_Rock1} does not hold for $0<C_{7}<F$.
\end{rem}
In the case $\beta=0$, the Theorem \ref{Hardy_GN_Rock_thm} gives the following corollary:
\begin{cor}\label{cor_crit_incl1} Let $\mathbb{G}$ be a graded Lie group of homogeneous dimension $Q$ and let $1<p<\infty$.
Then, $L^{p}_{Q/p}(\G)$ is continuously embedded in $L^{q}(\G)$ for any $p\leq q<\infty$.
\end{cor}
\begin{rem} We note that the Corollary \ref{cor_crit_incl1} was obtained e.g. in \cite{Oz95} on $\G=(\Rn,+)$ and in \cite[Lemma 4.1]{Yang14} on the Heisenberg group with $Q/p=1$.
\end{rem}
\begin{rem}\label{rem_crit_embed} By \cite[Proposition 4.4.13]{FR16}, we have $L^{p}_{a}(\G)\hookrightarrow L^{q}(\G)$ with $1/q=1/p-a/Q$ and
$0<a<Q/p$. Corollary \ref{cor_crit_incl1} shows the critical case $a=Q/p$ of this continuous embedding.
\end{rem}
\begin{proof}[Proof of Theorem \ref{Hardy_GN_Rock_thm}] Since $F\leq D$, in order to obtain \eqref{equiv_identity} it is enough to show
\eqref{Tru-Mos1}$\Rightarrow$\eqref{Hardy_GN_Rock1} with $\widetilde{\alpha}\leq(ep'D^{p'})^{-1}$ and
\eqref{Hardy_GN_Rock1}$\Rightarrow$\eqref{Tru-Mos1} with $1/\widetilde{\alpha}\leq p'eF^{p'}$.
Then, let us show first \eqref{Tru-Mos1}$\Rightarrow$\eqref{Hardy_GN_Rock1} with $\widetilde{\alpha}\leq(ep'D^{p'})^{-1}$. In the case $\|\R^{\frac{Q}{\nu p}}f\|_{L^{p}(\G)}=0$ we have $f\equiv0$
by Theorem \ref{crit_GN_thm} and the fact that $f$ belongs to the inhomogeneous Sobolev space $L^{p}_{Q/p}(\G)$, that is, \eqref{Hardy_GN_Rock1} is trivial. Therefore, we can assume that $\|\R^{\frac{Q}{\nu
p}}f\|_{L^{p}(\G)}\neq0$. Then, we can replace $f$ by $f/\|\R^{\frac{Q}{\nu p}}f\|_{L^{p}(\G)}$ in \eqref{Tru-Mos1} to get
\begin{multline}\label{Hardy_GN_ineq_2_0}
\int_{\G}\frac{1}{|x|^{\beta}}\left(\exp\left(\frac{\alpha|f(x)|^{p'}}{\|\R^{\frac{Q}{\nu
p}}f\|_{L^{p}(\G)}^{p'}}\right)-\sum_{0\leq k<p-1,
\;k\in\mathbb{N}}\frac{1}{k!}\left(\frac{\alpha|f(x)|^{p'}}{\|\R^{\frac{Q}{\nu p}}f\|_{L^{p}(\G)}^{p'}}\right)^{k}\right)dx\\
\leq
C_{3}\left(\frac{\|f\|^{p}_{L^{p}(\G)}}{\|\R^{\frac{Q}{\nu p}}f\|^{p}_{L^{p}(\G)}}+
\frac{\|f\|^{p/\mu}_{L^{p}(\G)}}{\|\R^{\frac{Q}{\nu p}}f\|^{p/\mu}_{L^{p}(\G)}}\right).
\end{multline}
From this, we note that for any $0<\varepsilon<\widetilde{\alpha}$ there is $C_{\varepsilon}$ such that
\begin{multline}\label{Hardy_GN_ineq_2_0}
\int_{\G}\frac{1}{|x|^{\beta}}\left(\exp\left(\frac{(\widetilde{\alpha}-\varepsilon)|f(x)|^{p'}}{\|\R^{\frac{Q}{\nu
p}}f\|_{L^{p}(\G)}^{p'}}\right)-\sum_{0\leq k<p-1,
\;k\in\mathbb{N}}\frac{1}{k!}\left(\frac{(\widetilde{\alpha}-\varepsilon)|f(x)|^{p'}}{\|\R^{\frac{Q}{\nu p}}f\|_{L^{p}(\G)}^{p'}}\right)^{k}\right)dx\\
\leq
C_{\varepsilon}\left(\frac{\|f\|^{p}_{L^{p}(\G)}}{\|\R^{\frac{Q}{\nu p}}f\|^{p}_{L^{p}(\G)}}+
\frac{\|f\|^{p/\mu}_{L^{p}(\G)}}{\|\R^{\frac{Q}{\nu p}}f\|^{p/\mu}_{L^{p}(\G)}}\right).
\end{multline}
Here, we can take $C_{\varepsilon}=C_{\varepsilon}(p,Q,\beta,\mu)=C_{3}(p,Q,\widetilde{\alpha}-\varepsilon, \beta, \mu)$, where $C_{3}$ is the constant from \eqref{Tru-Mos1}.
It follows that
\begin{multline*}
\int_{\G}\frac{1}{|x|^{\beta}}\sum_{k\geq p-1, \;k\in\mathbb{N}}\frac{1}{k!}\left(\frac{(\widetilde{\alpha}-\varepsilon)|f(x)|^{p'}}{\|\R^{\frac{Q}{\nu
p}}f\|_{L^{p}(\G)}^{p'}}\right)^{k}dx\\ \leq C_{\varepsilon}\left(\frac{\|f\|^{p}_{L^{p}(\G)}}{\|\R^{\frac{Q}{\nu p}}f\|^{p}_{L^{p}(\G)}}+
\frac{\|f\|^{p/\mu}_{L^{p}(\G)}}{\|\R^{\frac{Q}{\nu p}}f\|^{p/\mu}_{L^{p}(\G)}}\right),
\end{multline*}
that is, in particular,
\begin{equation}\label{Hardy_GN_ineq_3}
\begin{split}
\left\|\frac{f}{|x|^{\frac{\beta}{p'k}}}\right\|_{L^{p'k}(\G)}\leq (C_{\varepsilon}k!)^{1/p'k}(\widetilde{\alpha}-\varepsilon)^{-1/p'}
&\|\R^{\frac{Q}{\nu p}}f\|_{L^{p}(\G)}\times\\&\times\left(\frac{\|f\|^{p}_{L^{p}(\G)}}{\|\R^{\frac{Q}{\nu p}}f\|^{p}_{L^{p}(\G)}}+
\frac{\|f\|^{p/\mu}_{L^{p}(\G)}}{\|\R^{\frac{Q}{\nu p}}f\|^{p/\mu}_{L^{p}(\G)}}\right)^{1/p'k}
\end{split}
\end{equation}
for all $k\in\mathbb{N}$ with $k\geq p-1$. Moreover, for any $q>p$, there exists an integer $k\geq p-1$ satisfying $p'k\leq q <p'(k+1)$. Then, by \eqref{Hardy_GN_ineq_4} and \eqref{Hardy_GN_ineq_3}, we obtain
\begin{equation}\label{Hardy_GN_ineq_4_1}
\begin{split}
\left\|\frac{f}{|x|^{\frac{\beta}{q}}}\right\|_{L^{q}(\G)}\leq
(C_{\varepsilon}\Gamma(q/p'+2))^{1/q}&(\widetilde{\alpha}-\varepsilon)^{-1/p'}
\|\R^{\frac{Q}{\nu p}}f\|_{L^{p}(\G)}\times
\\&
\times\left(\frac{\|f\|^{p}_{L^{p}(\G)}}{\|\R^{\frac{Q}{\nu p}}f\|^{p}_{L^{p}(\G)}}+
\frac{\|f\|^{p/\mu}_{L^{p}(\G)}}{\|\R^{\frac{Q}{\nu p}}f\|^{p/\mu}_{L^{p}(\G)}}\right)^{1/q},
\end{split}
\end{equation}
where we have used $(k+1)!\leq \Gamma(q/p'+2)$ for $q\geq p'k$. Since the constant $C_{\varepsilon}$ does not depend on $q$, hence the constant in the right hand side of above does not have singularity at $q=p$, then taking the limit $q\rightarrow p$, we see that \eqref{Hardy_GN_ineq_4_1} is also holds true for $q=p$. Thus, we have obtained \eqref{Hardy_GN_ineq_4_1} for any $q\geq p$ and for all $f\in L^{p}_{Q/p}(\G)$, which is \eqref{Hardy_GN_Rock1}.

The inequality \eqref{Hardy_GN_ineq_4_1} with \eqref{Gamma1} implies that for any $\delta>0$ there exists $t\geq p$ such that
\begin{equation}\label{crit_Hardy_GN_ineq2_1}
\begin{split}
\left\|\frac{f}{|x|^{\frac{\beta}{q}}}\right\|_{L^{q}(\G)}\leq
((p'e(\widetilde{\alpha}-\varepsilon))^{-1/p'}+\delta)q^{1-1/p}&\|\R^{\frac{Q}{\nu p}}f\|_{L^{p}(\G)}\times
\\&\times\left(\frac{\|f\|^{p}_{L^{p}(\G)}}{\|\R^{\frac{Q}{\nu p}}f\|^{p}_{L^{p}(\G)}}+
\frac{\|f\|^{p/\mu}_{L^{p}(\G)}}{\|\R^{\frac{Q}{\nu p}}f\|^{p/\mu}_{L^{p}(\G)}}\right)^{1/q}
\end{split}
\end{equation}
holds for all $f\in L_{Q/p}^{p}(\G)$ and all $q$ with $t\leq q<\infty$.

Thus, we see that $D\leq (p'e(\widetilde{\alpha}-\varepsilon))^{-1/p'}+\delta$, then by arbitrariness of $\varepsilon$ and
$\delta$ we obtain $\widetilde{\alpha}\leq(ep'D^{p'})^{-1}$.

Now we show that \eqref{Hardy_GN_Rock1}$\Rightarrow$\eqref{Tru-Mos1} with $1/\widetilde{\alpha}\leq p'eF^{p'}$. By
\eqref{Hardy_GN_Rock1}, for any $q$ with $p\leq q<\infty$ there exists a positive constant $C_{7}=C_{7}(p,Q,\beta,\mu,q)$ such that
\begin{equation}\label{crit_Hardy_GN_1}
\left\|\frac{f}{|x|^{\frac{\beta}{q}}}\right\|_{L^{q}(\G)}\leq C_{7}q^{1-1/p}\|\R^{\frac{Q}{\nu p}}f\|_{L^{p}(\G)}
\left(\frac{\|f\|^{p}_{L^{p}(\G)}}{\|\R^{\frac{Q}{\nu p}}f\|^{p}_{L^{p}(\G)}}+
\frac{\|f\|^{p/\mu}_{L^{p}(\G)}}{\|\R^{\frac{Q}{\nu p}}f\|^{p/\mu}_{L^{p}(\G)}}\right)^{1/q}
\end{equation}
holds for all $f\in L^{p}_{Q/p}(\G)$.
Using this and $\|\R^{\frac{Q}{\nu p}}f\|_{L^{p}(\G)}\leq1$, we write
\begin{equation}\label{crit_Hardy_GN_2}
\begin{split}
\int_{\G}\frac{1}{|x|^{\beta}}&\left(\exp(\alpha|f(x)|^{p'})-\sum_{0\leq k<p-1,\;k\in\mathbb{N}}
\frac{1}{k!}(\alpha|f(x)|^{p'})^{k}\right)dx\\&
\leq \sum_{p'k\geq p,\;k\in\mathbb{N}}\frac{(\alpha p'kC_{7}^{p'})^{k}}{k!}
(\|f\|^{p/q}_{L^{p}(\G)}+\|f\|^{p/(q\mu)}_{L^{p}(\G)}).
\end{split}
\end{equation}
The last series in \eqref{crit_Hardy_GN_2} converges when $0\leq\alpha<1/(p'eC_{7}^{p'})$. Thus, we have obtained \eqref{Tru-Mos1} with $0\leq\alpha<1/(p'eC_{7}^{p'})$. Hence
$\widetilde{\alpha}\geq 1/(p'eC_{7}^{p'})$ for all $C_{7}\geq F$, which gives $\widetilde{\alpha}\geq 1/(p'eF)^{p'}$.

Thus, we have completed the proof of Theorem \ref{Hardy_GN_Rock_thm}.
\end{proof}

\end{document}